\documentclass{article}
\usepackage[a4paper, scale = 0.7, centering]{geometry}
\usepackage{amssymb,amsmath,amsthm}
\usepackage{bm}
\usepackage{booktabs}
\usepackage{graphicx,subfigure,epstopdf}
\usepackage[usenames]{color}
\usepackage{url}
\usepackage{algorithm,algorithmic}
\usepackage{dsfont,verbatim}
\usepackage{varwidth}
\usepackage{caption}

\graphicspath{{./pics/}}

\newtheorem{theorem}{Theorem}[section]
\newtheorem{lemma}{Lemma}[section]

\newtheorem{remark}{Remark}[section]
\newtheorem{corollary}{Corollary}[section]

\newtheorem{proposition}{Proposition}[section]
\numberwithin{equation}{section}

\title{Variational Gaussian Approximation for Poisson Data} 
\author{Simon R. Arridge\thanks{Department of Computer Science, University College London, Gower Street, London WC1E 6BT, UK (\texttt{s.r.arridge,b.jin,chen.zhang.16@ucl.ac.uk})}\and Kazufumi Ito\thanks{Department of Mathematics, North Carolina State University, Raleigh, NC 27607, USA (\texttt{kito@ncsu.edu})} \and Bangti Jin\footnotemark[1] \and Chen Zhang\footnotemark[1]}
\date{}
\begin{document}
\maketitle
\begin{abstract}
The Poisson model is frequently employed to describe count data, but in a Bayesian context it leads to
an analytically intractable posterior probability distribution. In this work, we analyze a variational Gaussian
approximation to the posterior distribution arising from the Poisson model with a Gaussian prior. This is
achieved by seeking an optimal Gaussian distribution minimizing the Kullback-Leibler divergence from
 the posterior distribution to the approximation, or
equivalently maximizing the lower bound for the model evidence. We derive an explicit expression for
the lower bound, and show the existence and uniqueness of the optimal Gaussian approximation. The lower
bound functional can be viewed as a variant of classical Tikhonov regularization that penalizes also the
covariance. Then we develop an efficient alternating direction maximization algorithm for solving
the optimization problem, and analyze its convergence. We discuss strategies for reducing the computational
complexity via low rank structure of the forward operator and the sparsity of the covariance. Further, as an
application of the lower bound, we discuss hierarchical Bayesian modeling for selecting the
hyperparameter in the prior distribution, and propose a monotonically convergent algorithm for determining
the hyperparameter. We present extensive numerical experiments to illustrate the Gaussian approximation and the algorithms.
\end{abstract}

\section{Introduction}

This work is concerned with Gaussian approximations to a Poisson noise model for linear
inverse problems. The Poisson model is popular for modeling count data, where the response variable
follows a Poisson distribution with a parameter that is the exponential of a linear combination of the
unknown parameters. The model is especially suitable
for low count data, where the standard Gaussian model is inadequate. It has found many successful
practical applications, including transmission tomography \cite{YavuzFessler:1997,ErdoganFessler:1999}.

One traditional approach to parameter estimation with the Poisson model is the maximum likelihood method
or penalized variants with a convex penalty. This leads to a convex optimization problem,
whose solution is then taken as an approximation to the true solution. This approach has been extensively
studied, and we refer interested readers to the survey \cite{HohageWerner:2016} for a comprehensive
account on important developments along this line. However, this approach gives only a point estimator,
and does not allow quantifying the associated uncertainties directly. In this work, we
aim at a full Bayesian treatment of the problem, where both the point estimator (mean) and the associated
uncertainties (covariance) are of interest \cite{KaipioSomersalo:2005,Stuart:2010}. We shall focus
on the case of a Gaussian prior, which forms the basis of many other important priors, e.g., sparsity
prior via scale mixture representation. Then following the Bayesian procedure, we arrive at a posterior
probability distribution, which however is analytically intractable due to the nonstandard form of the
likelihood function for the Poisson model. We will explain this more precisely in Section \ref{sec:Poisson}.
To explore the posterior state space, instead of applying popular general-purposed sampling techniques, e.g.,
Markov chain Monte Carlo (MCMC), we employ a variational Gaussian approximation (VGA). The VGA
is one extremely popular approximate inference technique in
machine learning \cite{WainwrightJordan:2008,challis2013gaussian}. Specifically, we seek an optimal {Gaussian}
approximation to the non-Gaussian posterior distribution with respect to the Kullback-Leibler divergence.
The approach leads to a large-scale optimization problem over the mean $\mathbf{\bar x}$ and covariance $\mathbf{C}$
(of the Gaussian approximation). In practice, it generally delivers an accurate approximation in an
efficient manner, and thus has received immense attention in recent years in many different areas \cite{hinton1993keeping,barber1998ensemble,
challis2013gaussian,archambeau2007gaussian}. By its very construction, it also gives a lower bound to the
model evidence, which  facilitates its use in model selection. However, a systematic theoretical
understanding of the approach remains largely missing.

In this work, we shall study the analytical properties and develop an efficient algorithm for the VGA in the
context of Poisson data (with the log linear link function). We shall provide a detailed analysis of the
resulting optimization problem. The study sheds interesting new insights into the approach from the perspective
of regularization. Our main contributions are as follows. First, we derive explicit expressions for the
objective functional and its gradient, and establish its strict concavity and the well-posedness of the
optimization problem. Second, we develop an efficient numerical algorithm for finding the optimal Gaussian
approximation, and discuss its convergence properties. The algorithm is of alternating maximization (coordinate
ascent) nature, and it updates the mean $\mathbf{\bar x}$ and covariance $\mathbf{C}$ alternatingly by a globally
convergent Newton method and a fixed point iteration, respectively. We also discuss strategies for its efficient
implementation, by leveraging realistic structure of inverse
problems, e.g., low-rank nature of the forward map $\mathbf{A}$ and sparsity of the covariance $\mathbf{C}$,
to reduce the computational complexity. Third, we illustrate the use of the evidence lower bound for
hyperparameter selection within a hierarchical Bayesian framework, leading to a purely data-driven approach for
determining the regularization parameter, whose proper choice is notoriously challenging. We shall develop a monotonically convergent
algorithm for determining the hyperparameter in the Gaussian prior. Last, we illustrate the approach and the algorithms
with extensive numerical experiments for one- and two-dimensional examples.

Last, we discuss existing works on Poisson models. The majority of existing works aim at recovering
point estimators, either iteratively or by a variational framework \cite{HohageWerner:2016}. Recently, Bardsley and Luttman \cite{BardsleyLuttman:2016} described
a Metroplis-Hastings algorithm for exploring the posterior distribution (with rectified linear inverse link function),
where the proposal samples are drawn from the Laplace approximation (cf. Remark \ref{rmk:Laplace}). The
Poisson model \eqref{eqn:poisson} belongs to generalized linear models (GLMs), to which the VGA has been applied in
statistics and machine learning \cite{OrmerodWand:2012,KhanMohamedMurphy:2012,challis2013gaussian,RohdeWand:2016}. Ormerod and Wand
\cite{OrmerodWand:2012} suggested a variational approximation strategy for fitting GLMs suitable for grouped data.
Challis and Barber \cite{challis2013gaussian} systematically studied VGA for GLMs and various extensions.
The focus of these interesting works \cite{OrmerodWand:2012,KhanMohamedMurphy:2012,challis2013gaussian,
RohdeWand:2016} is on the development of the general VGA methodology and its applications to concrete problems, and do not study
analytical properties and computational techniques for the lower bound functional, which is the main goal
of this work.

The rest of the paper is organized as follows. In Section \ref{sec:Poisson}, we describe
the Poisson model, and formulate the posterior probability distribution. Then in Section \ref{sec:vb},
we develop the variational Gaussian approximation, and analyze its basic analytical properties.
In Section \ref{sec:algorithm}, we propose an efficient numerical algorithm for finding the
optimal Gaussian approximation, and in Section \ref{sec:hyper}, we apply the lower bound to
hyperparameter selection within a hierarchical Bayesian framework.
In Section \ref{sec:numer} we present numerical results for several examples. In two
appendices, we provide further discussions on  the convergence
of the fixed point iteration \eqref{eqn:iter-C} and the differentiability of the regularized solution.

\section{Notation and problem setting}\label{sec:Poisson}
First we recall some standard notation in linear algebra. Throughout, (real-valued) vectors and matrices are
denoted by bold lower- and upper-case letters, respectively, and the vectors are
always column vectors. We will use the notation $(\cdot,\cdot)$ to denote the usual Euclidean inner
product. We shall slightly abuse the notation $(\cdot,\cdot)$ also for the inner product for
matrices. That is, for two matrices $\mathbf{X}, \mathbf{Y}\in\mathbb{R}^{n\times m}$, we define
\begin{equation*}
  (\mathbf{X},\mathbf{Y})=\mathrm{tr}(\mathbf{XY}^t)=\mathrm{tr}(\mathbf{X}^t\mathbf{Y}),
\end{equation*}
where $\mathrm{tr}(\cdot)$ denotes taking the trace of a square matrix, and the superscript $t$ denotes the transpose
of a vector or matrix. This inner product induces the usual Frobenius norm for matrices. We shall use extensively
the cyclic property of the trace operator $\mathrm{tr}(\cdot)$: for three matrices $\mathbf{X,Y,Z}$ of appropriate size, there holds
\begin{equation*}
  \mathrm{tr}(\mathbf{XYZ})=\mathrm{tr}(\mathbf{YZX})=\mathrm{tr}(\mathbf{ZXY}).
\end{equation*}
We shall also use the notation \texttt{diag}$(\cdot)$ for a vector and a square matrix, which gives a diagonal matrix
and a column vector from the diagonals of the matrix, respectively, in the same manner as the \texttt{diag} function
in \texttt{MATLAB}. The notation $\mathbb{N}=\{0,1,\ldots\}$ denotes the set of natural numbers.
Further, the notation $\circ$ denotes the Hadamard product of two matrices or vectors. Last, we
denote by $\mathcal{S}_m^+\subset\mathbb{R}^{m\times m}$ the set of symmetric positive definite matrices in $\mathbb{R}^{m
\times m}$, $\mathbf{I}_m$ the identity matrix in $\mathbb{R}^{m\times m}$, and by $|\cdot|$ and $\|\cdot\|$ the determinant
and the spectral norm, respectively, of a square matrix. Throughout, we view exponential,
logarithm and factorial of a vector as componentwise operation.

Next we recall the finite-dimensional Poisson data model. Let $\mathbf{x}\in\mathbb{R}^m$ be the unknown
signal, $\mathbf{a}_i\in\mathbb{R}^m$, $i = 1,\dots,n$, and $\mathbf{y}\in\mathbb{N}^n\subset \mathbb{R}^n$
be the data vector. We stack the column vectors $\mathbf{a}_i$ into a matrix $\mathbf{A}$
by $\mathbf{A} = [\mathbf{a}^t_i]\in\mathbb{R}^{n\times m}$. Given the matrix $\mathbf{A}$ and data $\mathbf{y}\in\mathbb{N}^n$,
the Poisson model takes the form:
\begin{equation*}
	y_i\sim\text{Pois}(e^{(\mathbf{a}_i,\mathbf{x})}),\quad i=1,2,\ldots,n.
\end{equation*}
Thus, the likelihood function $p(y_i|\mathbf{x})$ for the data point $y_i$ is given by
\begin{equation}\label{eqn:pois}
	p(y_i|\mathbf{x}) = \frac{\lambda_i^{y_i}e^{-\lambda_i}}{y_i!}, \quad \lambda_i = e^{(\mathbf{a}_i,\mathbf{x})}, \,\, i = 1,\ldots,n.
\end{equation}
It is worth noting that the exponential function enters into the Poisson parameter $\lambda$. This is commonly known as the log link
function or log-linear model in the statistical literature \cite{CameronTrivedi:1998}. There are several other models
for the (inverse) link functions, e.g.,
rectified-linear and softplus \cite{pillow2007likelihood}, each having its own pros and cons for
modeling count data. In this work, we shall focus on the log link function. Also this model can be viewed as a simplified
statistical model for transmission tomography \cite{YavuzFessler:1997,ErdoganFessler:1999}.

The likelihood function $p(y_i|\mathbf{x})$ can be equivalently written as
\begin{equation*}
	p(y_i|\mathbf{x}) = e^{y_i(\mathbf{a}_i,\mathbf{x})-e^{(\mathbf{a}_i,\mathbf{x})}-\text{ln}(y_i!)}.
\end{equation*}
Under the independent identically distributed (i.i.d.)
assumption on the data points $y_i$, the likelihood function $p(\mathbf{y}|\mathbf{x})$ of
the data vector $\mathbf{y}$ is given by
\begin{equation}\label{eqn:poisson}
	p(\mathbf{y}|\mathbf{x}) = \prod^n_{i=1}p(y_i|\mathbf{x})
	= e^{(\mathbf{A}\mathbf{x},\mathbf{y})-(e^{\mathbf{A}\mathbf{x}},\mathbf{1}_n)-(\text{ln}(\mathbf{y}!),\mathbf{1}_n)},
\end{equation}
where $\mathbf{1}_n\in\mathbb{R}^n$ is the vector with all entries equal to unity, i.e., $\mathbf{1}_n = [1,\dots,1]^t\in\mathbb{R}^n$.
	
Further, we assume that the unknown $\mathbf{x}$ follows a Gaussian prior $p(\mathbf{x})$, i.e.,
\begin{equation*}
  p(\mathbf{x})=\mathcal{N}(\mathbf{x};\bm{\mu}_0,\mathbf{C}_0):=(2\pi)^{-\frac{m}{2}}|\mathbf{C}_0|^{-\frac{1}{2}}e^{-\frac{1}{2}(\mathbf{x}-\bm \mu_0)^t\mathbf{C}_0^{-1}(\mathbf{x}-\bm \mu_0)},
\end{equation*}
where $\bm\mu_0\in\mathbb{R}^m$ and $\mathbf{C}_0\in\mathcal{S}_m^+$ denote the mean and covariance
of the Gaussian prior, respectively, and $\mathcal{N}$ denotes the normal distribution. In the
framework of variational regularization, the
corresponding penalty $\frac{1}{2}(\mathbf{x}-\bm \mu_0)^t\mathbf{C}_0^{-1}(\mathbf{x}-\bm\mu_0)$ often
imposes certain smoothness constraint. The
Gaussian prior $p(\mathbf{x})$ may depend on additional hyperparameters, cf. Section \ref{sec:hyper} for details. Then
by Bayes' formula, the posterior probability distribution $p(\mathbf{x} |\mathbf{y})$ is given by
\begin{equation}\label{eq:post}
	p(\mathbf{x}|\mathbf{y})= Z^{-1}p(\mathbf{x},\mathbf{y}),
\end{equation}
where the joint distribution $p(\mathbf{x},\mathbf{y})$ is defined by
\begin{equation*}
  p(\mathbf{x},\mathbf{y}) = (2\pi)^{-\frac{m}{2}}|\mathbf{C}_0|^{-\frac{1}{2}} e^{(\mathbf{A}\mathbf{x},\mathbf{y})-(e^{\mathbf{A}\mathbf{x}},\mathbf{1}_n)-(\text{ln}(\mathbf{y}!),
\mathbf{1}_n)-\frac{1}{2}(\mathbf{x}-\bm{\mu}_0)^t\mathbf{C}^{-1}_0(\mathbf{x}-\bm{\mu}_0)},
\end{equation*}
and the normalizing constant $Z(\mathbf{y})$, which depends only on the given
data $\mathbf{y}$, is given by
\begin{equation*}
  Z(\mathbf{y})=p(\mathbf{y})=\int p(\mathbf{x},\mathbf{y}){\rm d}\mathbf{x}.
\end{equation*}
That is, the normalizing constant $Z$ is an integral living in a very high-dimensional space if the parameter
dimension $m$ is large. Thus it is computationally intractable, and so is the posterior distribution $p(\mathbf{x}
|\mathbf{y})$, since it also involves the constant $Z$. The quantity $Z$ is commonly known as model evidence in
the literature, and it underlies many model selection rules, e.g., Bayes factor \cite{kass1995bayes}.
Thus the reliable approximation of $Z(\mathbf{y})$ is important in certain tasks.

The posterior distribution $p(\mathbf{x}|\mathbf{y})$ given in \eqref{eq:post} is the Bayesian solution to the
Poisson model \eqref{eqn:pois} (under a Gaussian prior), and it contains all the information about the
inverse problem. In order to explore the posterior state
space, one typically employs Markov chain Monte Carlo methods, which, however, can be prohibitively expensive for
high-dimensional problems, apart from the well-known challenge in diagnosing the convergence of the Markov chain.
To overcome the challenge, over the last two decades, a large number of approximate inference methods
have been developed, including mean-field approximation \cite{WainwrightJordan:2008}, expectation propagation
\cite{minka2001expectation} and variational Gaussian approximation (VGA) \cite{OpperArchambeau:2009,challis2013gaussian}.
In all these approximations, we aim at finding a best approximate yet tractable distribution $q(\mathbf{x})$ within
a family of parametric/nonparametric probability distributions, by minimizing the error in a certain
probability metric, prominently the Kullback-Leibler divergence $\mathrm{D}_{\text{KL}}(q||p)$, cf. Section
\ref{ssec:KL} below.

In this work, we shall employ the VGA to obtain an optimal Gaussian
approximation $q(\mathbf{x})$ to the posterior
distribution $p(\mathbf{x}|\mathbf{y})$ in the Kullback-Leibler divergence ${\rm D}_{\rm KL}(q||p)$.
Fitting a Gaussian to an intractable distribution is a well-established norm for approximate Bayesian inference,
and it has demonstrated success in many practical applications
\cite{hinton1993keeping,barber1998ensemble,challis2013gaussian,archambeau2007gaussian}.
The popularity can be largely attributed to the fact that the Gaussian approximation is
 computationally attractive, and yet delivers reasonable accuracy for a wide range of problems, due to the
good analytical properties and great flexibility of the Gaussian family.
However, analytical properties of approximate inference procedures are rarely
studied. In the context of Poisson mixed models,
the asymptotic normality of the estimator and its convergence rate was analyzed \cite{hall2011theory}.
In a general setting, some theoretical issues were studied in \cite{PinskiSimpson:2015,LuStuartWeber:2016}.

\section{Gaussian variational approximation}\label{sec:vb}

In this section, we recall the Kullback-Leibler divergence, derive explicit
expressions for the lower bound functional and its gradient, and discuss  basic
analytic properties, e.g., concavity and existence.

\subsection{Kullback-Leibler divergence}\label{ssec:KL}
The Kullback-Leibler divergence is one of the most popular metrics for measuring the distance between
two probability distributions. The Kullback-Leibler (KL) divergence \cite{KullbackLeibler:1951} from
one probability distribution $p$ to another distribution $q$ is a functional defined by
\begin{equation}\label{eq:kl}
	\text{D}_{\text{KL}}(q||p) = \int q(\mathbf{x})\text{ln}\frac{q(\textbf{x})}{p(\mathbf{x})}\text{d}\mathbf{x}.
\end{equation}
	
Clearly, KL divergence is not symmetric and thus not a metric in the mathematical sense.
Since the logarithm function $\ln x$ is concave and that $q$ is
normalized, i.e., $\int q(\mathbf{x}){\rm d}\mathbf{x}=1$, by Jensen's inequality, we can
derive the nonnegativity of the KL divergence:
\begin{equation}\label{eqn:KLD-nonnegative}
\begin{aligned}
	\text{D}_{\text{KL}}(q||p) &= \int q(\mathbf{x})\text{ln}\frac{q(\textbf{x})}{p(\mathbf{x})}\text{d}\mathbf{x} = - \int q(\mathbf{x})\text{ln}\frac{p(\mathbf{x})}{q(\textbf{x})}\text{d}\mathbf{x}\\
	&\ge -\text{ln}\int q(\mathbf{x})\frac{p(\mathbf{x})}{q(\textbf{x})}\text{d}\mathbf{x} = - \text{ln}\int p(\mathbf{x})\text{d}\mathbf{x} = 0.
\end{aligned}
\end{equation}
Further, $\text{D}_{\text{KL}}(q||p) = 0$ if and only if $p=q$ almost everywhere.

Due to unsymmetry of the KL divergence, to find an approximation $q$ to the target distribution $p$, there are
two options, i.e., minimizing either $\text{D}_{\text{KL}}(q||p)$ or $\text{D}_{\text{KL}}(p||q)$.
These two options lead to different approximations. It was pointed out in \cite[Section 10.1.2]{Bishop:2006}
that minimizing $\text{D}_{\text{KL}}(p||q)$ tends to find the average of modes of
$p$, while minimizing $\text{D}_{\text{KL}}(q||p)$ tends to find one exact mode.
Traditionally, the former is used in expectation propagation, and the latter in
variational Bayes. In this work, we focus on the approach $\min\text{D}_{\text{KL}}(q||p)$,
which leads to the VGA to be described below.

\begin{remark}\label{rmk:Laplace}
In practice, the so-called Laplace approximation is quite popular \cite{ThierneyKadane:1986}. Specifically, let
$\hat{\mathbf{x}}$ be the maximum a posteriori {\rm(}MAP{\rm)} estimator $\hat{\mathbf{x}}$, i.e.,
$\hat{\mathbf{x}}=\arg\min_{\mathbf{x}\in\mathbb{R}^m} g(\mathbf{x}),$ where $g(\mathbf{x}) = -\ln p(\mathbf{x}
|\mathbf{y})$ is the negative log posterior distribution. Consider the Taylor expansion of $g(\mathbf{x})$ at
the MAP estimator $\hat{\mathbf{x}}$:
\begin{equation*}
   \begin{aligned}
	g(\mathbf{x}) & \approx g(\hat{\mathbf{x}}) + (\nabla g(\hat{\mathbf{x}}),\mathbf{x} - \hat{\mathbf{x}}) + \tfrac{1}{2}(\mathbf{x}-\hat{\mathbf{x}})^t\mathbf{H}(\mathbf{x}-\hat{\mathbf{x}})\\
    & =g(\hat{\mathbf{x}}) + \tfrac{1}{2}(\mathbf{x}-\hat{\mathbf{x}})^t\mathbf{H}(\mathbf{x}-\hat{\mathbf{x}}),
   \end{aligned}
\end{equation*}
since $\nabla g(\hat{\mathbf{x}})$ vanishes.
The Hessian $\mathbf{H}$ of $g(\mathbf{x})$ is given by
\begin{equation*}
  \mathbf{H}=\mathbf{A}^t\mathrm{diag}(e^{\mathbf{A}\mathbf{\hat x}})\mathbf{A}+\mathbf{C}^{-1}_0.
\end{equation*}
Thus, $\mathbf{\hat x}$ might serve as an approximate posterior mean, and the inverse Hessian
$\mathbf{H}^{-1}$ as an approximate posterior covariance. However, unlike the VGA discussed below, it lacks the optimality as
evidence lower bound {\rm(}within the Gaussian family{\rm)}, and thus may be suboptimal for model selection etc.
\end{remark}
	
\subsection{Variational Gaussian lower bound}
Now we derive the variational Gaussian lower bound.
By substituting $p(\mathbf{x})$ with the posterior distribution $p(\mathbf{x}|\mathbf{y})$ in \eqref{eq:kl}, we obtain
\begin{equation*}
	\text{D}_{\text{KL}}(q(\mathbf{x})||p(\mathbf{x}|\mathbf{y})) = \int q(\mathbf{x})\text{ln}\frac{q(\textbf{x})}{p(\mathbf{x}|\mathbf{y})}\text{d}\mathbf{x}.
\end{equation*}
Since the posterior distribution $p(\mathbf{x}|\mathbf{y})$ depends on the unknown normalizing constant $Z(\mathbf{y})$,
the integral on the right hand side is not computable. Nonetheless,
given $\mathbf{y}$, $Z(\mathbf{y})$ is fixed. In view of the identity
\begin{equation*}
\begin{aligned}
	\text{ln}Z & 
	= \int q(\mathbf{x})\text{ln}\frac{p(\mathbf{x},\mathbf{y})}{q(\mathbf{x})}\text{d}\mathbf{x} + \int q(\mathbf{x})\text{ln}\frac{q(\mathbf{x})}{p(\mathbf{x}|\mathbf{y})}\text{d}\mathbf{x},
	\end{aligned}
\end{equation*}
instead of minimizing $\text{D}_{\text{KL}}(q(\mathbf{x})||p(\mathbf{x}|\mathbf{y}))$, we may equivalently maximize the functional
\begin{equation}\label{eq:lb}
	F(q,\mathbf{y})=\int q(\mathbf{x})\text{ln}\frac{p(\mathbf{x},\mathbf{y})}{q(\mathbf{x})}\text{d}\mathbf{x}.
\end{equation}

By \eqref{eqn:KLD-nonnegative}, we have $\text{D}_{\text{KL}}(q(\mathbf{x})||p(\mathbf{x}|\mathbf{y}))\ge0$, and thus
$ \text{ln}Z \ge F(q,\mathbf{y})$. That is, $F(q,\mathbf{y})$ provides a lower bound on the model evidence $Z$, for any choice of the
distribution $q$. For any fixed $q$, $F(q,\mathbf{y})$ may be used as a substitute for the analytically intractable model
evidence $Z(\mathbf{y})$, and hence it is called an evidence lower bound (ELBO). Since the data $\mathbf{y}$ is fixed,
it will be suppressed from $F(q,\mathbf{y})$ below. In the VGA, we
restrict our choice of $q$ to Gaussian distributions. Meanwhile, a Gaussian distribution $q(\mathbf{x})$ is fully
characterized by its mean ${\mathbf{\bar x}}\in \mathbb{R}^m$ and covariance $\mathbf{C}\in\mathcal{S}^+_m\subset \mathbb{R}^{m\times m}$, i.e.,
\begin{equation*}
q(\mathbf{x}) = \mathcal{N}(\mathbf{x};
\mathbf{\bar{x}},\mathbf{C}).
\end{equation*}
Thus, $F(q)$ is actually a function of $\mathbf{\bar{x}}\in\mathbb{R}^m$ and $\mathbf{C}\in \mathcal{S}_m^+$, and will
be written as $F(\mathbf{\bar{x}},\mathbf{C})$ below. Then the approach seeks optimal variational parameters
$(\mathbf{\bar x},\mathbf{C})$ to maximize ELBO. This step turns a challenging
sampling problem into a computationally more tractable optimization problem.
	
The next result gives an explicit expression for the lower bound $F(\mathbf{\bar{x}},\mathbf{C})$.
\begin{proposition}\label{prop:lb}
For any fixed $\mathbf{y}, \bm{\mu}_0$ and $\mathbf{C}_0$, the lower bound $F(\mathbf{\bar x},\mathbf{C})$ is given by
\begin{equation}\label{eq:lbxc}
	\begin{split}
	F(\mathbf{\bar{x}},\mathbf{C}) &= (\mathbf{y},\mathbf{A}\mathbf{\bar{x}})- (\mathbf{1}_n,e^{\mathbf{A}\mathbf{\bar{x}}+\frac{1}{2}\mathrm{diag}(\mathbf{A}\mathbf{C}\mathbf{A}^t)}) - \tfrac{1}{2}(\mathbf{\bar{x}}-\bm{\mu}_0)^t\mathbf{C}^{-1}_0(\mathbf{\bar{x}}-\bm{\mu}_0) - \tfrac{1}{2}\mathrm{tr}(\mathbf{C}^{-1}_0\mathbf{C})\\ &\hspace{1em}+ \tfrac{1}{2}\ln|\mathbf{C}| - \tfrac{1}{2}\ln|\mathbf{C_0}| + \tfrac{m}{2} - (\mathbf{1}_n,\ln(\mathbf{y}!)).
	\end{split}
\end{equation}
\end{proposition}
\begin{proof}
By the definition of the functional $F(\mathbf{\bar x},\mathbf{C})$ and the joint distribution
$p(\mathbf{x},\mathbf{y})$, we have
\begin{equation*}
	\begin{aligned}
		F(\bar{\mathbf{x}},\mathbf{C})
&=\int\mathcal{N}(\mathbf{x};\mathbf{\bar{x}},\mathbf{C})\Big[\text{ln}|\mathbf{C}_0|^{-\frac{1}{2}}-\text{ln}|\mathbf{C}|^{-\frac{1}{2}}+(\mathbf{A}\mathbf{x},\mathbf{y})-(e^{\mathbf{A}\mathbf{x}},\mathbf{1}_n)-(\text{ln}(\mathbf{y}!),\mathbf{1}_n)\\	 &\hspace{1em}-\tfrac{1}{2}(\mathbf{x}-\bm{\mu}_0)^t\mathbf{C}^{-1}_0(\mathbf{x}-\bm{\mu}_0)+\tfrac{1}{2}(\mathbf{x}-\mathbf{\bar{x}})^t\mathbf{C}^{-1}(\mathbf{x}-\mathbf{\bar{x}})\Big]\text{d}\mathbf{x}.
\end{aligned}
\end{equation*}
It suffices to evaluate the integrals termwise. Clearly, we have $
\int\mathcal{N}(\mathbf{x};\mathbf{\bar{x}},\mathbf{C})(\mathbf{A}\mathbf{x},\mathbf{y})
\text{d}\mathbf{x} = (\mathbf{A}\mathbf{\bar{x}},\mathbf{y}).$
Next, using moment generating function, we have
\begin{equation*}
   \begin{aligned}
	 \int\mathcal{N}(\mathbf{x};\mathbf{\bar{x}},\mathbf{C})(e^{\mathbf{A}\mathbf{x}},\mathbf{1}_n)\text{d}\mathbf{x}
	 &= \sum_i\int\mathcal{N}(\mathbf{x};\mathbf{\bar{x}},\mathbf{C}) e^{(\mathbf{a}_i,\mathbf{x})}\text{d}\mathbf{x}\\
	 &= \sum_ie^{(\mathbf{a}_i,\mathbf{\bar{x}})+\frac{1}{2}\mathbf{a}^t_i\mathbf{C}\mathbf{a}_i}
	 = (\mathbf{1}_n,e^{\mathbf{A}\mathbf{\bar{x}}+\frac{1}{2}\text{diag}(\mathbf{A}\mathbf{C}\mathbf{A}^t)}).
   \end{aligned}
\end{equation*}
With the Cholesky decomposition $\mathbf{C}=\mathbf{L}\mathbf{L}^t$, for $\mathbf{z}\sim \mathcal{N}(\bm{0},
\mathbf{I}_m)$, $\mathbf{x}=\bm{\mu}+\mathbf{L}\mathbf{z}\sim\mathcal{N}(\mathbf{x};\bm{\mu},\mathbf{C})$.
This and the bias-variance decomposition yield ($\mathbb{E}_{q(\mathbf{x})}[\cdot]$ takes expectation with respect to the density $q(\mathbf{x})$)
\begin{equation*}
		\mathbb{E}_{q(\mathbf{x})}[\mathbf{x}^{t}\mathbf{A}\mathbf{x}] = \mathbb{E}_{\mathcal{N}(\mathbf{z};\mathbf{0},\mathbf{I}_m)}[(\bm{\mu}+\mathbf{L}\mathbf{z})^t\mathbf{A}(\bm{\mu}+\mathbf{L}\mathbf{z})]
  = \bm{\mu}^t\mathbf{A}\bm{\mu} + \mathbf{E}_{\mathcal{N}(\mathbf{z};\mathbf{0},\mathbf{I}_m)}[\mathbf{z}^t\mathbf{L}^t\mathbf{A}\mathbf{L}\mathbf{z}].
\end{equation*}
By the cyclic property of trace, we have
$
  \mathbb{E}_{\mathcal{N}(\mathbf{z};\mathbf{0},\mathbf{I}_m)}[\mathbf{z}^t\mathbf{L}^t\mathbf{A}\mathbf{L}\mathbf{z}]= \text{tr}(\mathbf{L}^t\mathbf{A}\mathbf{L})=\text{tr}(\mathbf{A}\mathbf{L}\mathbf{L}^t)
		= \text{tr}(\mathbf{A}\mathbf{C}).
$
In particular, this gives
\begin{equation*}
  \begin{aligned}
	\mathbb{E}_{q(\mathbf{x})}[(\mathbf{x}-\bm{\mu}_0)^t\mathbf{C}^{-1}_0(\mathbf{x}-\bm{\mu}_0)]
   = (\mathbf{\bar{x}}-\bm{\mu}_0)^t\mathbf{C}^{-1}_0(\mathbf{\bar{x}}-\bm{\mu}_0) +\text{tr}(\mathbf{C}^{-1}_0\mathbf{C}),
  \end{aligned}
\end{equation*}
and
\begin{equation*}
\mathbb{E}_{q(\mathbf{x})}[(\mathbf{x}-\mathbf{\bar{x}})^t\mathbf{C}^{-1}
(\mathbf{x}-\mathbf{\bar{x}})] = m.
\end{equation*}
Collecting preceding identities completes the proof of the proposition.
\end{proof}

\begin{remark}\label{rmk:functional}
The terms in the functional $F(\mathbf{\bar{x}},\mathbf{C})$ in \eqref{eq:lbxc} admit interesting interpretation in
the lens of classical Tikhonov regularization {\rm(}see, e.g., \cite{EnglHankeNeubauer:1996,ito2014inverse,SchusterKaltenbacher:2012}{\rm)}.
To this end, it is instructive to rewrite it as
\begin{equation*}
	\begin{split}
	F(\mathbf{\bar{x}},\mathbf{C}) =& (\mathbf{y},\mathbf{A}\mathbf{\bar{x}})- (\mathbf{1}_n,e^{\mathbf{A}\mathbf{\bar{x}}+\frac{1}{2}\mathrm{diag}(\mathbf{A}\mathbf{C}\mathbf{A}^t)})  - (\mathbf{1}_n,\ln(\mathbf{y}!))\\
    & - \tfrac{1}{2}(\mathbf{\bar{x}}-\bm{\mu}_0)^t\mathbf{C}^{-1}_0(\mathbf{\bar{x}}-\bm{\mu}_0) \\
    & - \tfrac{1}{2}\mathrm{tr}(\mathbf{C}^{-1}_0\mathbf{C})+ \tfrac{1}{2}\ln|\mathbf{C}| - \tfrac{1}{2}\ln|\mathbf{C_0}| + \tfrac{m}{2}.
	\end{split}
\end{equation*}
The first line represents the fidelity or ``pseudo-likelihood'' function. It is worth noting that it actually involves
the covariance $\mathbf{C}$. In the absence of the covariance $\mathbf{C}$, it recovers the familiar log likelihood for
Poisson data, cf. Remark \ref{rmk:Laplace}.  The second line imposes a quadratic penalty on the mean $\mathbf{\bar x}$.
This term recovers the familiar penalty in Tikhonov regularization {\rm(}except that it is imposed on $\mathbf{\bar x}${\rm)}.
Recall that the function $-\ln|\mathbf{C}|$ is strictly convex in $\mathbf{C}\in\mathcal{S}_m^+$
\cite[Lemma 6.2.2]{gartner2012approximation}. Thus, one may define the corresponding Bregman divergence $d(\mathbf{C},
\mathbf{C}_0)$. In view of the identities \cite{Dwyer:1967}
\begin{equation}\label{eqn:deriv-logdet}
  \frac{\partial}{\partial\mathbf{C}}\mathrm{tr}(\mathbf{CC}_0^{-1})=\mathbf{C}_0^{-1}\quad \mbox{and} \quad \frac{\partial}{\partial\mathbf{C}}\ln|\mathbf{C}|=\mathbf{C}^{-1}.
\end{equation}
simple computation gives the following expression for the divergence:
\begin{equation*}
 d(\mathbf{C},\mathbf{C}_0)= \mathrm{tr}(\mathbf{C}^{-1}_0\mathbf{C})- \ln|\mathbf{C}_0^{-1}\mathbf{C}| - m \geq 0.
\end{equation*}
Statistically, it is the Kullback-Leibler divergence between two Gaussians
of identical mean. The divergence $d(\mathbf{C},\mathbf{C}_0)$ provides a distance measure between the prior covariance $\mathbf{C}_0$ and
the posterior one $\mathbf{C}$. Let $\{(\lambda_i,\mathbf{v}_i)\}_{i=1}^m$ be the pairs of generalized eigenvalues and eigenfunctions of
the pencil $(\mathbf{C},\mathbf{C}_0)$, i.e., $\mathbf{C}\mathbf{v}_i=\lambda_i\mathbf{C}_0\mathbf{v}_i$. Then it can be expressed as
\begin{equation*}
  d(\mathbf{C},\mathbf{C}_0)= \sum_{i=1}^m (\lambda_i-\ln \lambda_i-1).
\end{equation*}
This identity directly indicates that $d(\mathbf{C},\mathbf{C}_0)\leq c$ implies $\|\mathbf{C}\|\leq c$
and $\|\mathbf{C}^{-1}\|\leq c$, where here and below $c$ denotes a generic constant which
may change at each occurrence.

Thus, the third line regularizes the posterior covariance $\mathbf{C}$ by requesting nearness to the prior one
$\mathbf{C}_0$ in Bregman divergence. It is interesting to observe that the Gaussian prior implicitly induces
a penalty on $\mathbf{C}$, although it is not directly enforced. In statistics, the Bregman divergence $d(\mathbf{C},
\mathbf{C}_0)$ is also known as Stein's loss \cite{JamesStein:1961}. In recent years, the Bregman divergence
$d(\mathbf{C},\mathbf{C}_0)$ has been employed in clustering, graphical models, sparse covariance estimate
and low-rank matrix recovery etc. \cite{KulisSustik:2009,RavikumarWainwrightPaskuttiYu:2011}.
\end{remark}

\subsection{Theoretical properties of the lower bound}
Now we study  basic analytical properties, i.e., concavity, existence and uniqueness of maximizer, and
gradient of the functional $F(\mathbf{\bar x},\mathbf{C})$ defined in \eqref{eq:lbxc}, from the perspective
of optimization.

A first result shows the concavity of $F(\mathbf{\bar x},\mathbf{C})$. Let $X$ and $Y$
be two convex sets. Recall that a functional $f: X\times Y\to\mathbb{R}$ is said to be jointly concave, if and only if
\begin{equation*}
	f(\lambda x_1+(1-\lambda)x_2,\lambda y_1+(1-\lambda)y_2) \ge \lambda f(x_1,y_1) + (1-\lambda)f(x_2,y_2)
\end{equation*}
for all $x_1,x_2\in X$, $y_1,y_2\in Y$ and $\lambda\in [0,1]$. Further, $f$ is called strictly jointly concave if
the inequality is strict for any $(x_1,y_1)\neq(x_1,y_1)$ and $\lambda\in (0,1)$. It is easy to see that
$\mathcal{S}_m^+$ is a convex set.
\begin{theorem}\label{thm:concavity}
For any $\mathbf{C}_0\in\mathcal{S}^+_m$, the functional $F(\mathbf{\bar{x}},\mathbf{C})$ is strictly
jointly concave with respect to $\mathbf{\bar{x}}\in\mathbb{R}^m$ and $\mathbf{C}\in\mathcal{S}^+_m$.
\end{theorem}
\begin{proof}
It suffices to consider the terms apart from the linear terms $(\mathbf{y},\mathbf{A}\mathbf{\bar{x}})$ and
 $- \frac{1}{2}\text{tr}(\mathbf{C}^{-1}_0\mathbf{C})$ and the constant term $-\frac{1}{2}\text{ln}|\mathbf{C_0}|
+ \frac{m}{2} - (\mathbf{1}_n,\text{ln}(\mathbf{y}!))$. Since $\mathbf{A}\mathbf{\bar{x}}+\frac{1}{2}
\text{diag}(\mathbf{A}\mathbf{C}\mathbf{A}^t)$ is linear in $\mathbf{\bar x}$ and $\mathbf{C}$, and exponentiation
preserves convexity, the term $- (\mathbf{1}_n,e^{\mathbf{A}\mathbf{\bar{x}}+\frac{1}{2}\text{diag}(\mathbf{A}
\mathbf{C}\mathbf{A}^t)})$ is also jointly concave.
Next, the term $-\frac{1}{2}(\mathbf{\bar{x}}-\bm{\mu}_0)^t\mathbf{C}^{-1}_0(\mathbf{\bar{x}}-\bm{\mu}_0)$ is
strictly concave for any $\mathbf{C}_0\in\mathcal{S}_m^+$. Last, the log-determinant $\ln|\mathbf{C}|$ is strictly concave over
 $\mathcal{S}_m^+$ \cite[Lemma 6.2.2]{gartner2012approximation}. The assertion follows since strict concavity is preserved
under summation.
\end{proof}
	
Next, we show the well-posedness of the optimization problem in VGA.
\begin{theorem}\label{thm:existence}
There exists a unique pair
of $(\mathbf{\bar{x}},\mathbf{C})$ solving the optimization problem
\begin{equation}\label{eqn:opt}
	\max_{\mathbf{\bar x}\in\mathbb{R}^m,\mathbf{C}\in \mathcal{S}_m^+}F(\mathbf{\bar{x}},\mathbf{C})
\end{equation}
	\end{theorem}
\begin{proof}
The proof follows by direct methods in calculus of variation, and we only briefly sketch it. Clearly,
there exists a maximizing sequence, denoted by $\{(\mathbf{\bar x}^k,\mathbf{C}^k)\}
\subset\mathbb{R}^m\times\mathcal{S}_m^+$, and
we may assume $F(\mathbf{\bar x}^k,\mathbf{C}^k) \geq c=:F(\bm\mu_0,\mathbf{C}_0)$. Thus, by
\eqref{eq:lbxc} in Proposition \ref{prop:lb} and the divergence $d(\mathbf{C},\mathbf{C}_0)$, we have
\begin{equation*}
  (\mathbf{A}\mathbf{\bar x}^k,\mathbf{y}) - (\mathbf{\bar x}^k-\bm\mu_0)^t\mathbf{C}_0^{-1}(\mathbf{\bar x}^k-\bm\mu_0) - d(\mathbf{C}^k,\mathbf{C}_0) \geq c+ (e^{\mathbf{A\bar x}^k+\frac{1}{2}{\rm diag}(\mathbf{AC}^k\mathbf{A}^t)},\mathbf{1}_n)\geq c.
\end{equation*}
By the Cauchy-Schwarz inequality, we have $(\mathbf{\bar x}^k-\bm\mu_0)^t\mathbf{C}_0^{-1}
(\mathbf{\bar x}^k-\bm\mu_0) + d(\mathbf{C}^k,\mathbf{C}_0) \leq c.$ This immediately
implies a uniform bound on $\{(\mathbf{\bar x}^k,\mathbf{C}^k)\}$ and $\{(\mathbf{C}^k)^{-1}\}$.
Thus, there exists a convergent subsequence, relabeled as $\{(\mathbf{\bar x}^k,\mathbf{C}^k)\}$,
with a limit $(\mathbf{\bar x}^*,\mathbf{C}^*)\in\mathbb{R}^m\times\mathcal{S}_m^+$. Then
by the continuity of the functional $F$ in $(\mathbf{\bar x},\mathbf{C})$, we deduce
that $(\mathbf{\bar x}^*,\mathbf{C}^\ast)$ is a maximizer to $F(\mathbf{\bar x},\mathbf{C})$,
i.e., the existence of a maximizer. The uniqueness follows from
the strict joint-concavity of $F(\mathbf{\bar{x}},\mathbf{C})$, cf. Theorem \ref{thm:concavity}.
\end{proof}

Since $F$ is composed of smooth functions, clearly it is smooth. Next
we give the gradient formulae, which are useful for developing numerical algorithms below.
\begin{theorem}
The gradients of the functional $F(\mathbf{\bar{x}},\mathbf{C})$ with respect to $\mathbf{\bar{x}}$ and $\mathbf{C}$ are respectively given by
\begin{align*}
	\frac{\partial F}{\partial \mathbf{\bar{x}}} &= \mathbf{A}^t\mathbf{y} - \mathbf{A}^te^{\mathbf{A}\mathbf{\bar{x}}
+\frac{1}{2}{\rm diag}(\mathbf{A}\mathbf{C}\mathbf{A}^t)} - \mathbf{C}^{-1}_0(\mathbf{\bar{x}}-\bm{\mu}_0),\\
	\frac{\partial F}{\partial \mathbf{ C}} &= \tfrac{1}{2}[-\mathbf{A}^t{\rm diag}(e^{\mathbf{A}\mathbf{\bar{x}}
+\frac{1}{2}{\rm diag}(\mathbf{A}\mathbf{C}\mathbf{A}^t)})\mathbf{A} - \mathbf{C}^{-1}_0 + \mathbf{C}^{-1}].
\end{align*}
\end{theorem}
\begin{proof}
Let $\mathbf{d} = \mathbf{A}\mathbf{\bar{x}}+\frac{1}{2}\text{diag}(\mathbf{A}\mathbf{C}\mathbf{A}^t)$. Then by the chain rule
\begin{align*}
		\frac{\partial}{\partial \bar{x}_i}(\mathbf{1}_n,e^\mathbf{d}) &= \frac{\partial}{\partial \bar{x}_i}\sum_{j=1}^n e^{d_j} = \sum_{j=1}^n \frac{\partial e^{d_j}}{\partial d_j}\frac{\partial d_j}{\partial \bar{x}_i}= \sum_{j=1}^ne^{d_j}(A)_{ji}.
\end{align*}
That is, we have $\frac{\partial}{\partial \mathbf{\bar{x}}}(\mathbf{1}_n,e^\mathbf{d})= \mathbf{A}^te^{\mathbf{d}}$,
showing the first formula. Next we derive the gradient with respect to the covariance $\mathbf{C}$. In view of \eqref{eqn:deriv-logdet},
it remains to differentiate the term $(\mathbf{1}_n,e^{\mathbf{A}\mathbf{\bar{x}}+\frac{1}{2}\mathrm{diag}
(\mathbf{A}\mathbf{C}\mathbf{A}^t)})$ with respect to $\mathbf{C}$. To this end, let $\mathbf{H}$ be a small
perturbation to $\mathbf{C}$. By Taylor expansion, and with the diagonal matrix $\mathbf{D}=\text{diag}
(e^{\mathbf{A}\mathbf{\bar{x}}+\frac{1}{2}\text{diag}(\mathbf{A}\mathbf{C}\mathbf{A}^t)})$, we deduce
\begin{equation*}
		\hspace{1em}(\mathbf{1}_n,e^{\mathbf{A}\mathbf{\bar{x}}+\frac{1}{2}\text{diag}(\mathbf{A}\mathbf{(C+H)}\mathbf{A}^t)}) - (\mathbf{1}_n,e^{\mathbf{A}\mathbf{\bar{x}}+\frac{1}{2}\text{diag}(\mathbf{A}\mathbf{C}\mathbf{A}^t)})
		=(\mathbf{D}, \tfrac{1}{2}\text{diag}(\mathbf{A}\mathbf{H}\mathbf{A}^t))+\mathcal{O}(\|\mathbf{H}\|^2).
\end{equation*}
Since the matrix $\mathbf{D}$ is diagonal, by the cyclic property of trace, we have
\begin{equation*}
	\begin{aligned}
	(\mathbf{D}, \tfrac{1}{2}\text{diag}(\mathbf{A}\mathbf{H}\mathbf{A}^t)) =(\mathbf{D},\tfrac{1}{2}(\mathbf{A}\mathbf{H}\mathbf{A}^t))
		=\tfrac{1}{2}\text{tr}(\mathbf{DAH}^t\mathbf{A}^t)=\tfrac{1}{2}\text{tr}(\mathbf{A}^t\mathbf{DA}\mathbf{H}^t) =\tfrac{1}{2}(\mathbf{A^tDA},\mathbf{H}).
	\end{aligned}
\end{equation*}
Now the definition of the gradient completes the proof.
\end{proof}

An immediate corollary is the following optimality system.
\begin{corollary}\label{cor:opt}
The necessary and sufficient optimality system of problem \eqref{eqn:opt} is given by
\begin{align*}
	 \mathbf{A}^t\mathbf{y} - \mathbf{A}^te^{\mathbf{A}\mathbf{\bar{x}}+\frac{1}{2}{\rm diag}(\mathbf{A}\mathbf{C}\mathbf{A}^t)} - \mathbf{C}^{-1}_0(\mathbf{\bar{x}}-\bm{\mu}_0)&=0,\\
	\mathbf{C}^{-1}-\mathbf{A}^t{\rm diag}(e^{\mathbf{A}\mathbf{\bar{x}}+\frac{1}{2}{\rm diag}(\mathbf{A}\mathbf{C}\mathbf{A}^t)})\mathbf{A} - \mathbf{C}^{-1}_0&=0.
	\end{align*}
\end{corollary}

\begin{remark}
Challis and Barber \cite{challis2013gaussian} showed that for log-concave site
posterior potentials, the variational lower bound is jointly concave in $\mathbf{\bar{x}}$
and the Cholesky factor $\mathbf{L}$ of the covariance $\mathbf{C}$.
This assertion holds also for the lower bound $F(\mathbf{\bar x},\mathbf{C})$ in \eqref{eq:lbxc},
i.e., joint concavity with respect to $(\mathbf{\bar x},\mathbf{L})$.
\end{remark}
\begin{remark}
Corollary \ref{cor:opt} indicates that the covariance $\mathbf{C}^*$  of the
optimal Gaussian approximation $q^*(\mathbf{ x})$ is of the following form:
\begin{equation*}
  (\mathbf{C}^*)^{-1} = \mathbf{C}_0^{-1} + \mathbf{A}^t\mathbf{D}\mathbf{A},
\end{equation*}
for some diagonal matrix $\mathbf{D}$. Thus it is tempting that one may minimize with respect
to $\mathbf{D}$ instead of $\mathbf{C}$ in order to reduce the complexity of the algorithm, by reducing the
number of unknowns from $m^2$ to $m$. However, $F$ is generally not concave with respect
to $\mathbf{D}$; see \cite{KhanMohamedMurphy:2012} for a one-dimensional counterexample. The loss of concavity
might complicate the analysis and computation.
\end{remark}

\begin{remark}
In practice, the parameter $\mathbf{x}$ in the model \eqref{eqn:poisson} often admits physical constraint. Thus it is
natural to impose a box constraint on the mean $\mathbf{\bar x}$ in problem \eqref{eqn:opt}, e.g., $c_l\leq \bar x_i\leq c_u$,
$i=1,\ldots,m$, for some $c_l<c_u$. This can be easily
incorporated into the optimality system in Corollary \ref{cor:opt}, and the algorithms below
remain valid upon minor changes, e.g., including a pointwise projection operator in the update of $\mathbf{\bar x}$.
\end{remark}

\section{Numerical algorithm and its complexity analysis}\label{sec:algorithm}

Now we develop an efficient numerical algorithm, which is of alternating direction maximization
type, provide an analysis of its complexity, and discuss strategies for complexity reduction.

\subsection{Numerical algorithm}

In view of the strict concavity of $F(\mathbf{\bar x},\mathbf{C})$,
it suffices to solve the optimality system (cf. Corollary \ref{cor:opt}):
\begin{align}
	 \mathbf{A}^t\mathbf{y} - \mathbf{A}^te^{\mathbf{A}\mathbf{\bar{x}}+\frac{1}{2}{\rm diag}(\mathbf{A}\mathbf{C}\mathbf{A}^t)} - \mathbf{C}^{-1}_0(\mathbf{\bar{x}}-\bm{\mu}_0)&=0,\label{eqn:barx}\\
	\mathbf{C}^{-1}-\mathbf{A}^t{\rm diag}(e^{\mathbf{A}\mathbf{\bar{x}}+\frac{1}{2}{\rm diag}(\mathbf{A}\mathbf{C}\mathbf{A}^t)})\mathbf{A} - \mathbf{C}^{-1}_0&=0.\label{eqn:C}
\end{align}
This consists of a coupled nonlinear system for $(\mathbf{\bar x},\mathbf{C})$.
We shall solve the system by alternatingly maximizing
$F(\mathbf{\bar x},\mathbf{C})$ with respect to $\mathbf{\bar x}$ and
$\mathbf{C}$, i.e., coordinate ascent. From the strict concavity in Theorem \ref{thm:concavity}, we
deduce that for a fixed $\mathbf{C}$, \eqref{eqn:barx} has a unique
solution $\mathbf{\bar x}$, and similarly, for a fixed $\mathbf{\bar x}$,
\eqref{eqn:C} has a unique solution $\mathbf{C}$. Below, we discuss the efficient numerical
solution of \eqref{eqn:barx}--\eqref{eqn:C}.

\subsubsection{Newton method for updating $\mathbf{\bar x}$}

To solve $\mathbf{\bar x}$ from \eqref{eqn:barx}, for a fixed $\mathbf{C}$, we employ a Newton
method. Let the nonlinear map $\mathbf{G}:\mathbb{R}^m\to\mathbb{R}^m$ be defined by
\begin{equation*}
  \mathbf{G}(\mathbf{\bar x})= \mathbf{A}^te^{\mathbf{A}\mathbf{\bar{x}}+\frac{1}{2}{\rm diag}(\mathbf{A}\mathbf{C}\mathbf{A}^t)} + \mathbf{C}^{-1}_0(\mathbf{\bar{x}}-\bm{\mu}_0)-\mathbf{A}^t\mathbf{y}.
\end{equation*}
The Jacobian $\partial \mathbf{G}$ of the map $\mathbf{G}$ is given by
\begin{equation*}
  \partial\mathbf{G}(\mathbf{\bar x})=\mathbf{A}^t\mathrm{diag}(e^{\mathbf{A}\mathbf{\bar{x}}+\frac{1}{2}{\rm diag}(\mathbf{A}\mathbf{C}\mathbf{A}^t)})\mathbf{A} + \mathbf{C}^{-1}_0 \geq \mathbf{C}_0^{-1},
\end{equation*}
where the partial ordering $\geq$ is in the sense of symmetric positive definite matrix, i.e., $\mathbf{X}\geq \mathbf{Y}$ if and
only if $\mathbf{X}-\mathbf{Y}$ is positive semidefinite. That is, the Jacobian $\partial\mathbf{G}
(\mathbf{\bar x})$ is uniformly invertible (since the prior covariance $\mathbf{C}_0^{-1}$ is invertible). This concurs
with the strict concavity of the functional $F(\mathbf{\bar x},\mathbf{C})$ in $\mathbf{\bar x}$.

This motivates the use of the Newton method or its variants: for a nonlinear system with uniformly invertible
Jacobians, the Newton method converges globally \cite{Kelley:1995}. Specifically, given $\mathbf{\bar x}^0$, we iterate
\begin{equation}\label{eqn:iter-barx}
    \partial\mathbf{G}(\mathbf{\bar x}^k) \delta \mathbf{\bar x} = -\mathbf{G}(\mathbf{\bar x}^k),\qquad
    \mathbf{\bar x}^{k+1}  = \mathbf{\bar  x}^k + \delta\mathbf{\bar x}.
\end{equation}
The main cost of the Newton update \eqref{eqn:iter-barx} lies in solving the linear system involving
$\partial\mathbf{G}(\mathbf{\bar x}^k)$. Clearly, the Jacobian $\partial\mathbf{G}(\mathbf{\bar x}^k)$
is symmetric and positive definite, and thus the (preconditioned) conjugate gradient method is a natural
choice for solving the linear system. One may use $\mathbf{C}_0^{-1}$ (or the diagonal part
of the Jacobian $\mathbf{\partial G}(\mathbf{\bar x})$) as a preconditioner.
It is worth noting that inverting the Jacobian $\partial \mathbf{G}(\mathbf{\bar x})$ is identical
with one fixed point update of the covariance $\mathbf{C}$ below. In the presence of \textit{a priori}
structural information, this can be carried out efficiently even for very large-scale problems; see Section
\ref{ssec:complexity} below for further details. By the fast local convergence of the Newton method,
a few iterations suffice the desired accuracy, which is fully confirmed by our numerical experiments.

\subsubsection{Fixed-point method for updating $\mathbf{C}$}
Next we turn to the solution of \eqref{eqn:C} for updating  $\mathbf{C}$, with  $\mathbf{\bar x}$ fixed. There
are several different strategies, and we shall describe two of them below. The first option is to employ
a Newton method. Let the nonlinear map $\mathbf{S}:\mathbb{R}^{m\times m}\to\mathbb{R}^{m\times m}$ be defined by
\begin{equation*}
  \mathbf{S}(\mathbf{C}) = \mathbf{C}^{-1}-\mathbf{C}_0^{-1}-\mathbf{A}^t{\rm diag}(e^{\mathbf{A\bar x}+{\rm diag}(\mathbf{ACA}^t)})\mathbf{A}.
\end{equation*}
The Jacobian $\partial\mathbf{S}$ of the map $\mathbf{S}$ is given by
\begin{equation*}
  \partial\mathbf{S}(\mathbf{C})[\mathbf{H}] = -\mathbf{C}^{-1}\mathbf{H}\mathbf{C}^{-1}-\mathbf{A}^t{\rm diag}(e^{\mathbf{A\bar x}+{\rm diag}(\mathbf{ACA}^t)})\mathrm{diag}(\mathbf{AHA}^t )\mathbf{A}.
\end{equation*}
It can be verified that the map $\partial\mathbf{S}(\mathbf{C})$ is symmetric with a uniformly bounded
inverse (see the proof of Theorem \ref{thm:sensitivity-sol} in the appendix for details).
However, its explicit form seems not available due to the presence of the operator $\mathrm{diag}$.
Nonetheless, one can apply a (preconditioned) conjugate gradient method for updating $\mathbf{C}$.
 The Newton iteration is guaranteed to converge globally.

The second option is to use a fixed-point iteration. This choice is very
attractive since it avoids solving huge linear systems. Specifically, given
an initial guess $\mathbf{C}^0$, we iterate by
\begin{equation}\label{eqn:iter-C}
    \mathbf{D}^k  = \mathrm{diag}(e^{\mathbf{A\bar x}+\frac{1}{2}\mathrm{diag}(\mathbf{AC}^k\mathbf{A}^t)}),\qquad
    \mathbf{C}^{k+1}  =(\mathbf{C}_0^{-1}+ \mathbf{A}^t\mathbf{D}^k\mathbf{A})^{-1}.
\end{equation}
Conceptually, it has the flavor of a classical fixed point scheme for solving algebraic
Riccati equations in Kalman filtering \cite{AndersonKleindorfer:1969}, and it has also been used in a slightly
different context of variational inference with Gaussian processes \cite{KhanMohamedMurphy:2012}. Numerically, each inner iteration
of \eqref{eqn:iter-C} involves computing the vector ${\rm diag}(\mathbf{AC}^k\mathbf{A}^t)$ (which should be regarded
as computing $\mathbf{a}_i\mathbf{C}^k\mathbf{a}_i^t$, $i=1,\ldots,m$, instead of forming the full matrix
$\mathbf{AC}^k\mathbf{A}^t$) and a matrix inversion.

Next we briefly discuss the convergence of \eqref{eqn:iter-C}. Clearly, for all iterates $\mathbf{C}^k$,
we have $\mathbf{C}^k\leq \mathbf{C}_0$. We claim $\lambda_{\max}(\mathbf{C}^k)\leq \lambda_{\max}(\mathbf{C}_0).$ To see this,
let $\mathbf{v}\in\mathbb{R}^m$ be a unit eigenvector corresponding to the largest eigenvalue $\lambda_{\max}(\mathbf{C}^k)$, i.e.,
$\mathbf{v}^t\mathbf{C}^k\mathbf{v}=\lambda_{\max}(\mathbf{C}^k)$. Then by the minmax principle
\begin{equation*}
  \lambda_{\max}(\mathbf{C}^k)=\mathbf{v}^t\mathbf{C}^k\mathbf{v} \leq \mathbf{v}^t\mathbf{C}_0\mathbf{v} \leq \sup_{\mathbf{v}\in\mathbb{S}^{m}}\mathbf{v}^t\mathbf{C}_0\mathbf{v}=\lambda_{\max}(\mathbf{C}_0).
\end{equation*}
Thus, the sequence $\{\mathbf{C}^k \}_{k=1}^\infty$ generated by the iteration \eqref{eqn:iter-C}
is uniformly bounded in the spectral norm (and thus any norm due to the norm equivalence in a finite-dimensional space).
Hence, there exists a convergent subsequence, also relabeled as $\{\mathbf{C}^k \}$, such that $\mathbf{C}^k\to
\mathbf{C}^*$, for some $\mathbf{C}^*$. In practice, the iterates converge fairly steadily to the unique
solution to \eqref{eqn:C}, which however remains to be established. In Appendix \ref{app:iter-C}, we show a certain
``monotone'' type convergence of \eqref{eqn:iter-C} for the initial guess $\mathbf{C}^0=\mathbf{C}_0$.

\subsubsection{Variational Gaussian approximation algorithm}

With the preceding two inner solvers, we are ready to state the complete procedure in Algorithm \ref{alg:vb}.
One natural stopping criterion at Step 7 is to monitor ELBO. However, computing ELBO can be expensive and
cheap alternatives, e.g., relative change of the mean $\mathbf{\bar x}$, might be considered.
Note that Step 3 of Algorithm \ref{alg:vb}, i.e., randomized singular value decomposition (rSVD), has to be carried
out only once, and it constitutes a preprocessing step. Its crucial role will be discussed in Section \ref{ssec:complexity} below.

With exact inner updates $(\mathbf{\bar x}^k,\mathbf{C}^k)$, by
the alternating maximizing property, the sequence $\{F(\mathbf{\bar x}^k,
\mathbf{C}^k)\}$ is guaranteed to be monotonically increasing, i.e.,
\begin{equation*}
  F(\mathbf{\bar x}^0,\mathbf{C}^0) \leq F(\mathbf{\bar x}^1,\mathbf{C}^0)\leq F(\mathbf{\bar x}^1,\mathbf{C}^1)\leq ... \leq F(\mathbf{\bar x}^k,\mathbf{C}^k)\leq ...,
\end{equation*}
with the inequality being strict until convergence is reached. Further, $F(\mathbf{\bar x}^k,\mathbf{C}^k)\leq \ln Z(\mathbf{y})$. Thus,
$\{F(\mathbf{\bar x}^k,\mathbf{C}^k)\}$ converges. Further, by \cite[Prop. 2.7.1]{Bertsekas:2016}, the coordinate ascent method converges if
the maximization with respect to each coordinate is uniquely attained. Clearly, Algorithm \ref{alg:vb} is
a coordinate ascent method for $F(\mathbf{\bar x},\mathbf{C})$, and $F(\mathbf{\bar x},\mathbf{C})$ satisfies the unique solvability condition.
Thus the sequence $\{(\mathbf{\bar x}^k,\mathbf{C}^k)\}$ generated by Algorithm \ref{alg:vb}
 converges to the unique maximizer of $F(\mathbf{\bar x},\mathbf{C})$.

\begin{algorithm}[hbt!]
\centering
\caption{Variational Gaussian Approximation Algorithm\label{alg:vb}}
\begin{algorithmic}[1]
	\STATE  Input: $(\mathbf{A},\mathbf{y})$, specify the prior $(\bm{\mu}_0,\mathbf{C}_0)$, and the maximum number $K$ of iterations
    \STATE  Initialize $\mathbf{\bar x}=\mathbf{\bar x}^1$ and $\mathbf{C}=\mathbf{C}^1$;
	\STATE  SVD: $(\mathbf{U}, \mathbf{\Sigma}, \mathbf{V}) = \text{rSVD}(\mathbf{A})$;
	\FOR{$k=1,2,\ldots,K$}
	  \STATE Update the mean $\mathbf{\bar x}^{k+1}$ by \eqref{eqn:iter-barx};
      \STATE Update the covariance $\mathbf{C}^{k+1}$ by \eqref{eqn:iter-C};
    \STATE Check the stopping criterion.
	\ENDFOR
     \STATE Output: $(\mathbf{\bar{x}},\mathbf{C})$
\end{algorithmic}
\end{algorithm}

\subsection{Complexity analysis and reduction}\label{ssec:complexity}
Now we analyze the computational complexity of Algorithm \ref{alg:vb}, and describe strategies
for complexity reduction, in order to arrive at a scalable implementation.
When evaluating the functional $F(\mathbf{\bar x},\mathbf{C})$ and its gradient, the constant terms can be
precomputed. Thus, it suffices to analyze the terms that will be updated.
Standard linear algebra \cite{GolubVanLoan:2013} gives the following operational complexity.
\begin{itemize}
	\item The complexity of evaluating the objective functional $F(\mathbf{\bar{x}},\mathbf{C})$ is $\mathcal{O}(m^2n+m^3)$:
    \begin{itemize}
      \item the inner product $-(\mathbf{1}_n,e^{\mathbf{A}\mathbf{\bar{x}}+\frac{1}{2}\text{diag}(\mathbf{A}\mathbf{C}\mathbf{A}^t)})\sim\mathcal{O}(m^2n)$
      \item the matrix determinant $\text{ln}|\mathbf{C}|\sim\mathcal{O}(m^3)$
    \end{itemize}
	\item The complexity of evaluating the gradient $\frac{\partial F}{\partial \mathbf{\bar{x}}} $ is $\mathcal{O}(m^2n)$:
\begin{itemize}
  \item the matrix-vector product $\mathbf{A}^te^{\mathbf{A}\mathbf{\bar{x}}+\frac{1}{2}\text{diag}(\mathbf{A}\mathbf{C}\mathbf{A}^t)}\sim\mathcal{O}(m^2n)$
\end{itemize}
	\item The complexity of evaluating the gradient $\frac{\partial F}{\partial\mathbf{ C}}$ is $\mathcal{O}(m^2n+m^3)$:
 \begin{itemize}
   \item the matrix product $\mathbf{A}^t\text{diag}(e^{\mathbf{A}\mathbf{\bar{x}}+\frac{1}{2}\text{diag}(\mathbf{A}\mathbf{C}\mathbf{A}^t)})\mathbf{A}\sim\mathcal{O}(m^2n)$
   \item the matrix inversion $\mathbf{C}^{-1}\sim\mathcal{O}(m^3)$.
 \end{itemize}
\end{itemize}

In summary, evaluating ELBO $F(\mathbf{\bar x},\mathbf{C})$ and its gradients
each involves $\mathcal{O}(nm^2+m^3)$ complexity, which is infeasible for large-scale
problems. The most expensive piece lies in matrix products/inversion, e.g., $(\mathbf{1}_n,e^{\mathbf{A}
\mathbf{\bar{x}}+\frac{1}{2}\text{diag}(\mathbf{A}\mathbf{C}\mathbf{A}^t)})$, $\mathbf{A}^te^{\mathbf{A}
\mathbf{\bar{x}}+\frac{1}{2}\text{diag}(\mathbf{A}\mathbf{C}\mathbf{A}^t)}$ and $\mathbf{A}^t\text{diag}
(e^{\mathbf{A}\mathbf{\bar{x}}+\frac{1}{2}\text{diag}(\mathbf{A}\mathbf{C}\mathbf{A}^t)})\mathbf{A}$.
The log-determinant $\text{ln}|\mathbf{C}|$ can be approximated accurately
with $\mathcal{O}(m^2)$ operations by a stochastic algorithm \cite{ZhangLeithead:2007}.
In many practical inverse problems, there do exist structures: (i) $\mathbf{A}$ is
low rank, and (ii) $\mathbf{C}$ is sparse, which can be leveraged to reduce the per-iteration cost.

First, for many inverse problems, the matrix $\mathbf{A}$ is ill-conditioned, and the singular
values decay to zero. Thus, $\mathbf{A}$ naturally has a low-rank structure. The effective rank $r$
is determined by the decay rate of the singular values. In this work, we assume a known rank $r$.
The rSVD is a powerful technique for obtaining low-rank approximations \cite{halko2011finding}. For a
rank $r$ matrix, the rSVD can yield an accurate approximation with $\mathcal{O}(mn\ln r + (m + n)r^2)$
operations \cite[pp. 225]{halko2011finding}. We denote the rSVD approximation by $\mathbf{A}\approx
\mathbf{U}\mathbf{\Sigma}\mathbf{V}^t$, where the matrices $\mathbf{U}\in\mathbb{R}^{n\times r}$
and $\mathbf{V}\in\mathbb{R}^{m\times r}$ are column orthonormal, and $\mathbf{\Sigma}\in
\mathbb{R}^{r\times r}$ is diagonal with its entries ordered nonincreasingly.

Second, the covariance $\mathbf{C}$ is approximately sparse, and each row/column has at most $s$
nonzero entries. This reflects the fact that only (physically) neighboring elements
are highly correlated, and there is no long range correlation. This choice will be implemented in
the numerical experiments for 2D image deblurring. Naturally, one can also consider a more
flexible option by adaptively selecting the sparsity pattern. This can be achieved by penalizing of the
off-diagonal entries of $\mathbf{C}$ by the $\ell^1$-norm, which allows automatically detecting
significant correlation \cite{RavikumarWainwrightPaskuttiYu:2011}. Other structures, e.g., low-rank
plus sparsity, offer potential alternatives. We leave these advanced
options to a future study.
	
Under these structural assumptions, the complexity of computing the terms $(\mathbf{1}_n,e^{\mathbf{A}\mathbf{\bar{x}}+\frac{1}{2}\text{diag}(\mathbf{A}\mathbf{C}\mathbf{A}^t)})$, $ \mathbf{A}^te^{\mathbf{A}\mathbf{\bar{x}}+\frac{1}{2}\text{diag}(\mathbf{A}\mathbf{C}\mathbf{A}^t)}$ and $\mathbf{A}^t\text{diag}(e^{\mathbf{A}\mathbf{\bar{x}}+\frac{1}{2}\text{diag}(\mathbf{A}\mathbf{C}\mathbf{A}^t)})\mathbf{A}$
can be reduced to $\mathcal{O}(smn)$. Thus, the complexity of calculating $F$ and $\frac{\partial F}{\partial
\mathbf{\bar{x}}}$ is reduced to $\mathcal{O}(smn+m^2)$. For the matrix inversion in
\eqref{eqn:iter-C}, we exploit the low-rank structure of $\mathbf{A}$. Upon recalling the low-rank approximation
of $\mathbf{A}$ and the Sherman-Morrison-Woodbury formula \cite[pp. 65]{GolubVanLoan:2013}, i.e.,
\begin{equation*}
(\mathbf{\tilde{A}}+\mathbf{\tilde{U}}\mathbf{\tilde{V}})^{-1} = \mathbf{\tilde{A}}^{-1}-\mathbf{\tilde{A}}^{-1}\mathbf{\tilde{U}}(\mathbf{I}+\mathbf{\tilde{V}}
\mathbf{\tilde{A}}^{-1}\mathbf{\tilde{U}})^{-1}\mathbf{\tilde{V}}\mathbf{\tilde{A}}^{-1},
\end{equation*}
 we deduce (with $\mathbf{D}=\text{diag}(e^{\mathbf{A}\mathbf{\bar{x}}+
\frac{1}{2}\text{diag}(\mathbf{A}\mathbf{C}\mathbf{A}^t)})$)
\begin{equation}\label{eqn:update-C-woodbury}
\mathbf{C} = \mathbf{C}_0 -\mathbf{C}_0\mathbf{V}\mathbf{\Sigma}\mathbf{U}^t\mathbf{D}\mathbf{U}\mathbf{\Sigma}(\mathbf{I}
+\mathbf{V}^t\mathbf{C}_0\mathbf{V}\mathbf{\Sigma}\mathbf{U}^t\mathbf{D}\mathbf{U}\mathbf{\Sigma})^{-1}\mathbf{V}^t\mathbf{C}_0.
\end{equation}
Note that the inversion step only involves a matrix in $\mathbb{R}^{r\times r}$, and can be carried out
efficiently. The sparsity structure on $\mathbf{C}$ can be enforced by computing only the respective entries.
Then the update formula \eqref{eqn:update-C-woodbury} can be achieved
in $\mathcal{O}(smn+r^2n+r^2m)$ operations. In comparison with the $\mathcal{O}(m^3+nm^2)$ complexity
of the direct implementation, this represents a substantial complexity reduction.

\section{Hyperparameter choice with hierarchical model}\label{sec:hyper}
When encoding prior knowledge about the unknown $\mathbf{x}$ into the prior $p(\mathbf{x})$, it is often necessary
to tune its strength, a scalar parameter commonly known as hyperparameter.
It plays the role of the regularization parameter in variational regularization
\cite[Chapter 7]{ito2014inverse}, where its proper choice is notoriously challenging.
In the Gaussian prior $p(\mathbf{x})$, $\mathbf{C}_0=\alpha^{-1}\bar{\mathbf{C}}_0$, where
$\bar{\mathbf{C}}_0$ describes the interaction structure and
the scalar $\alpha$ determines the strength of the interaction which has to be specified.

In the Bayesian paradigm, one principled approach to handle hyperparameters is hierarchical modeling,
by assuming a hyperprior and treating them as a part of the
inference procedure. Specifically, we write the Gaussian prior $p(\mathbf{x}|\alpha)=\mathcal{N}(\mathbf{x}
|\mathbf{0},\alpha^{-1}\bar{\mathbf{C}}_0)$, and employ
a Gamma distribution $p(\alpha|a,b)=\text{Gamma}(\alpha|a,b)$ on $\alpha$, where $(a,b)$
are the parameters. The Gamma distribution is the conjugate prior for $\alpha$, and it is
analytically and computationally convenient. In practice, one may take $(a,b)$ close
to $(1,0)$ to mimic a noninformative prior. Then appealing to Bayes' formula again, one obtains a
 posterior distribution (jointly over $(\mathbf{x},\alpha)$).
Conceptually, with the VGA, this construction determines the optimal parameter by maximizing ELBO as a function of
$\alpha$, i.e., model selection within a parametric family. Thus it can be viewed as a direct application
of ELBO in model selection.

One may explore the resulting joint posterior distribution in several ways \cite[Chapter 7]{ito2014inverse}.
In this work, we employ an EM type method to maximize the following (joint) lower bound
\begin{equation*}
\begin{aligned}
	F(\mathbf{\bar{x}},\mathbf{C},\alpha) &= \int q(\mathbf{x})\text{ln}\frac{p(\mathbf{x},\mathbf{y}|\alpha)p(\alpha|a,b)}{q(\mathbf{x})}\text{d}\mathbf{x}\\
	&= \int q(\mathbf{x})\text{ln}\frac{p(\mathbf{x},\mathbf{y}|\alpha)}{q(\mathbf{x})}\text{d}\mathbf{x} + \int q(\mathbf{x})\text{ln}p(\alpha|a,b)\text{d}\mathbf{x}\\
	&= F_\alpha(\mathbf{\bar{x}},\mathbf{C}) + (a-1)\ln\alpha -\alpha b + \ln\frac{b^a}{\Gamma(a)},
\end{aligned}
\end{equation*}
where the subscript  $\alpha$ indicates the dependence of ELBO
on $\alpha$. Then, using \eqref{eq:lbxc} and substituting $\mathbf{C}_0$ with $\alpha^{-1}\bar{\mathbf{C}}_0$, we have
\begin{equation}\label{eq:hyperbd}
	\begin{split}
	F(\mathbf{\bar{x}},\mathbf{C},\alpha) &= (\mathbf{y},\mathbf{A}\mathbf{\bar{x}})- (\mathbf{1}_n,e^{\mathbf{A}\mathbf{\bar{x}}+\tfrac{1}{2}\mathrm{diag}(\mathbf{A}\mathbf{C}\mathbf{A}^t)}) - \tfrac{\alpha}{2}(\mathbf{\bar{x}}-\bm{\mu}_0)^t\bar{\mathbf{C}}_0^{-1}(\mathbf{\bar{x}}-\bm{\mu}_0) - \tfrac{\alpha}{2}\mathrm{tr}(\bar{\mathbf{C}}_0^{-1}\mathbf{C})\\
&\hspace{1em}+ \tfrac{1}{2}\ln|\mathbf{C}| + \tfrac{m}{2}\ln\alpha - \tfrac{1}{2}\ln|\bar{\mathbf{C}}_0|+  (a-1)\ln\alpha -\alpha b + \tfrac{m}{2} - (\mathbf{1}_n,\ln(\mathbf{y}!)) + \ln\frac{b^a}{\Gamma(a)}.
	\end{split}
\end{equation}
This functional extends ELBO $F(\mathbf{\bar x},\mathbf{C})$ to estimate the
hyperparameter $\alpha$ simultaneously with $(\mathbf{\bar x},\mathbf{C})$ in a way
analogous to augmented Tikhonov regularization \cite{JinZou:2009}.

To maximize $F(\mathbf{\bar x},\mathbf{C},\alpha)$, we employ an EM algorithm \cite[Chapter 9.3]{Bishop:2006}. In
the E-step, we fix $\alpha$, and maximize $F(\mathbf{\bar{x}},\mathbf{C},\alpha)$ for a new Gaussian approximation
$\mathcal{N}(\mathbf{x}|\mathbf{\bar{x}},\mathbf{C})$ by Algorithm \ref{alg:vb}, with the unique maximizer
denoted by $(\mathbf{\bar x}_\alpha,\mathbf{C}_\alpha)$. Then in the M-step, we fix $(\mathbf{\bar{x}},\mathbf{C})$
and update $\alpha$ by
\begin{equation}\label{eq:alpha}
	\alpha = \frac{m+2(a-1)}{(\mathbf{\bar{x}}_\alpha-\bm{\mu}_0)^t\mathbf{\bar C}_0^{-1}(\mathbf{\bar{x}}_\alpha-\bm{\mu}_0)+\text{tr}(\bar{\mathbf{C}}_0^{-1}\mathbf{C}_\alpha)+2b}.
\end{equation}
This follows from the condition $\frac{\partial F}{\partial \alpha}=0$. These
discussions lead to the procedure in Algorithm \ref{alg:hyper}. A natural
stopping criterion at line 5 is the change of $\alpha$. Below we analyze the convergence of
Algorithm \ref{alg:hyper}.

\begin{remark}
The first two terms in the denominator of the iteration \eqref{eq:alpha} is given by
\begin{equation*}
  \alpha(\mathbf{\bar{x}}_\alpha-\bm{\mu}_0)^t\mathbf{\bar C}_0^{-1}(\mathbf{\bar{x}}_\alpha-\bm{\mu}_0)+\alpha{\rm tr}(\bar{\mathbf{C}}_0^{-1}\mathbf{C}_\alpha)
  =\mathbb{E}_{q(\mathbf{x})}[\|\mathbf{x}-\bm{\mu}_0\|^2_{\mathbf{C}_0^{-1}}],
\end{equation*}
i.e., the expectation of the negative logarithm of the Gaussian prior $p(\mathbf{x})$ with respect to the Gaussian posterior approximation
$q(\mathbf{x})$. Formally, the fixed point iteration \eqref{eq:alpha} can be viewed as an extension of that for
a balancing principle for Tikhonov regularization in \cite{JinZou:2009,ItoJinTakeuchi:2011} to a probabilistic
context.
\end{remark}

\begin{algorithm}[hbt!]
	\centering
     \caption{Hierarchical variational Gaussian approximation\label{alg:hyper}}
	\begin{algorithmic}[1]
	\STATE  Input $(\mathbf{A},\mathbf{y})$, and initialize $\alpha^1$
	\FOR{$k=1,2,\ldots$}
	\STATE  E-step: Update $(\mathbf{\bar{x}}^k,\mathbf{C}^k)$ by Algorithm \ref{alg:vb}:
      \begin{equation*}
         (\mathbf{\bar{x}}^k,\mathbf{C}^k)=\arg\max_{(\mathbf{\bar x},\mathbf{C})\in\mathbb{R}^m\times\mathcal{S}_m^+} F_{\alpha^k}(\mathbf{\bar{x}},\mathbf{C});
      \end{equation*}
	\STATE  M-step: Update $\alpha$ by \eqref{eq:alpha}.
    \STATE  Check the stopping criterion;
	\ENDFOR
     \STATE Output: $(\mathbf{\bar{x}},\mathbf{C})$
	\end{algorithmic}
\end{algorithm}

In order to analyze the convergence of Algorithm \ref{alg:hyper}, we write the functional $F_\alpha(\mathbf{\bar x},\mathbf{C})$ as
\begin{equation*}
  F_\alpha(\mathbf{\bar x},\mathbf{C}) = \phi(\mathbf{\bar x},\mathbf{C}) + \alpha \psi(\mathbf{\bar x},\mathbf{C}),
\end{equation*}
where
\begin{equation*}
  \begin{aligned}
     \phi(\mathbf{\bar{x}},\mathbf{C})&= (\mathbf{y},\mathbf{A}\mathbf{\bar{x}})- (\mathbf{1}_n,e^{\mathbf{A}\mathbf{\bar{x}}+\tfrac{1}{2}\mathrm{diag}(\mathbf{A}\mathbf{C}\mathbf{A}^t)})+\tfrac{1}{2}\ln|\mathbf{C}|- \tfrac{1}{2}\ln|\bar{\mathbf{C}}_0|+ - (\mathbf{1}_n,\ln(\mathbf{y}!)),\\
     \psi(\mathbf{\bar{x}},\mathbf{C})&=- \tfrac{1}{2}(\mathbf{\bar{x}}-\bm{\mu}_0)^t\bar{\mathbf{C}}_0^{-1}(\mathbf{\bar{x}}-\bm{\mu}_0)- \tfrac{1}{2}\mathrm{tr}(\bar{\mathbf{C}}_0^{-1}\mathbf{C})\le 0.
  \end{aligned}
\end{equation*}
Thus the functional $F_\alpha(\mathbf{\bar x},\mathbf{C})$ resembles classical Tikhonov regularization.
By Theorem \ref{thm:existence}, for any $\alpha>0$, there exists a unique maximizer $(\mathbf{\bar x}_\alpha,\mathbf{C}_\alpha)$ to $F_\alpha$,
and the value function $\psi(\mathbf{\bar x}_\alpha,\mathbf{C}_\alpha)$ is continuous in $\alpha$, cf. Lemma \ref{lem:cont}
below. In Appendix \ref{app:sensitivity}, we show that the maximizer $(\mathbf{\bar x}_\alpha,\mathbf{C}_\alpha)$ is actually differentiable in $\alpha$.

\begin{lemma}\label{lem:bdd}
For any $\alpha>0$, the maximizer $(\mathbf{\bar x}_\alpha,\mathbf{C}_\alpha)$ is bounded, with the bound depending only on $\alpha$.
\end{lemma}
\begin{proof}
Taking inner product between \eqref{eqn:barx} and $\mathbf{\bar x}_\alpha$, we deduce
\begin{equation*}
   (\mathbf{C}_0^{-1}\mathbf{\bar x}_\alpha,\mathbf{\bar x}_\alpha) + (e^{\mathbf{A\bar x}_\alpha+{\rm diag}(\mathbf{ACA}^t)},\mathbf{A}\mathbf{\bar x}_\alpha)=(\mathbf{A}^t\mathbf{y},\mathbf{\bar x}_\alpha).
\end{equation*}
It can be verified directly that the function $f(t)=te^t$ is bounded from below by $-e^{-1}$ for $t\in\mathbb{R}$.
Meanwhile, by \eqref{eqn:C}, $\mathbf{C}\leq \mathbf{C}_0$, and thus
\begin{equation*}
  (e^{\mathbf{A\bar x}_\alpha+{\rm diag}(\mathbf{ACA}^t)},\mathbf{A}\mathbf{\bar x}_\alpha) \geq -e^{-1}\sum_ie^{{\rm diag}(\mathbf{ACA}^t)_i} \geq -e^{-1}\sum_ie^{{\rm diag}(\mathbf{AC}_0\mathbf{A}^t)_i}=-ce^{-1}.
\end{equation*}
This and the Cauchy-Schwarz inequality give $\|\mathbf{\bar x}_\alpha\|\leq c\alpha^{-1}$,
with $c$ depending only on $\mathbf{y}$. Next, by \eqref{eqn:C}, we have
\begin{equation*}
 0\leq  e^{(\mathbf{A\bar x})_i+{\rm diag}(\mathbf{ACA}^t)_i}\leq e^{(\mathbf{A\bar x})_i+{\rm diag}(\mathbf{AC}_0\mathbf{A}^t)_i}\leq c,
\end{equation*}
and consequently appealing to \eqref{eqn:C} again yields
$(\mathbf{C}_0^{-1}+c\mathbf{A}^t\mathbf{A})^{-1}\leq \mathbf{C}\leq \mathbf{C}_0$, completing the proof.
\end{proof}

\begin{lemma}\label{lem:cont}
The functional value $\psi(\mathbf{\bar{x}}_\alpha,\mathbf{C}_\alpha)$
is continuous at any $\alpha>0$.
\end{lemma}
\begin{proof}
Let $\{\alpha^k\}\subset\mathbb{R}^+$ be a sequence convergent to $\alpha$. By Theorem \ref{thm:existence}, for each $\alpha^k$,
there exists a unique maximizer $(\mathbf{\bar x}^k,\mathbf{C}^k)$ to $F_{\alpha^k}(\mathbf{\bar x},\mathbf{C})$.
By Lemma \ref{lem:bdd}, the sequence $\{(\mathbf{\bar x}^k,\mathbf{C}^k)\}$
is uniformly bounded, and there exists a convergent subsequence, relabeled as $\{(\mathbf{\bar x}^k,\mathbf{C}^k)\} $, with a
limit $(\mathbf{\bar x}^*,\mathbf{C}^*)$. By the continuity of the functionals $\phi(\mathbf{\bar x},\mathbf{C})$ and
$\psi(\mathbf{\bar x},\mathbf{C})$, we have for any $(\mathbf{\bar x},\mathbf{C})\in\mathbb{R}^m\times\mathcal{S}_m^+$
\begin{equation*}
  \begin{aligned}
    F_\alpha(\mathbf{\bar x}^*,\mathbf{C}^*)& = \lim_{k\to\infty}(\phi(\mathbf{\bar x}^k,\mathbf{C}^k)+\alpha_k \psi(\mathbf{\bar x}^k,\mathbf{C}^k))
      \geq \lim_{k\to\infty}(\phi(\mathbf{\bar x},\mathbf{C})+\alpha_k \psi(\mathbf{\bar x},\mathbf{C}))\\
     & = \phi(\mathbf{\bar x},\mathbf{C})+\alpha \psi(\mathbf{\bar x},\mathbf{C})=F_\alpha(\mathbf{\bar x},\mathbf{C}).
  \end{aligned}
\end{equation*}
That is, $(\mathbf{\bar x}^*,\mathbf{C}^*)$ is a maximizer of $F_\alpha(\mathbf{\bar x},\mathbf{C})$.
The uniqueness of the maximizer to $F_\alpha(\mathbf{\bar x},\mathbf{C})$ and a standard subsequence
argument imply that the whole sequence converges. The desired continuity now follows by the
continuity of $\psi(\mathbf{\bar x},\mathbf{C})$ in $(\mathbf{\bar x},\mathbf{C})$.
\end{proof}

Next we give an important monotonicity relation for $\psi(\mathbf{\bar x}_\alpha,\mathbf{C}_\alpha)$
in $\alpha$, in a manner similar to classical Tikhonov regularization \cite{ItoJinTakeuchi:2011}.
In Appendix \ref{app:sensitivity}, we show that it is actually strictly monotone.
\begin{lemma}\label{lem:mon}
The functional $\psi(\mathbf{\bar{x}}_\alpha,\mathbf{C}_\alpha)$ is monotonically increasing in $\alpha$.
\end{lemma}
\begin{proof}
This result follows by a standard comparison principle.
For any $\alpha_1,\alpha_2$, by the maximizing property of $(\mathbf{C}_{\alpha_1},\mathbf{\bar x}_{\alpha_1})$
and $(\mathbf{C}_{\alpha_2},\mathbf{\bar x}_{\alpha_2})$, we have
\begin{align*}
		F_{\alpha_1}(\mathbf{\bar{x}}_{\alpha_1},\mathbf{C}_{\alpha_1}) \ge F_{\alpha_1}(\mathbf{\bar{x}}_{\alpha_2},\mathbf{C}_{\alpha_2})\quad \mbox{and}\quad
		F_{\alpha_2}(\mathbf{\bar{x}}_{\alpha_2},\mathbf{C}_{\alpha_2}) \ge F_{\alpha_2}(\mathbf{\bar{x}}_{\alpha_1},\mathbf{C}_{\alpha_1}).
\end{align*}
Summing up these two inequalities and collecting terms yield
\begin{align*}
	(\alpha_1-\alpha_2)[\psi(\mathbf{\bar{x}}_{\alpha_1},\mathbf{C}_{\alpha_1})-\psi(\mathbf{\bar{x}}_{\alpha_2},\mathbf{C}_{\alpha_2})] \ge 0.
\end{align*}
Then the desired monotonicity relation follows.
\end{proof}
\begin{theorem}\label{thm:mono}
For any initial guess $\alpha^1>0$,
the sequence $\{\alpha^k\}$ generated by Algorithm \ref{alg:hyper} is monotonically convergent
to some $\alpha^*\geq0$, and if the limit $\alpha^*>0$, then it satisfies
the fixed point equation \eqref{eq:alpha}.
\end{theorem}
\begin{proof}
By the fixed point iteration \eqref{eq:alpha}, we have (with $c=\frac{m}{2}+a-1$)
	\begin{equation*}
	\begin{split}
		\alpha^{k+1}-\alpha^{k} &= \frac{c}{-\psi(\mathbf{\bar{x}}_{\alpha^{k}},\mathbf{C}_{\alpha^{k}})+b} - \frac{c}{-\psi(\mathbf{\bar{x}}_{\alpha^{k-1}},\mathbf{C}_{\alpha^{k-1}})+b}\\ &=\frac{c[\psi(\mathbf{\bar{x}}_{\alpha^{k}},\mathbf{C}_{\alpha^{k}})-\psi(\mathbf{\bar{x}}_{\alpha^{k-1}},\mathbf{C}_{\alpha^{k-1}})]}{(-\psi(\mathbf{\bar{x}}_{\alpha^{k}},\mathbf{C}_{\alpha^{k}})+b)(-\psi(\mathbf{\bar{x}}_{\alpha^{k-1}},\mathbf{C}_{\alpha^{k-1}})+b)}.
	\end{split}
	\end{equation*}
Since $\psi\leq 0$, the denominator is positive.
By Lemma \ref{lem:mon}, $\alpha^{k+1}-\alpha^{k}$
and $\alpha^k-\alpha^{k-1}$ have the same sign, and thus
 $\{\alpha^k\}$ is monotone. Further, for all $\alpha^k$, we have
$0 \le \alpha^k \le \frac{m+2(a-1)}{2b},$
i.e., $\{\alpha^k\}$ is uniformly bounded. Thus
$\{\alpha^k\}$ is convergent. By Lemma   \ref{lem:cont}, $\psi(\mathbf{\bar{x}}_\alpha,\mathbf{C}_\alpha)$ is
continuous in $\alpha$ for $\alpha>0$, and  $\alpha^*$ satisfies \eqref{eq:alpha}.
\end{proof}

\begin{remark}
The proof of Theorem \ref{thm:mono} provides a constructive approach to the existence of a solution to \eqref{eq:alpha}.
The uniqueness of the solution $\alpha^*$ to \eqref{eq:alpha} is generally not ensured. However, in practice, it
seems to have only two fixed points: one is in the neighborhood of $+\infty$, which is uninteresting, and the other
is the desired one.
\end{remark}

\section{Numerical experiments and discussions}\label{sec:numer}
Now we present numerical results to examine algorithmic features (Sections \ref{ssec:iter}--\ref{ssec:mcmc},
with the example \texttt{phillips}) and to illustrate the VGA (Section \ref{ssec:recon}). All one-dimensional examples are taken from
public domain \texttt{MATLAB} package \texttt{Regutools}\footnote{\url{http://www.imm.dtu.dk/~pcha/Regutools/}, last
accessed on April 15, 2017}, and the discrete problems are of size $100\times 100$. We refer the prior with a zero mean $\bm \mu_0=\mathbf{0}$
and the covariance $\alpha^{-1}\mathbf{I}_m$ and $\alpha^{-1}\mathbf{L}_1^{-1}\mathbf{L}_1^{-t}$ (with $\mathbf{L}_1$
being the 1D first-order forward difference matrix) to as the $L^2$- and $H^1$-prior, respectively, and let
$\bar {\mathbf{C}}_{0}=\mathbf{I}_m$, and $\bar{\mathbf C}_1=\mathbf{L}_1^{-1}\mathbf{L}_1^{-t}$. Unless otherwise
specified, the parameter $\alpha$ is determined in a trial-and-error manner, and in Algorithm
\ref{alg:vb}, the Newton update $\delta\mathbf{\bar x}$ in \eqref{eqn:iter-barx} is computed by
the \texttt{MATLAB} built-in function \texttt{pcg} with a default tolerance, the prior covariance $\mathbf{C}^{-1}_0$
as the preconditioner and a maximum $10$ PCG iterations.

\subsection{Convergence behavior of inner and outer iterations of Algorithm \ref{alg:vb}}\label{ssec:iter}

First, we examine the convergence behavior of inner iterations for updating $\mathbf{\bar x}$ and $\mathbf{C}$,
i.e., \eqref{eqn:iter-barx} and \eqref{eqn:iter-C}, for the example \texttt{phillips} with the $L^2$-prior
$\mathbf{C}_0=1.0\times 10^{-1}\mathbf{\bar C}_0$ and $H^1$-prior $\mathbf{C}_0=2.5\times10^{-3}\mathbf{\bar C}_1$.
To study the convergence, we fix $\mathbf{C}$ at $\mathbf{C}^1=\mathbf{I}$ for $\mathbf{\bar x}$
and present the $\ell^2$-norm of the update $\delta\mathbf{\bar x}$ (initialized with $\mathbf{\bar x}^0=\mathbf{0}$),
and similarly fix $\mathbf{\bar x}$ at the converged iterate $\mathbf{\bar x}^1$ for $\mathbf{C}$ and present
the spectral norm of the change $\delta\mathbf{C}$. For both \eqref{eqn:iter-barx} and \eqref{eqn:iter-C},
these initial guesses are quite far away from the solutions, and thus the choice allows showing their global
convergence behavior. The convergence is fairly rapid and steady for both inner iterations, cf. Fig. \ref{fig:inner}.
For example, for a tolerance $10^{-5}$, the Newton method \eqref{eqn:iter-barx} converges after about 10 iterations,
and the fixed point method \eqref{eqn:iter-C} converges after $4$ iterations,
respectively. The global as well as local superlinear convergence of the Newton method \eqref{eqn:iter-barx} are
clearly observed, confirming the discussions in Section \ref{sec:algorithm}. The convergence behavior of
the inner iterations is similar for both priors. In practice, it is unnecessary to solve the inner iterates to a
very high accuracy, and it suffices to apply a few inner updates within each outer iteration. Since the iteration
\eqref{eqn:iter-C} often converges faster than \eqref{eqn:barx}, we take five Newton
updates and one fixed point update per outer iteration for the numerical experiments below.

\begin{figure}[htb!]
\centering
\begin{tabular}{cc}
\includegraphics[scale=0.45]{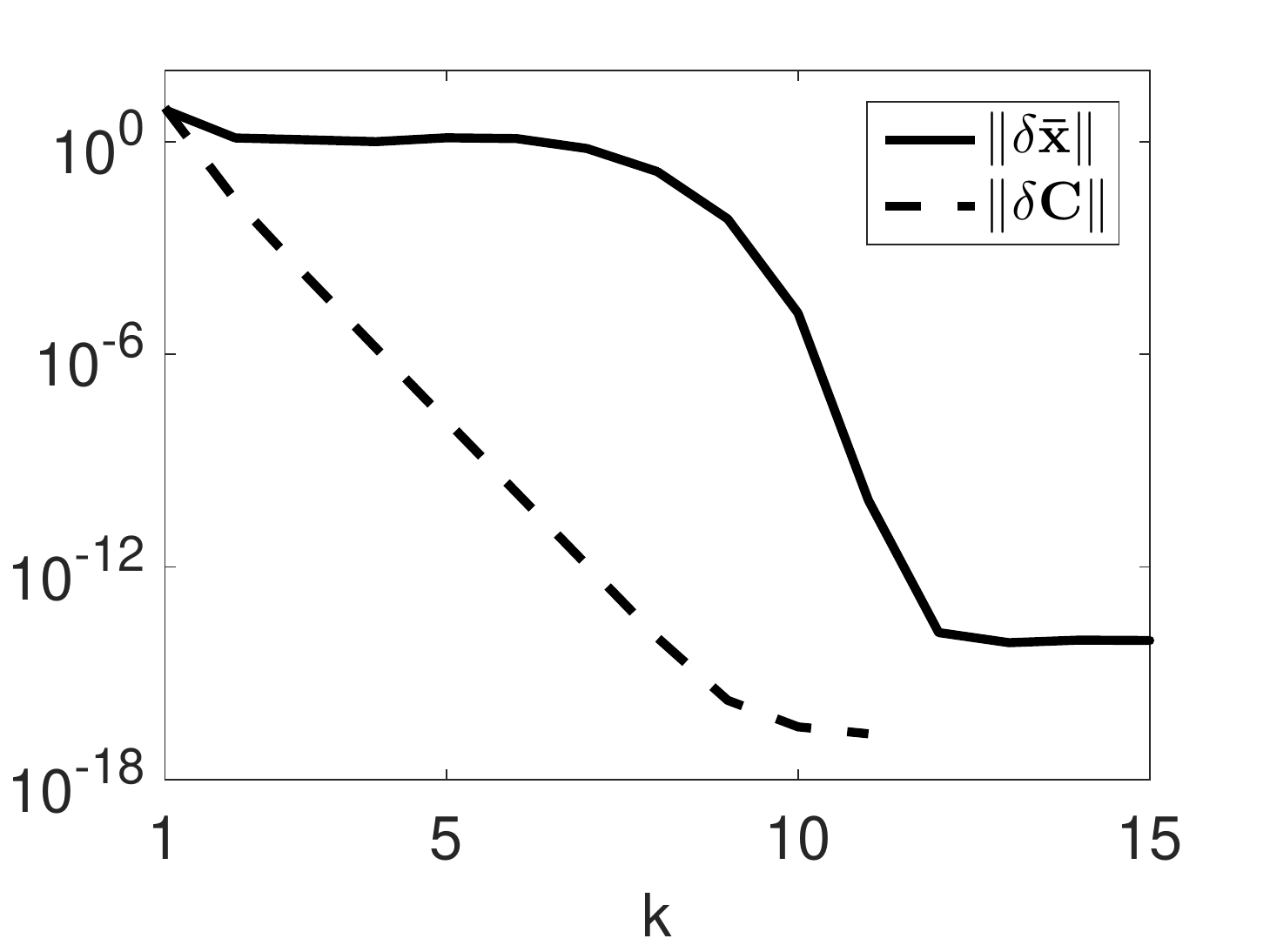} & \includegraphics[scale=0.45]{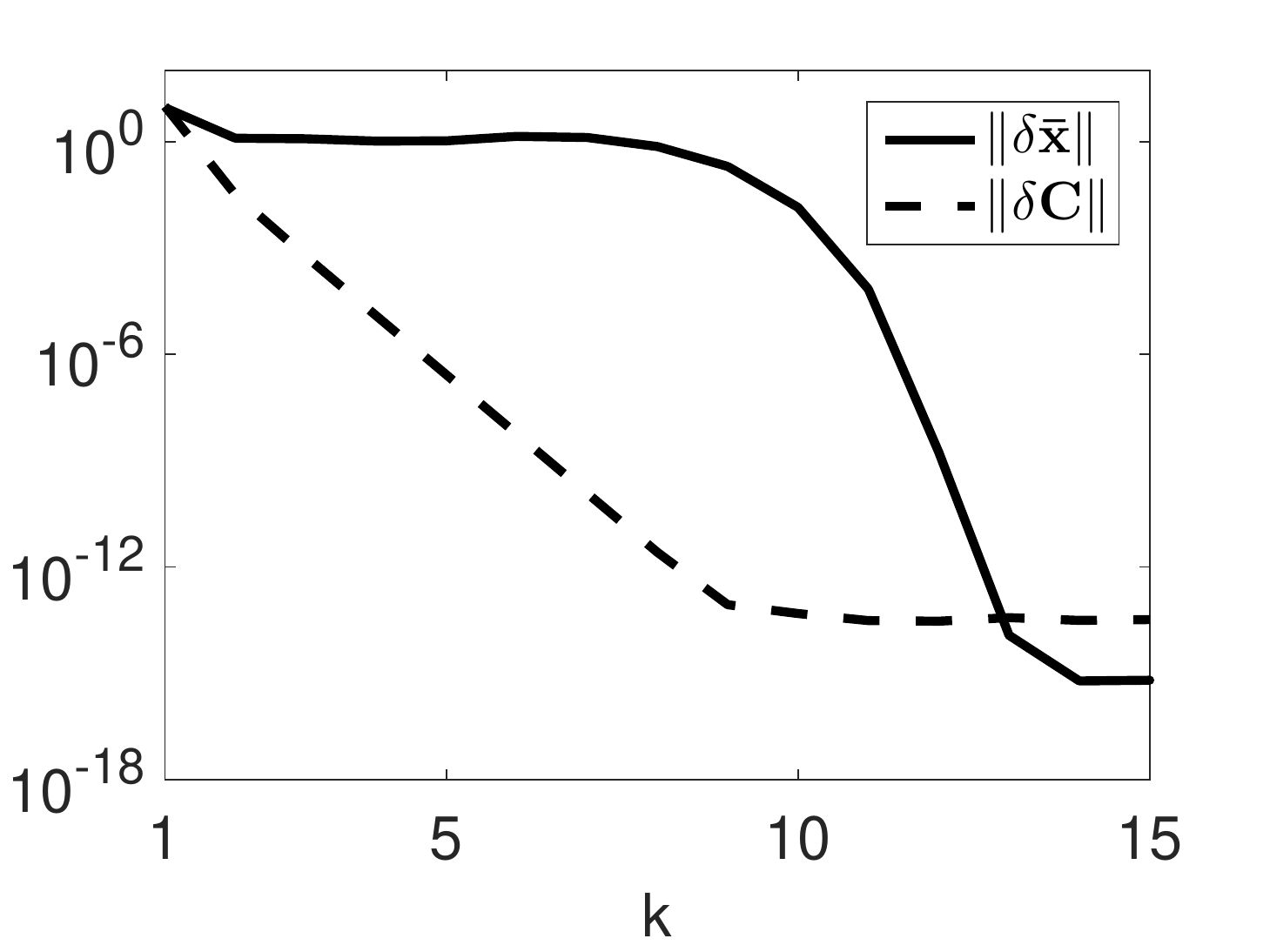}\\
(a) $L^2$-prior & (b) $H^1$-prior
\end{tabular}
\caption{The convergence of the inner iterations of Algorithm \ref{alg:vb} for \texttt{phillips}.\label{fig:inner}}
\end{figure}

\begin{figure}[hbt!]
\centering
\begin{tabular}{cc}
\includegraphics[scale=0.45]{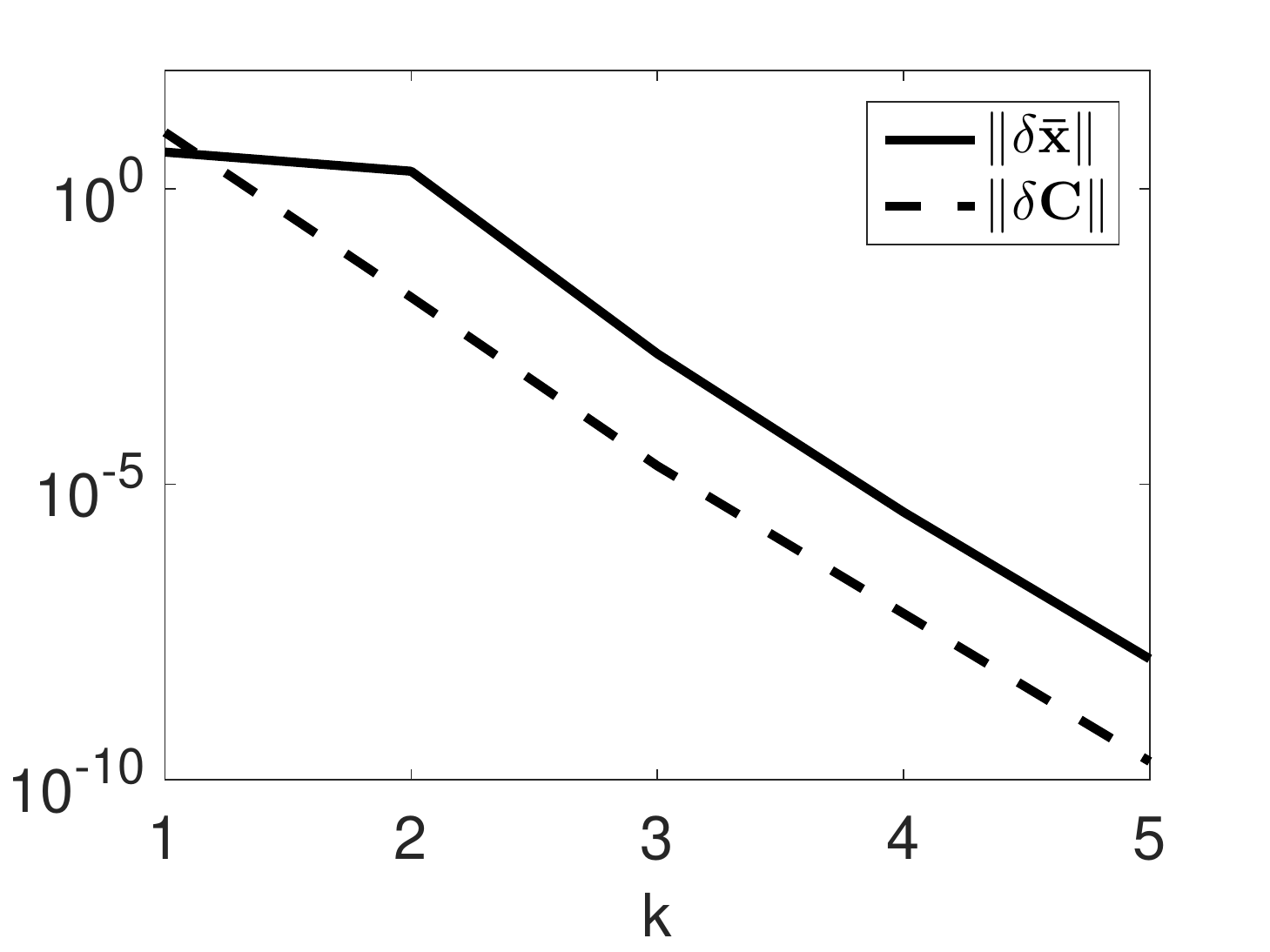} & \includegraphics[scale=0.45]{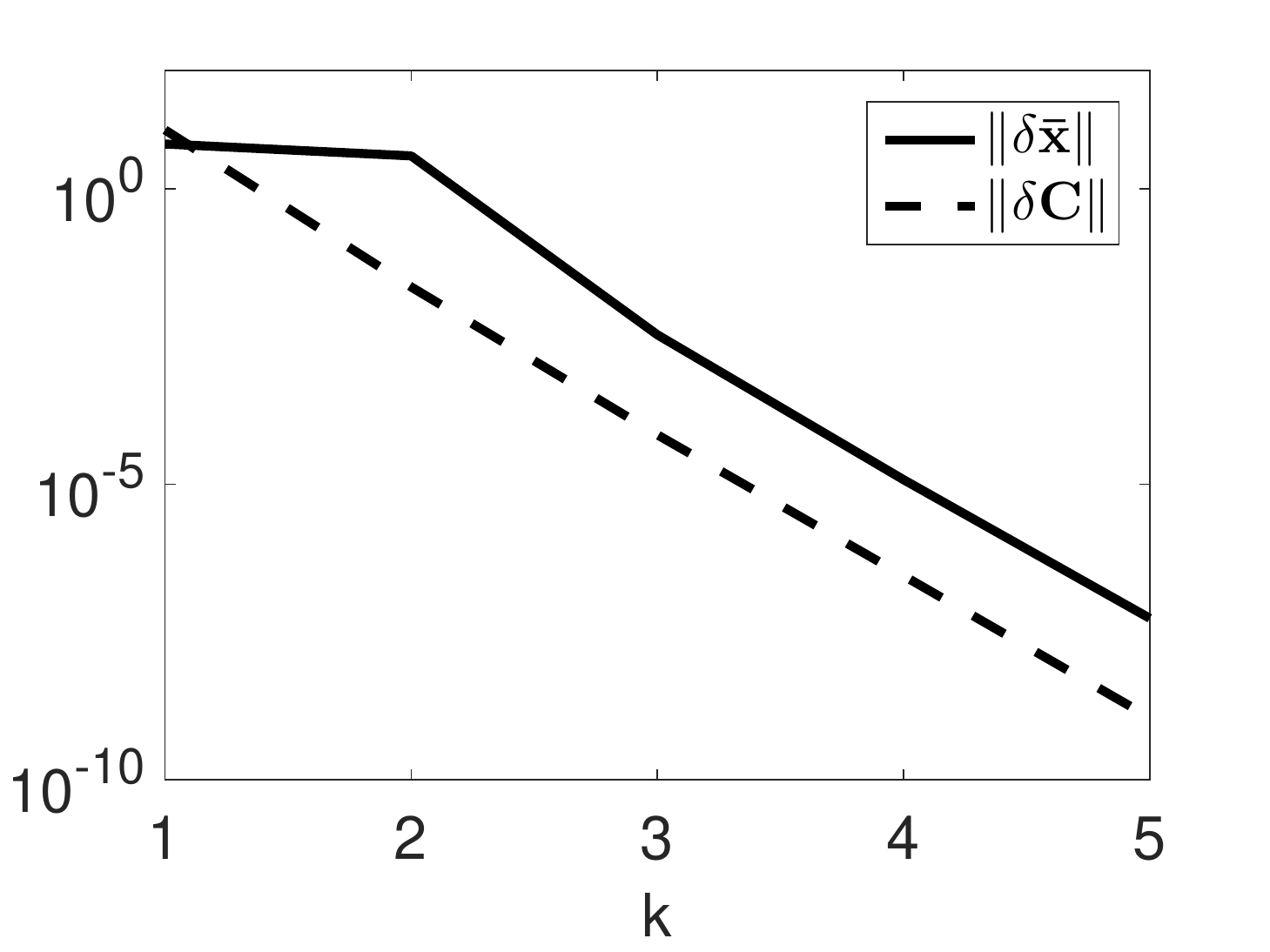}\\
(a) $L^2$ prior & (b) $H^1$-prior
\end{tabular}
\caption{The convergence of outer iterations of Algorithm \ref{alg:vb} for \texttt{phillips}.}\label{fig:outer}
\end{figure}

\begin{figure}[hbt!]
\centering
\begin{tabular}{cc}
\includegraphics[scale=0.45]{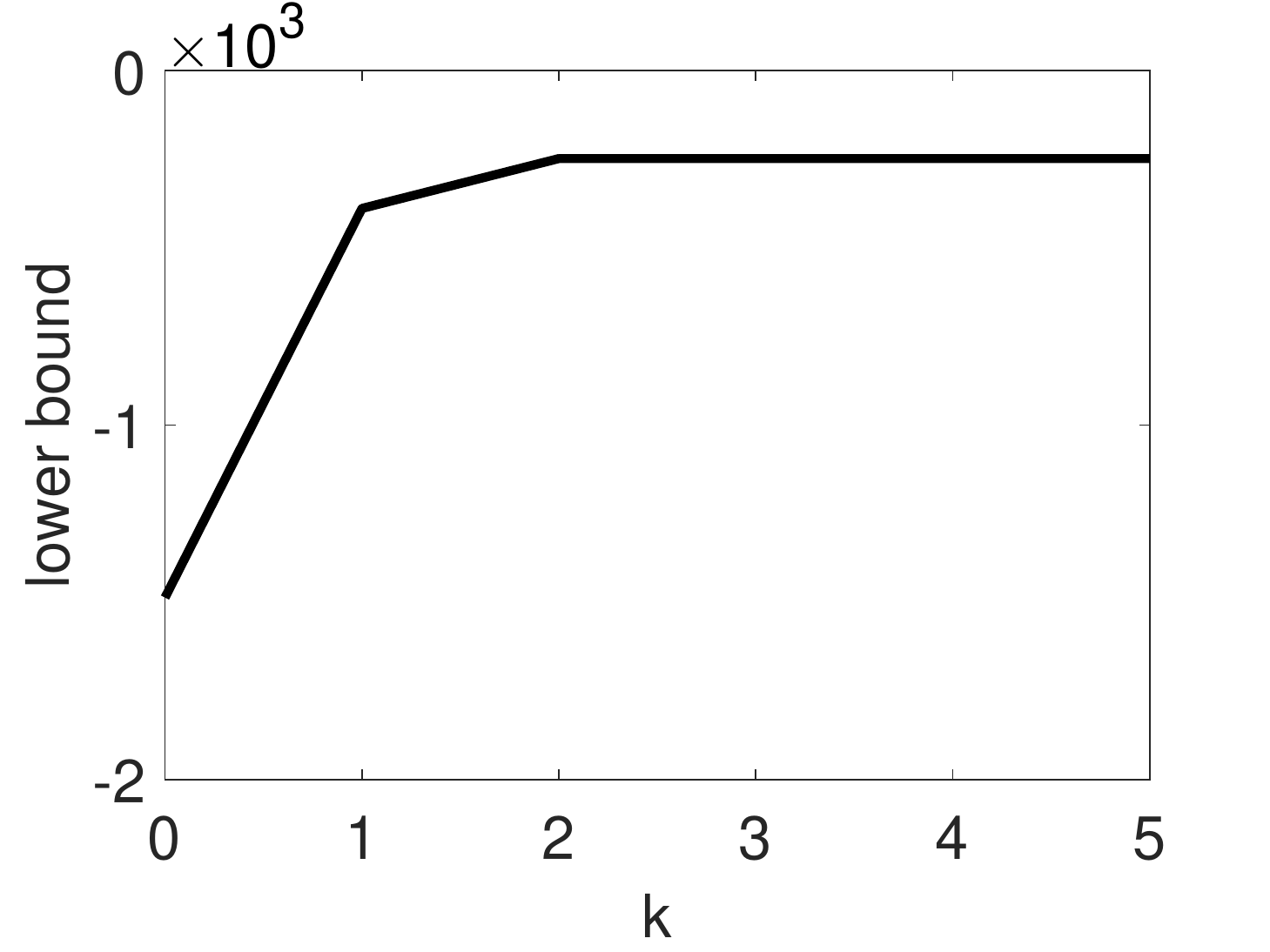} &	\includegraphics[scale=0.45]{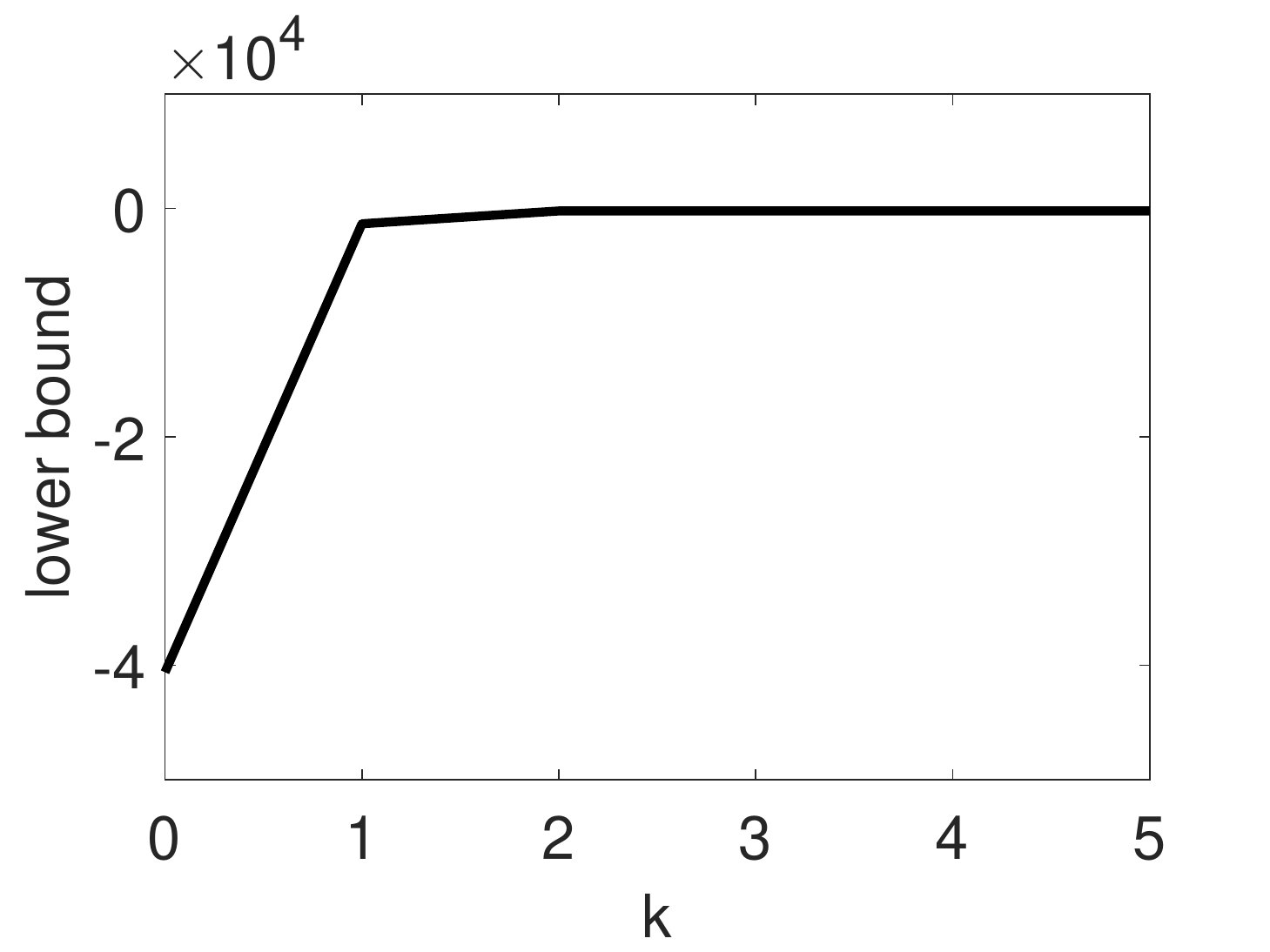}\\
(a) $L^2$-prior & (b) $H^1$-prior
\end{tabular}
\caption{The convergence of the lower bound $F(\mathbf{\bar x},\mathbf{C})$ for \texttt{phillips}.	\label{fig:1dconv}}
\end{figure}

To examine the convergence of outer iterations, we show the errors of the mean $\mathbf{\bar x}$ and
covariance $\mathbf{C}$ and the lower bound $F(\mathbf{\bar x},\mathbf{C})$ in Figs. \ref{fig:outer}
and \ref{fig:1dconv}, respectively. Algorithm \ref{alg:vb} is terminated when the change of the lower
bound falls below $10^{-10}$. For the $L^2$-prior,
Algorithm \ref{alg:vb} converges after $5$ iterations and the last increments $\delta\mathbf{\bar x}$ and
$\delta\mathbf{C}$ are of order $10^{-8}$ and $10^{-9}$, respectively. This observation holds also
for the $H^1$-prior, cf. Figs. \ref{fig:outer}(b) and \ref{fig:1dconv}(b). Thus, both inner and outer
iterations converge rapidly and steadily, and Algorithm \ref{alg:vb} is very efficient.

\subsection{Low-rank approximation of $\mathbf{A}$ and sparsity of $\mathbf{C}$}\label{ssec:low-rank}

The discussions in Section \ref{ssec:complexity} show that the structure on $\mathbf{A}$ and $\mathbf{C}$ can
be leveraged to reduce the complexity of Algorithm \ref{alg:vb}. Now we evaluate their
influence on the accuracy of the VGA.

First, we examine the influence of low-rank approximation to $\mathbf{A}$. Since the kernel function of
the example \texttt{phillips} is smooth, the inverse problem is mildly ill-posed and the singular values
$\sigma_k$ decay algebraically, cf. Fig. \ref{fig:low-rank}(a). A low-rank matrix $\mathbf{A}_r$
of rank $r\approx10$ can already approximate $\mathbf{A}$ well. To study its influence
on the VGA, we denote by $(\mathbf{\bar x}_r,\mathbf{C}_r)$ and $(\mathbf{\bar x}^*,\mathbf{C}^*)$ the VGA
for $\mathbf{A}_r$ and $\mathbf{A}$, respectively. The errors $e_{\mathbf{\bar x}} =
\|\mathbf{\bar x}_r-\mathbf{\bar x}^*\|$ and $e_{\mathbf{C}}=\|\mathbf{C}_r-\mathbf{C}^*\|$ for
different ranks $r$ are shown in Figs. \ref{fig:low-rank} (b) and (c) for
the $L^2$- and $H^1$-prior, respectively. Too small a rank $r$ of
the approximation $\mathbf{A}_r$ can lead to pronounced errors in both the mean
$\mathbf{\bar x}$ and the covariance $\mathbf{C}$, whereas for a rank of $r= 10$, the errors already fall below
one percent. Interestingly, the decay of the error $e_{\mathbf{\bar x}}$ is much faster than that
of the singular values $\sigma_k$, and the error $e_\mathbf{C}$ decays slower than $e_{\mathbf{\bar x}}$.
The fast decay of the errors $e_{\mathbf{\bar x}}$ and $e_\mathbf{C}$ indicates the robustness of the VGA,
which justifies using low-rank approximations in Algorithm \ref{alg:vb}.

\begin{figure}[h]
\centering
\setlength{\tabcolsep}{0mm}
\begin{tabular}{ccc}
\includegraphics[scale=0.35]{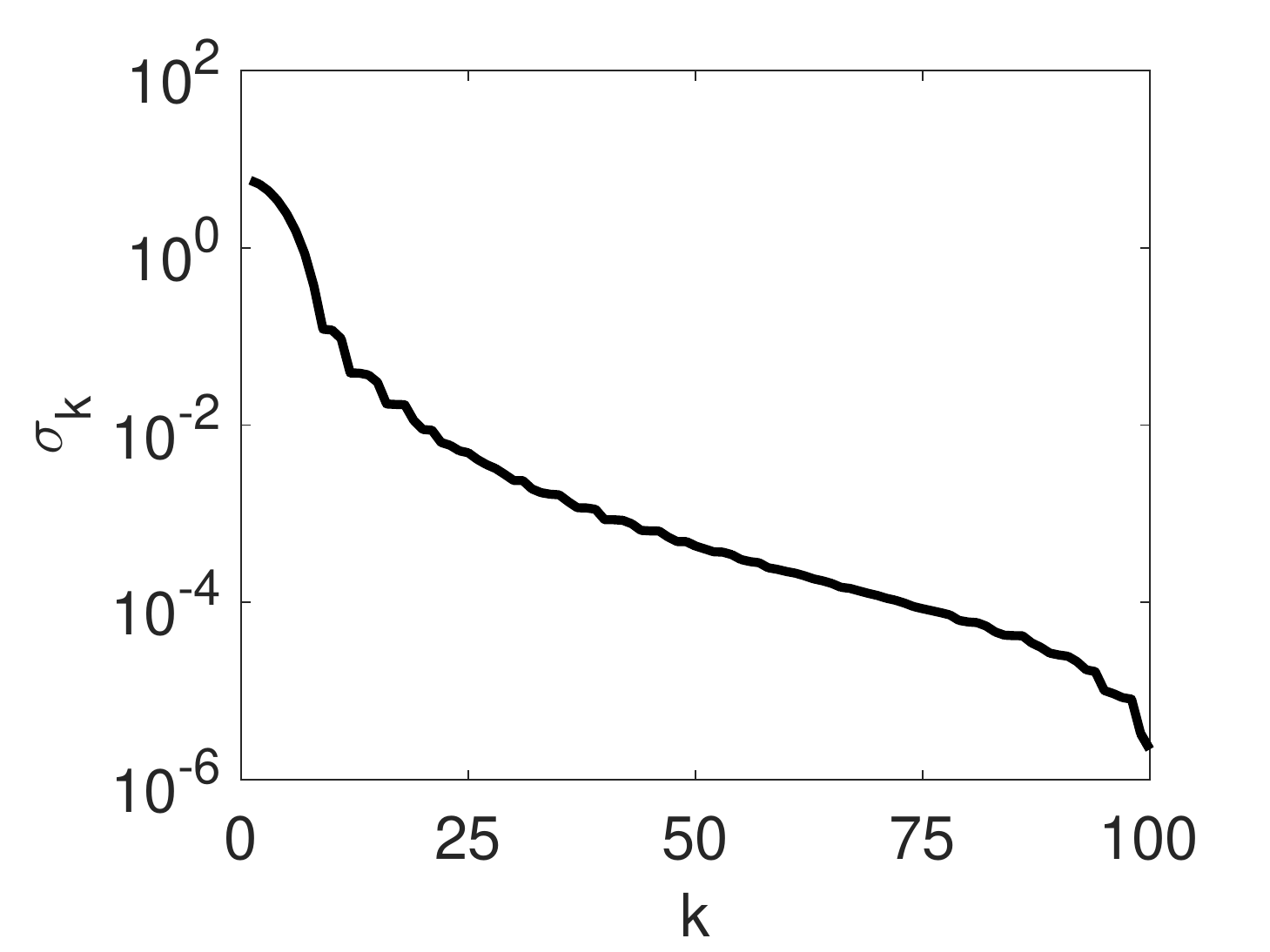}&\includegraphics[scale=0.35]{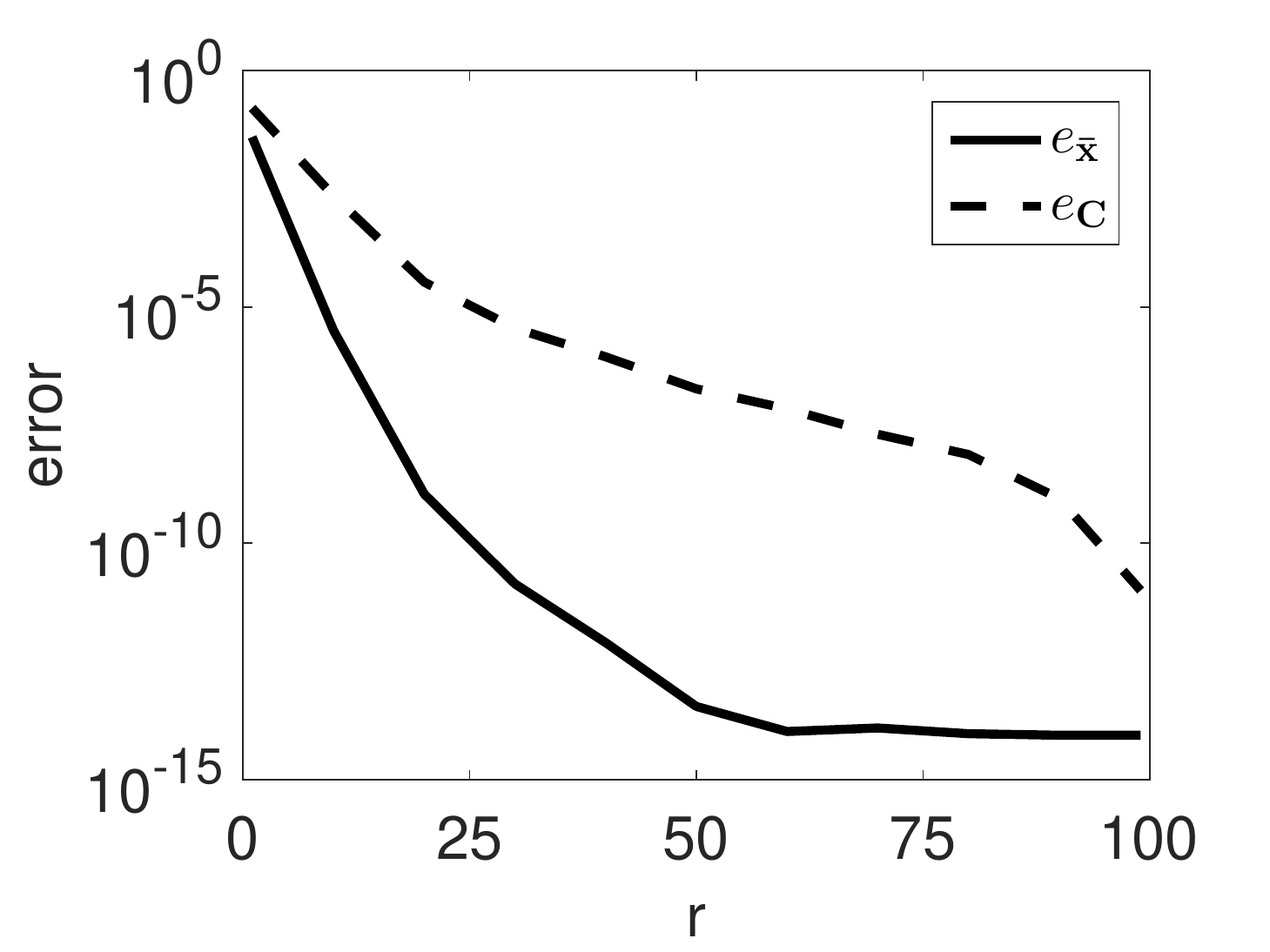}&\includegraphics[scale=0.35]{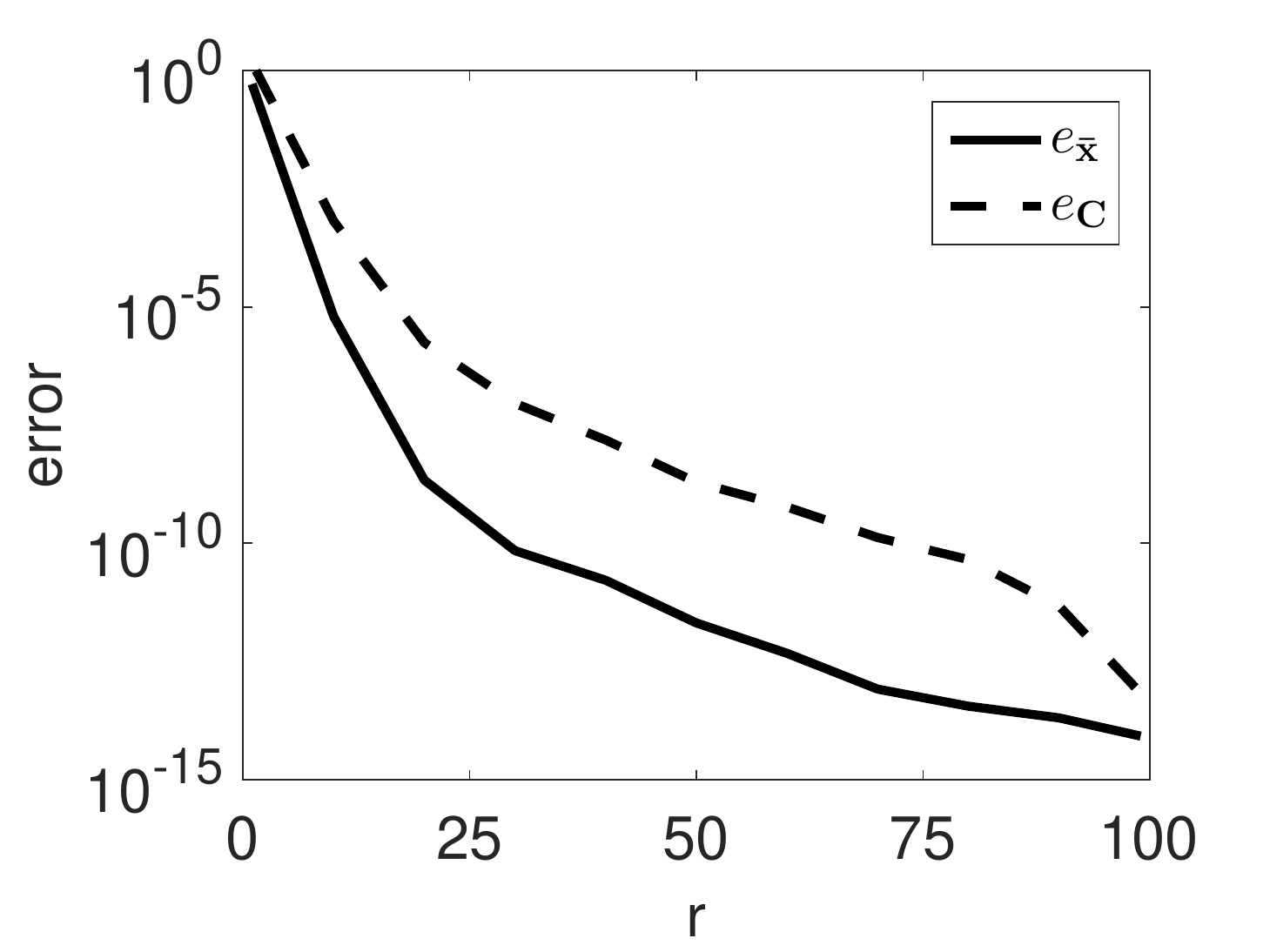}\\
(a) singular values $\sigma_k$& (b) $L^2$-prior & (c) $H^1$-prior
\end{tabular}
\caption{(a) singular values and (b)--(c): the errors of the mean and covariance for \texttt{phillips}. \label{fig:low-rank}}
\end{figure}

Next we examine the influence of the sparsity assumption on the covariance $\mathbf{C}$, which is used to reduce
the complexity of Algorithm \ref{alg:vb}. Due to the coupling between $\mathbf{\bar x}$ and $\mathbf{C}$,
cf. \eqref{eqn:barx}--\eqref{eqn:C}, the sparsity assumption on $\mathbf{C}$ affects the accuracy of both
$\mathbf{\bar x}$ and $\mathbf{C}$. To illustrate this, we take different sparsity levels $s$ on $\mathbf{C}$
in Algorithm \ref{alg:vb}, i.e., at most $s$ nonzero entries around the
diagonal of $\mathbf{C}$. Surprisingly, a diagonal $\mathbf{C}$ already gives an acceptable approximation measured by
the errors $e_{\mathbf{\bar x}}=\|\mathbf{\bar x}_s-\mathbf{\bar x}^*\|_2$ and $e_\mathbf{C}=\|\mathbf{C}_s-\mathbf{C}^*\|_2$,
where $(\mathbf{\bar x}_s,\mathbf{C}_s)$ is the VGA with a sparsity level $s$. The
errors $e_{\mathbf{\bar x}}$ and $e_\mathbf{C}$ decrease with the sparsity level $s$, cf. Table \ref{tab:err-sparsity}.
Thus the sparsity assumption on $\mathbf{C}$ can reduce significantly the complexity while retaining the accuracy.

\begin{table}[htbp]
 \caption{The errors $e_{\mathbf{\bar x}}$ and $e_\mathbf{C}$
 v.s. the sparsity level $s$ of $\mathbf{C}$ for
 \texttt{phillips}. \label{tab:err-sparsity}}
 \label{tab:C_sparse}
 \centering
 \begin{tabular}{c|cc|cc}
  \hline
   prior & \multicolumn{2}{|c|}{$L^2$ prior} & \multicolumn{2}{|c}{ $H^1$ prior}\\
  $s$    & $e_{\mathbf{\bar x}}$ & $e_{\mathbf{C}}$ & $e_{\mathbf{\bar x}}$ & $e_{\mathbf{C}}$\\
  \midrule
  $1$  & 6.38e-2 & 9.20e-2 & 1.92e-2 & 7.06e-2 \\
  $3$  & 5.62e-2 & 8.10e-2 & 1.27e-2 & 5.42e-2\\
  $5$  & 4.88e-2 & 7.02e-2 & 1.00e-2 & 4.29e-2\\
  \hline
 \end{tabular}
\end{table}

\subsection{Hierarchical parameter choice}\label{ssec:param}

Now we examine the convergence of Algorithm \ref{alg:hyper} for choosing the parameter $\alpha$ in
the prior $p(\mathbf{x})$. By Theorem \ref{thm:mono}, the sequence $\{\alpha^k\}$ generated by Algorithm
\ref{alg:hyper} is monotone. We illustrate this by two initial guesses, i.e., $\alpha^1=0.1$ and
$\alpha^1=10$. Both sequences of iterates generated by Algorithm \ref{alg:hyper} converge monotonically to the limit
$\alpha^* = 0.7778$, and the convergence of Algorithm \ref{alg:hyper} is fairly steady, cf. Fig.
\ref{fig:conv_a}(a). Further, Algorithm \ref{alg:hyper} indeed maximizes the joint lower bound
\eqref{eq:hyperbd} with its maximum attained at $\alpha^*=0.7778$, cf. Fig. \ref{fig:conv_a}(b). Though
not shown, the lower bound $F_\alpha(\mathbf{\bar x},\mathbf{c}|\alpha)$ is also increasing during the
iteration. Thus, the hierarchical approach is
indeed performing model selection by maximizing ELBO.

\begin{figure}[hbt!]
\centering
\begin{tabular}{cc}
\includegraphics[scale=0.45]{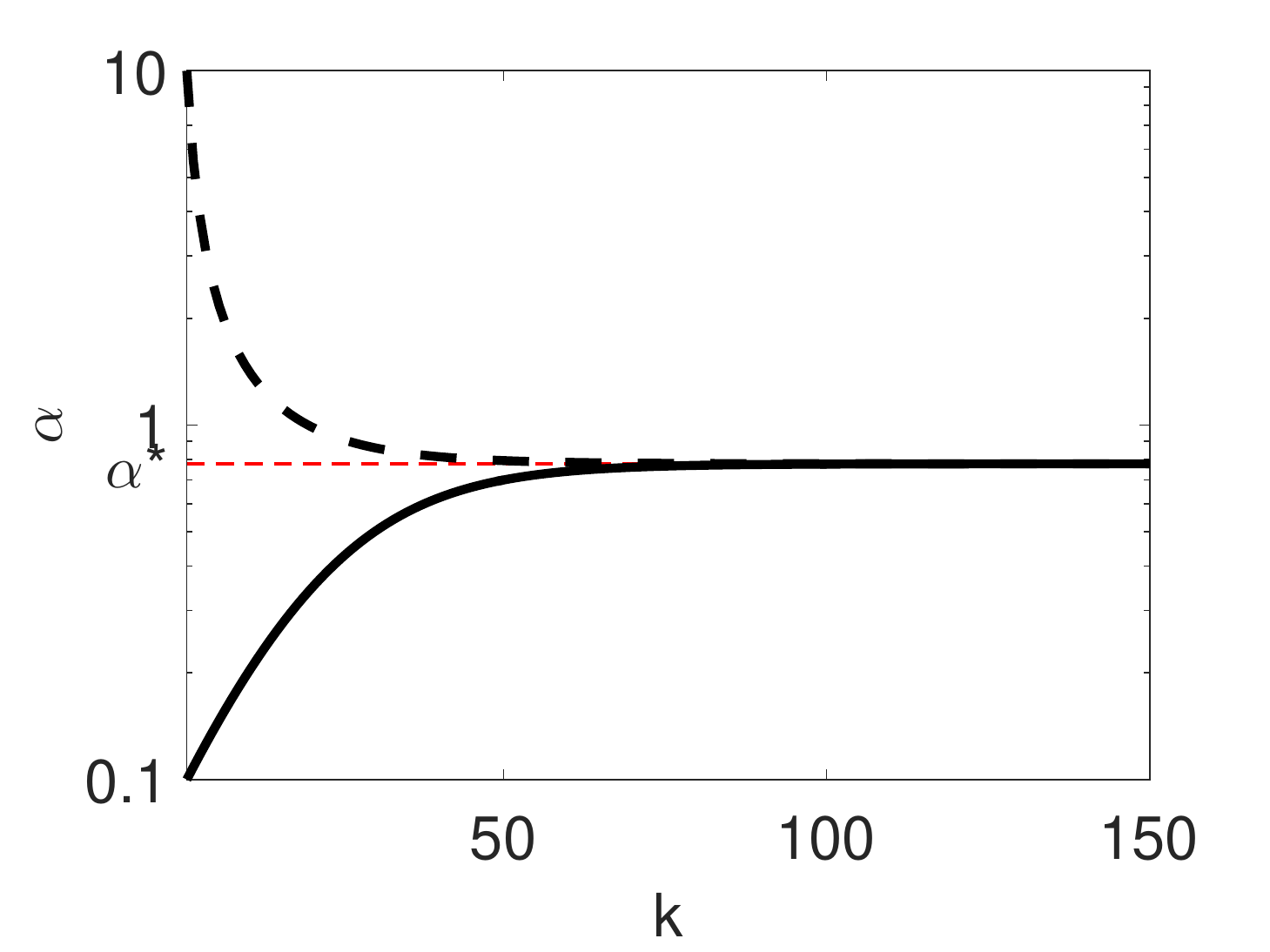} & \includegraphics[scale=0.45]{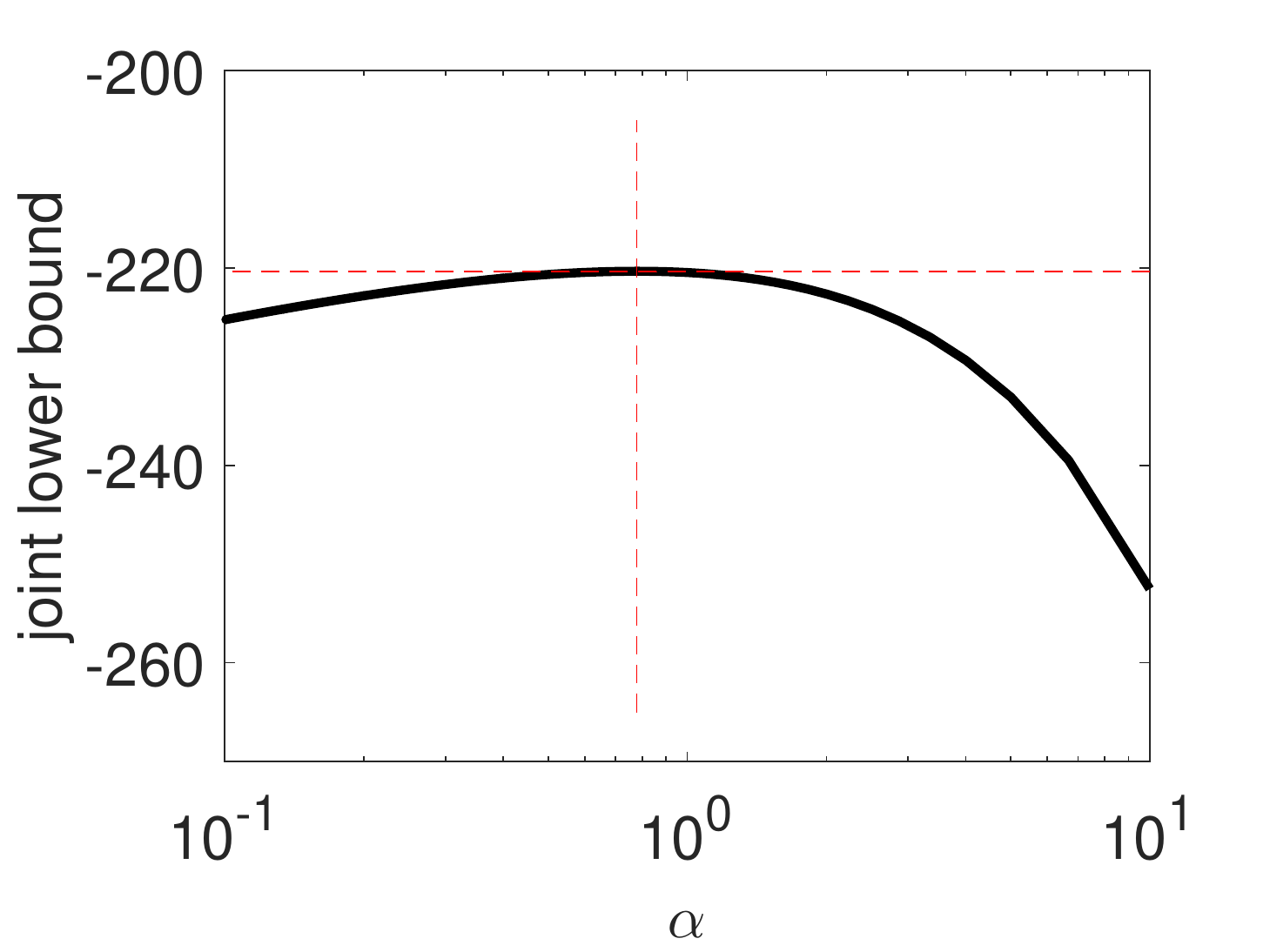}\\
(a) convergence of $\alpha$  & (b) joint lower bound
\end{tabular}
\caption{(a)The convergence of Algorithm \ref{alg:hyper} initialized with $0.1$ and $10$, both
convergent to $\alpha^*=0.7778$ (b)
the joint lower bound versus $\alpha$, for \texttt{phillips}
with $L^2$-prior.\label{fig:conv_a}}
\end{figure}

To illustrate the quality of the automatically chosen parameter $\alpha$, we take
six realizations of the Poisson data $\mathbf{y}$ and compare the mean $\mathbf{\bar x}$ of the
VGA with the optimal regularized solutions, where $\alpha$
is tuned so that the error is smallest (and thus it is infeasible in practice).
The means $\mathbf{\bar x}$ by Algorithm \ref{alg:hyper} are comparable with the optimal ones,
cf. Fig. \ref{fig:hypertest}, and thus the hierarchical approach can yield reasonable approximations.
The parameter $\alpha$ by the hierarchical approach is slightly smaller than the optimal one, cf. Table
\ref{tab:alpha}, and hence the corresponding
reconstruction tends to be slightly more oscillatory than the optimal one. The
value of the parameter $\alpha$ by the hierarchical approach is
relatively independent of the realization, whose precise mechanism is to be ascertained.

\begin{figure}[hbt!]
\centering
\setlength{\tabcolsep}{0mm}
\begin{tabular}{ccc}
\includegraphics[scale=0.35]{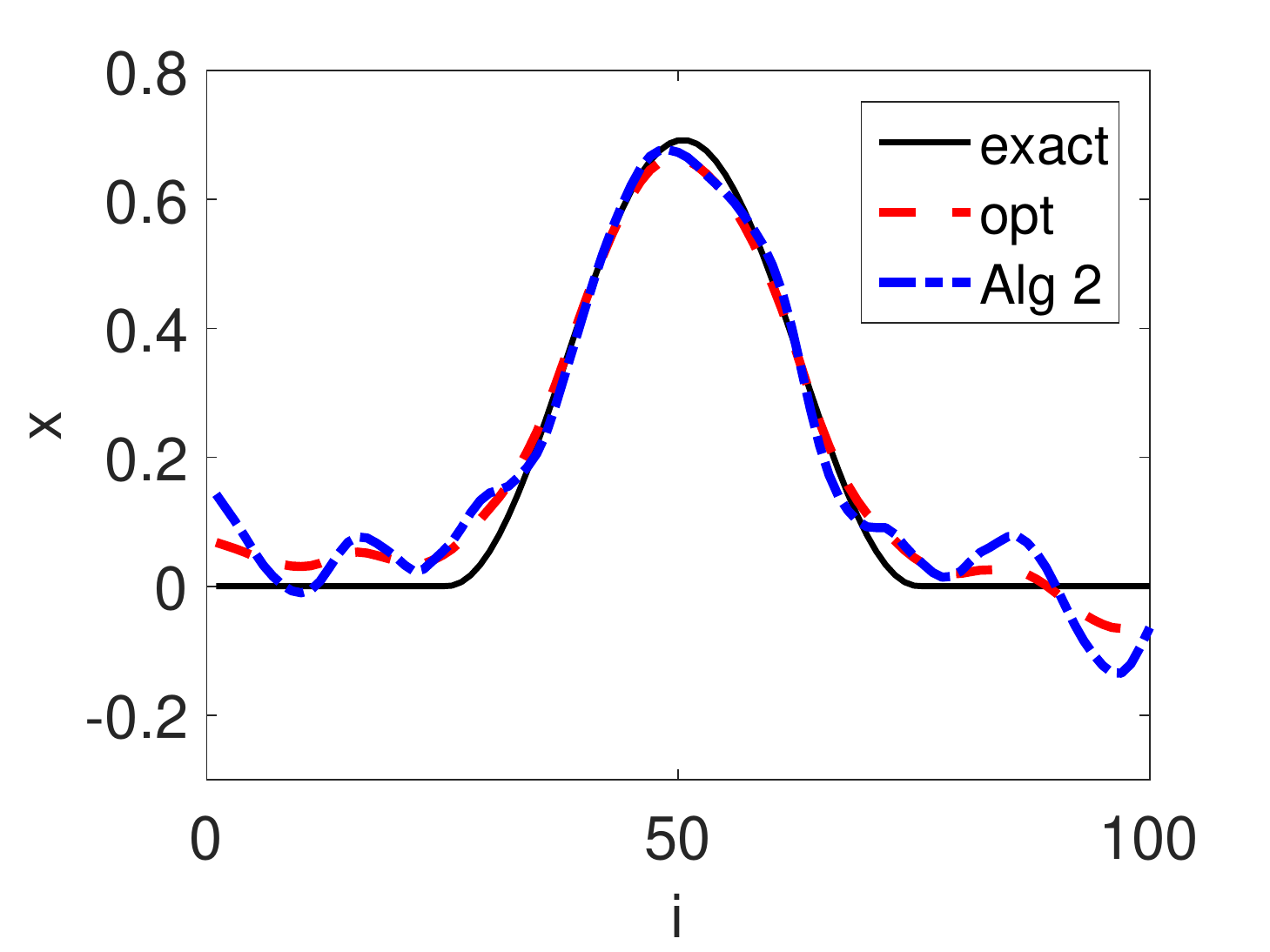} & \includegraphics[scale=0.35]{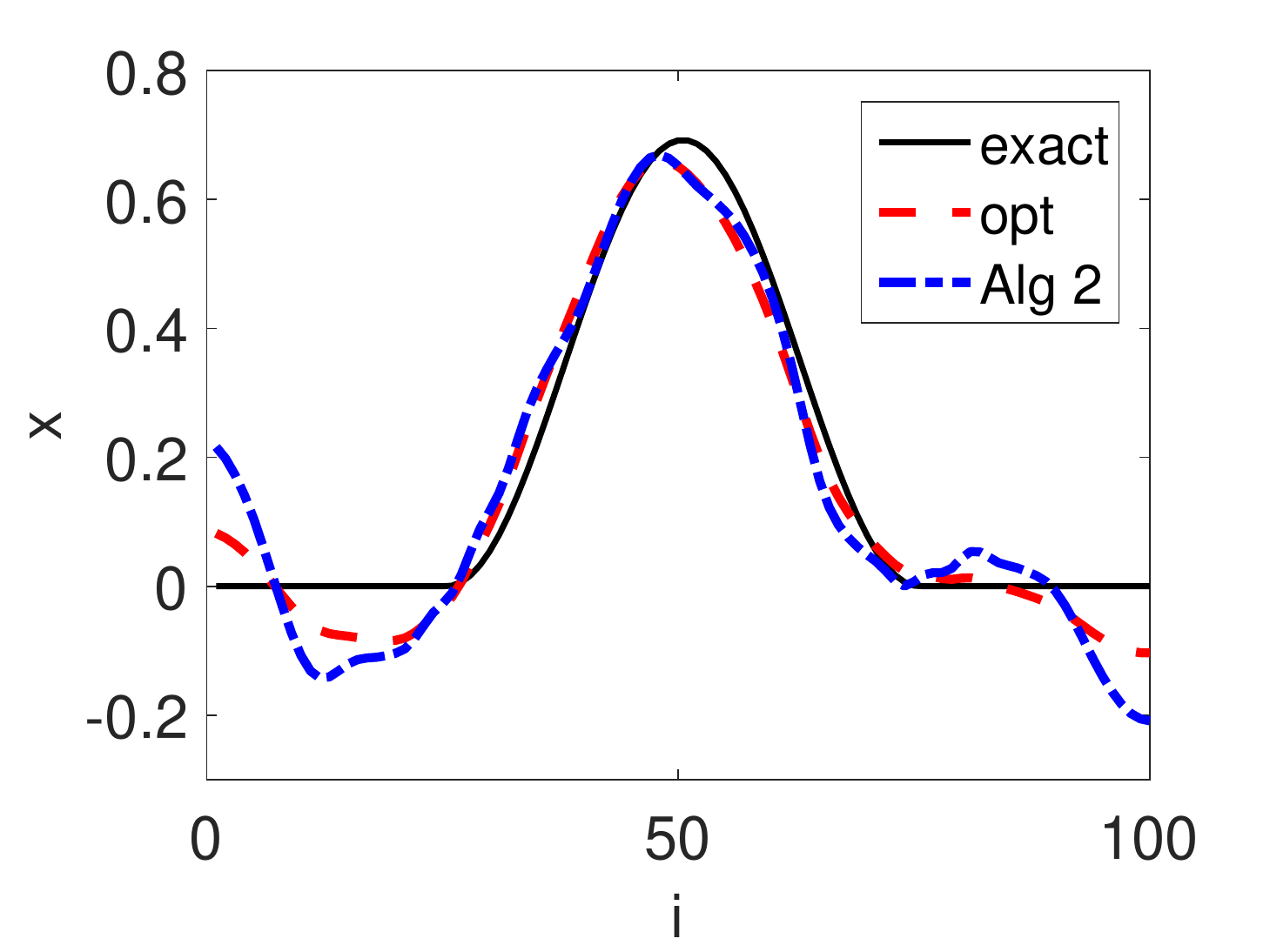} & \includegraphics[scale=0.35]{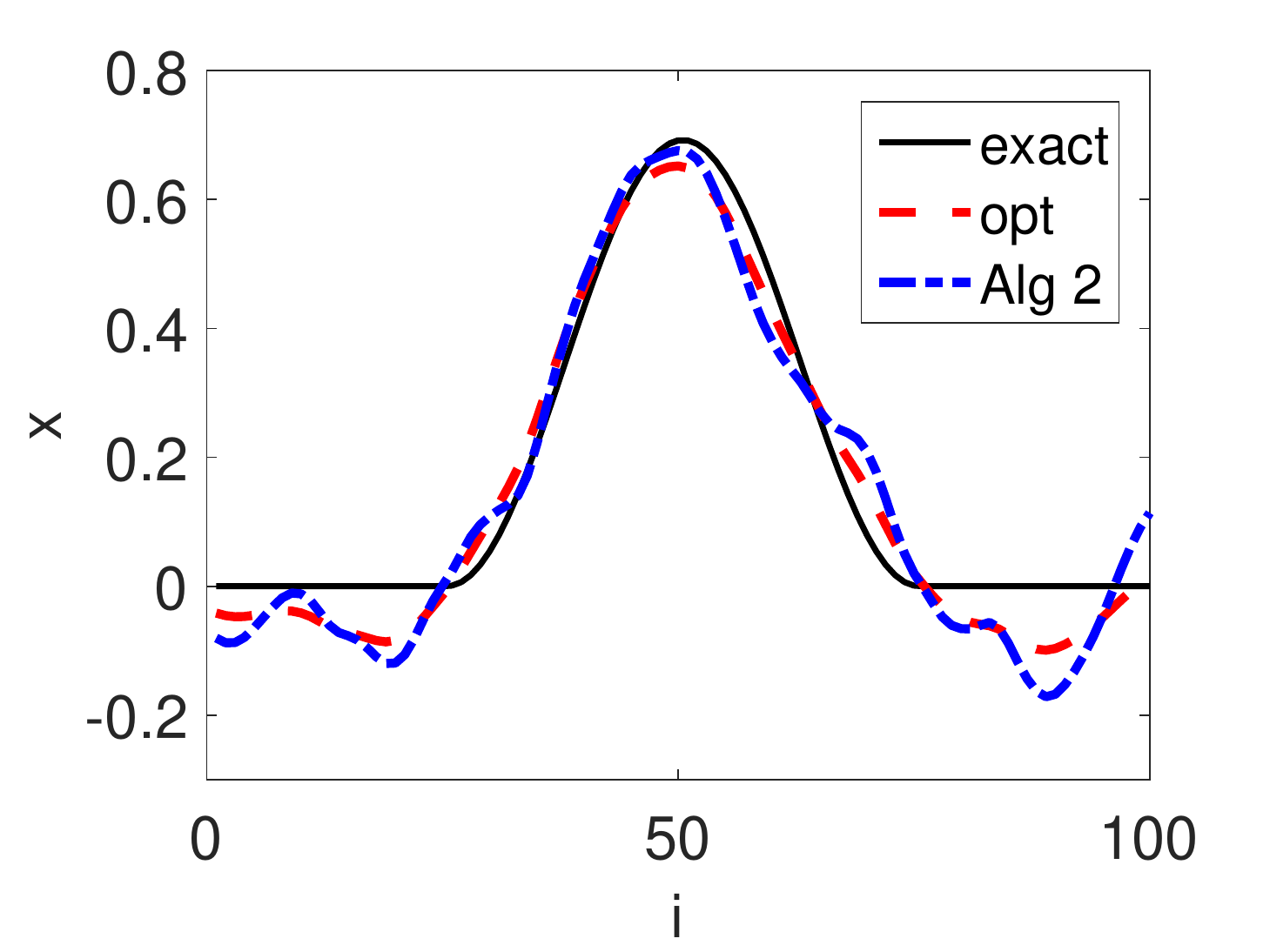}\\
\includegraphics[scale=0.35]{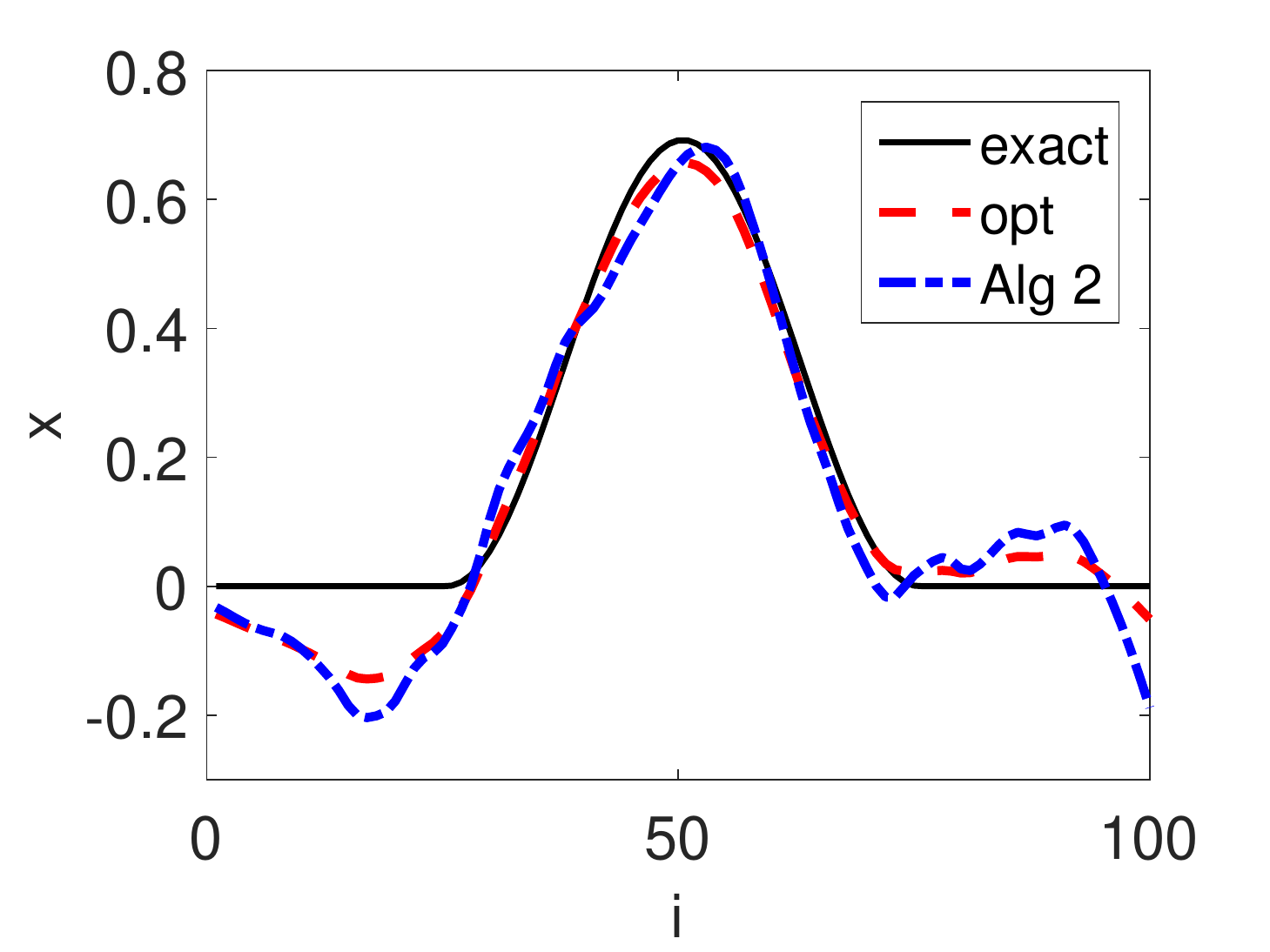} & \includegraphics[scale=0.35]{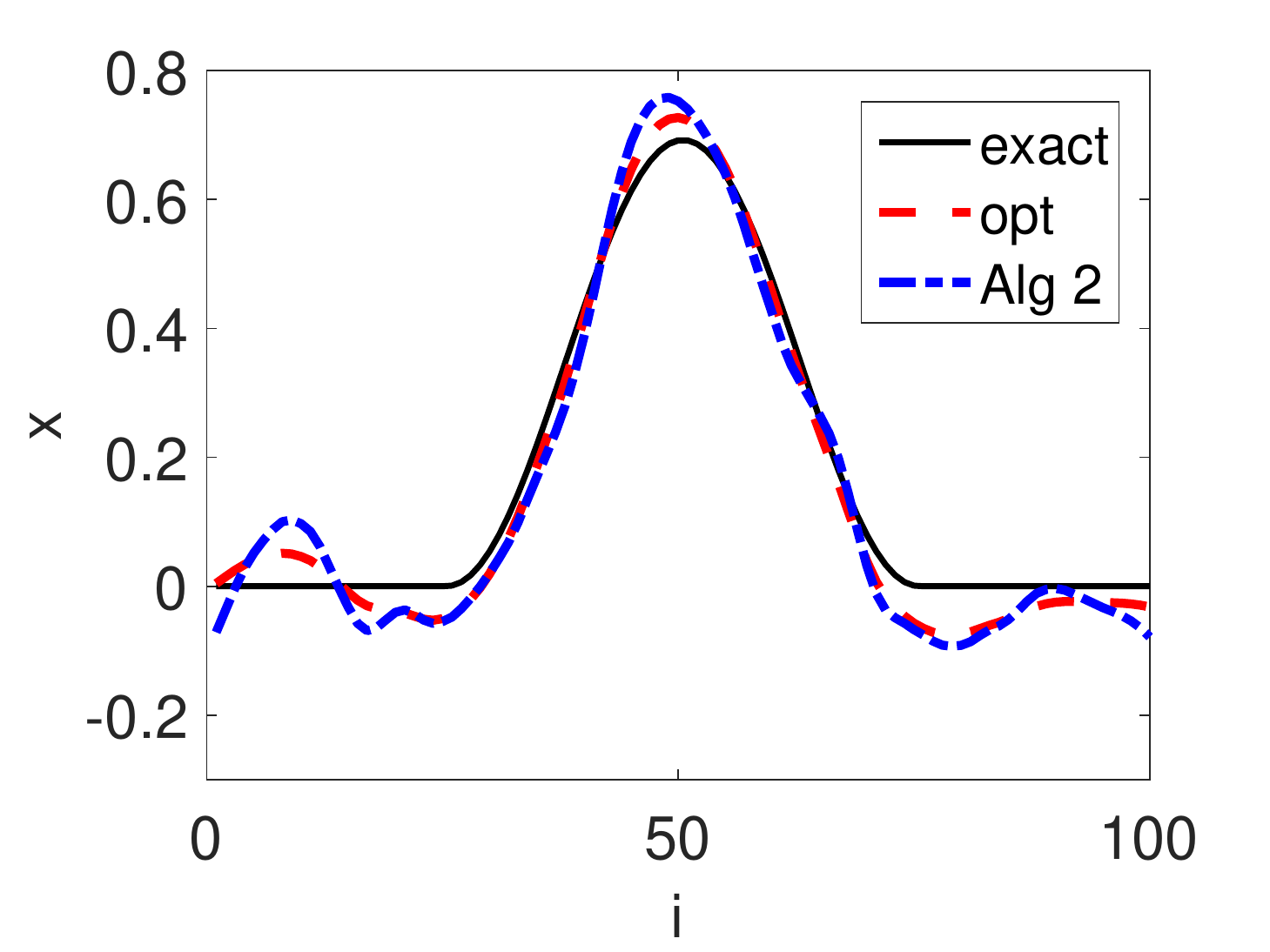} & \includegraphics[scale=0.35]{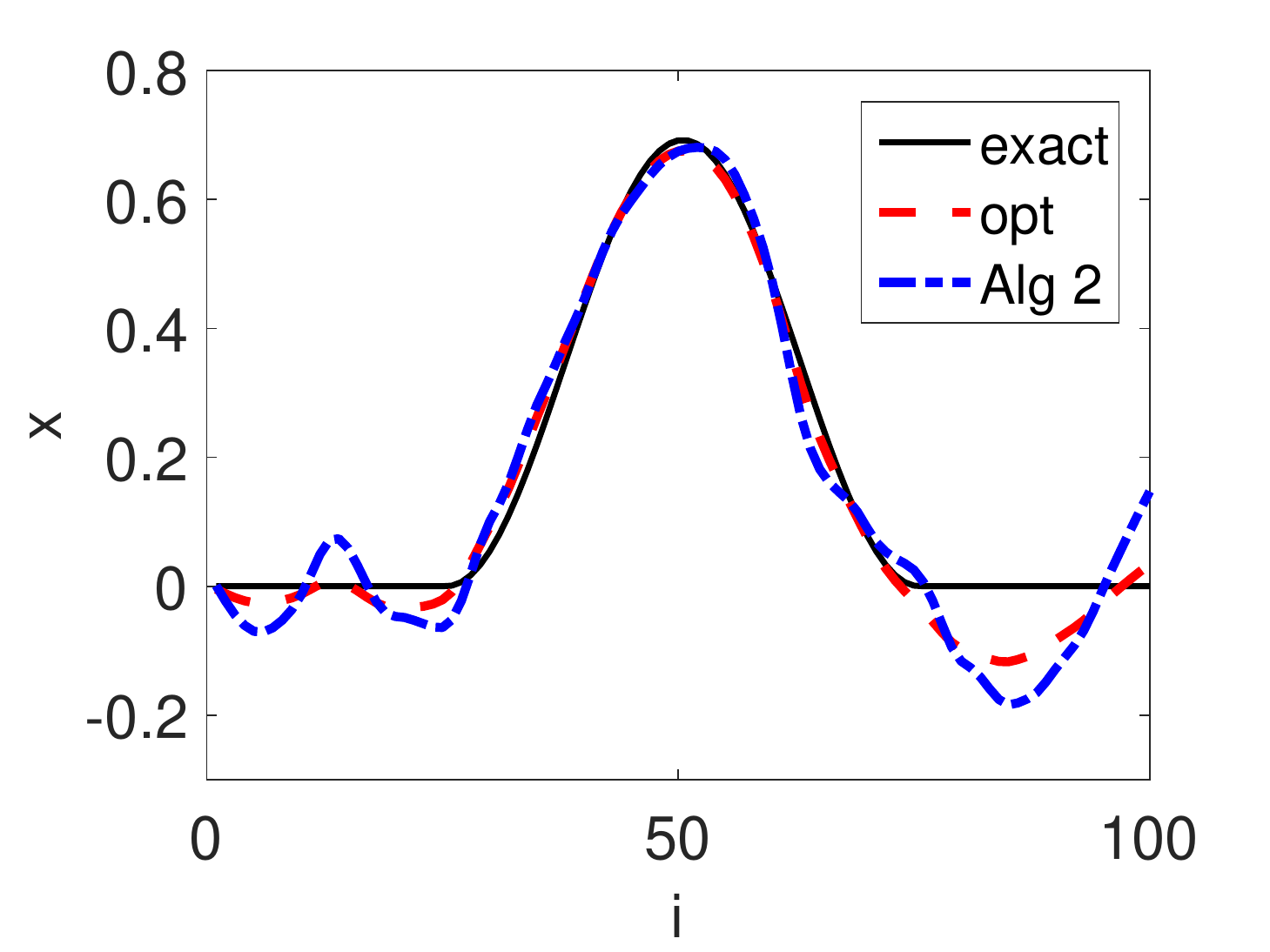}\\
\end{tabular}
\caption{The mean $\mathbf{\bar x}$ of the Gaussian approximation by Algorithm \ref{alg:hyper} (Alg2) and the ``optimal" solution (opt) for
6 realizations of Poisson data for \texttt{phillips} with the $L^2$-prior.\label{fig:hypertest}}
\end{figure}

\begin{table}[hbt!]
  \centering
  \caption{The values of the hyperparameter $\alpha$ for the results in Fig. \ref{fig:hypertest}.\label{tab:alpha}}
  \begin{tabular}{c|cccccc}
  \hline
  case & 1 & 2 & 3 & 4 & 5 & 6\\
  \hline
  opt   & 2.64 & 3.35 & 2.59 & 1.35 & 9.31 & 4.04\\
  Alg 2 & 0.78 & 0.76 & 0.76 & 0.77 & 0.73 & 0.74\\
  \hline
  \end{tabular}
\end{table}

\subsection{VGA versus MCMC}\label{ssec:mcmc}

Despite the widespread use of variational type techniques in practice, the accuracy of the approximations
is rarely theoretically studied. This has long been a challenging issue for approximate Bayesian inference,
including the VGA. In this part, we conduct an experiment to numerically validate the VGA against the
results by Markov chain Monte Carlo (MCMC). To this end, we employ the standard Metropolis-Hastings algorithm,
with the Gaussian approximation from the VGA as the proposal distribution (i.e., independence sampler). In other
words, we correct the samples drawn from VGA by a Metropolis-Hastings step.  The length of the MCMC chain is $2\times 10^5$, and
the last $1\times 10^5$ samples are used for computing the summarizing statistics. The acceptance rate
in the Metropolis-Hastings algorithm is $96.06\%$. This might be attributed to the fact that the VGA approximates
the posterior distribution fairly accurately, and thus nearly all the proposals are accepted. The numerical results are
presented in Fig. \ref{fig:HPD}, where the mean and the $90\%$ highest posterior density (HPD) credible set are
shown. It is observed that the mean and HPD regions by MCMC and VGA are very close to each other, cf.
Figs. \ref{fig:HPD} and \ref{fig:mean_comp}, thereby validating the accuracy of the VGA. The $\ell^2$ error
between the mean by MCMC and GVA is $9.80\times 10^{-3}$, and the error between corresponding
covariance in spectral norm is $6.40\times 10^{-3}$. Just as expected, graphically the means and covariances
are indistinguishable, cf. Fig. \ref{fig:mean_comp}.

\begin{figure}[hbt!]
\centering
\begin{tabular}{cc}
\includegraphics[scale=0.35]{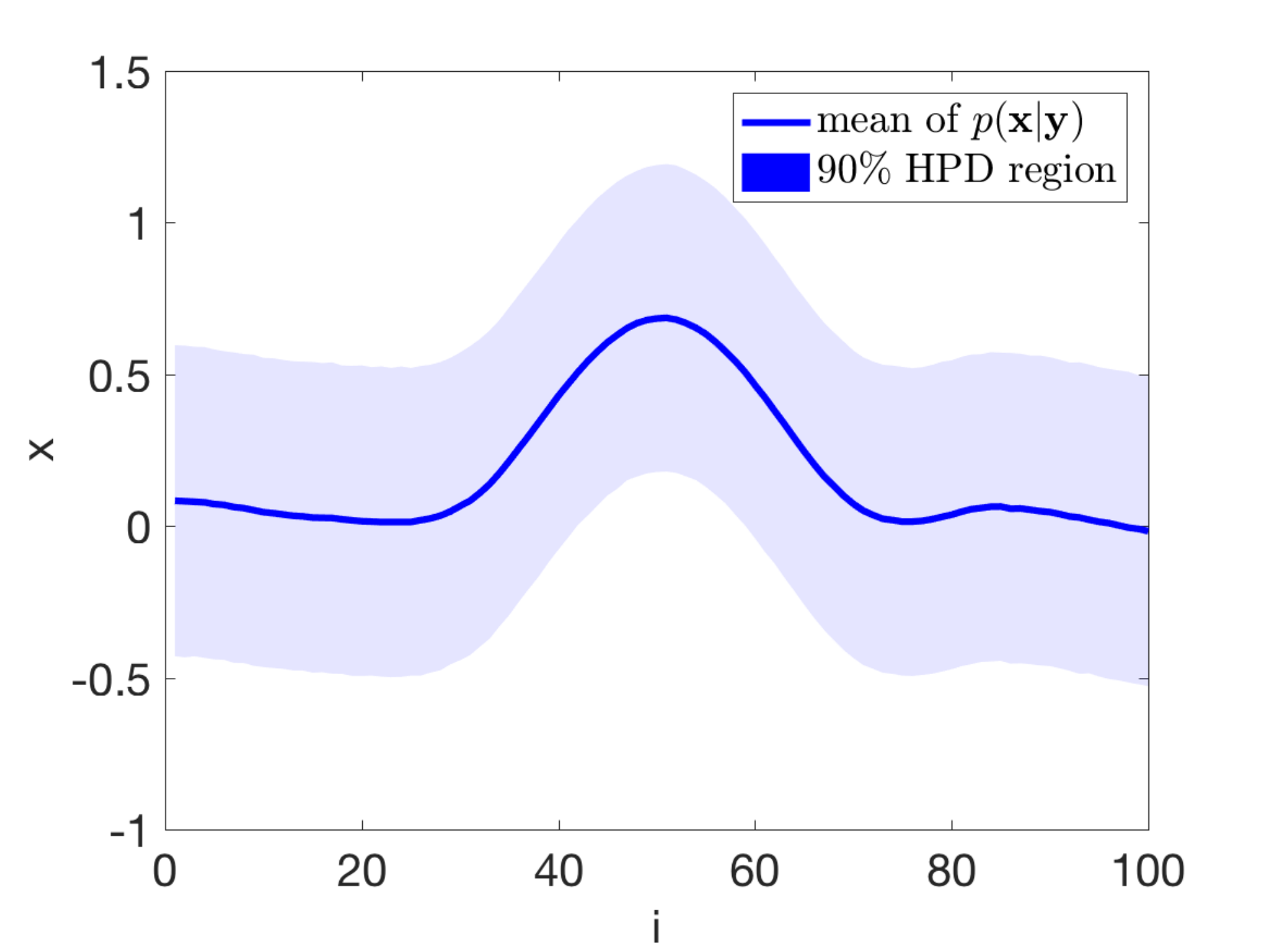} & \includegraphics[scale=0.35]{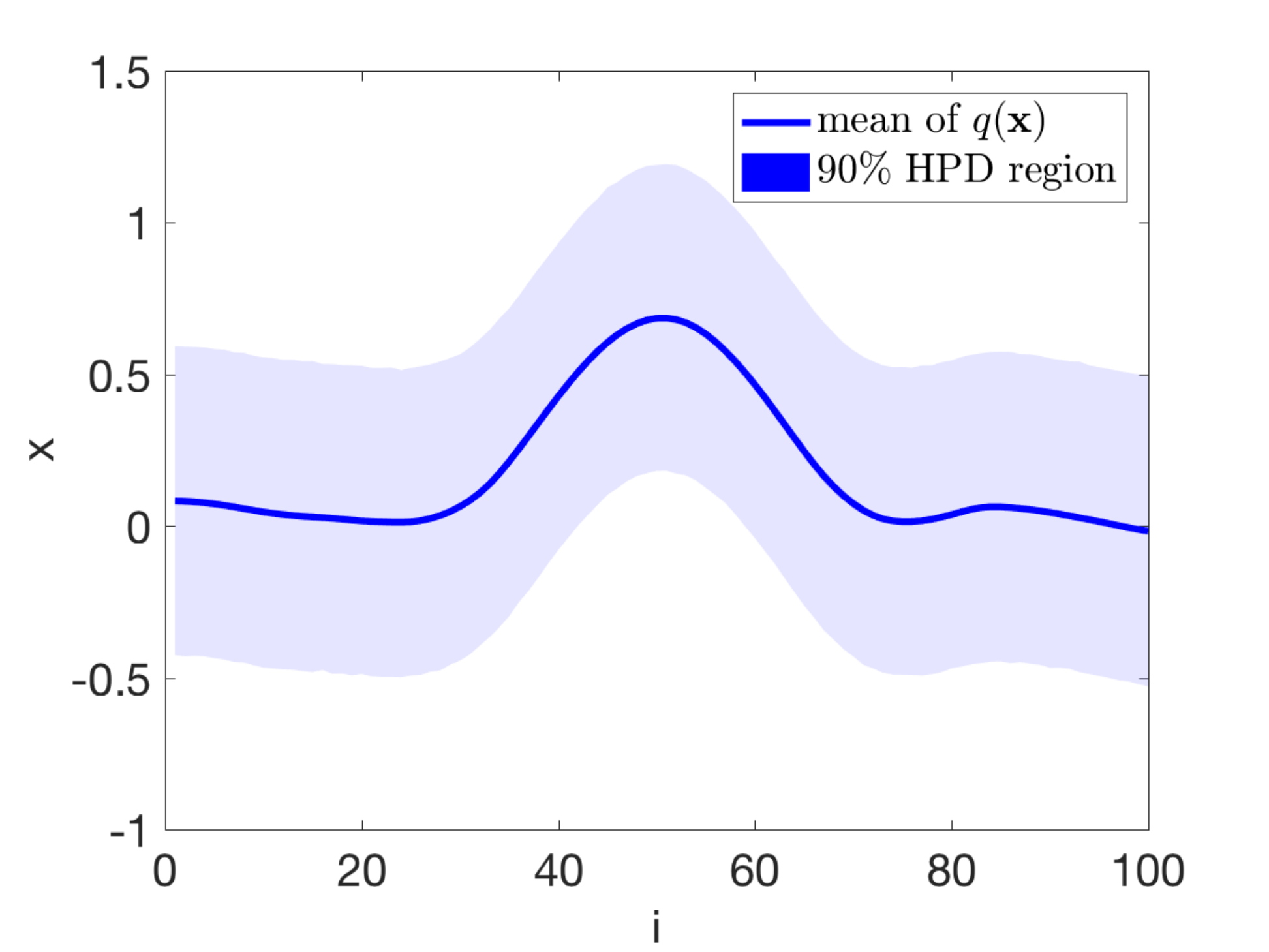}\\
(a) MCMC & (b) VGA
\end{tabular}
\caption{The mean and $90\%$ HPD by (a) MCMC and (b) VGA for \texttt{phillips} with $\mathbf{C}_0=1.00\times 10^{-1}\mathbf{\bar C}_0$.\label{fig:HPD}}
\end{figure}

\begin{figure}[hbt!]
\centering
\begin{tabular}{ccc}
\includegraphics[width=0.33\textwidth]{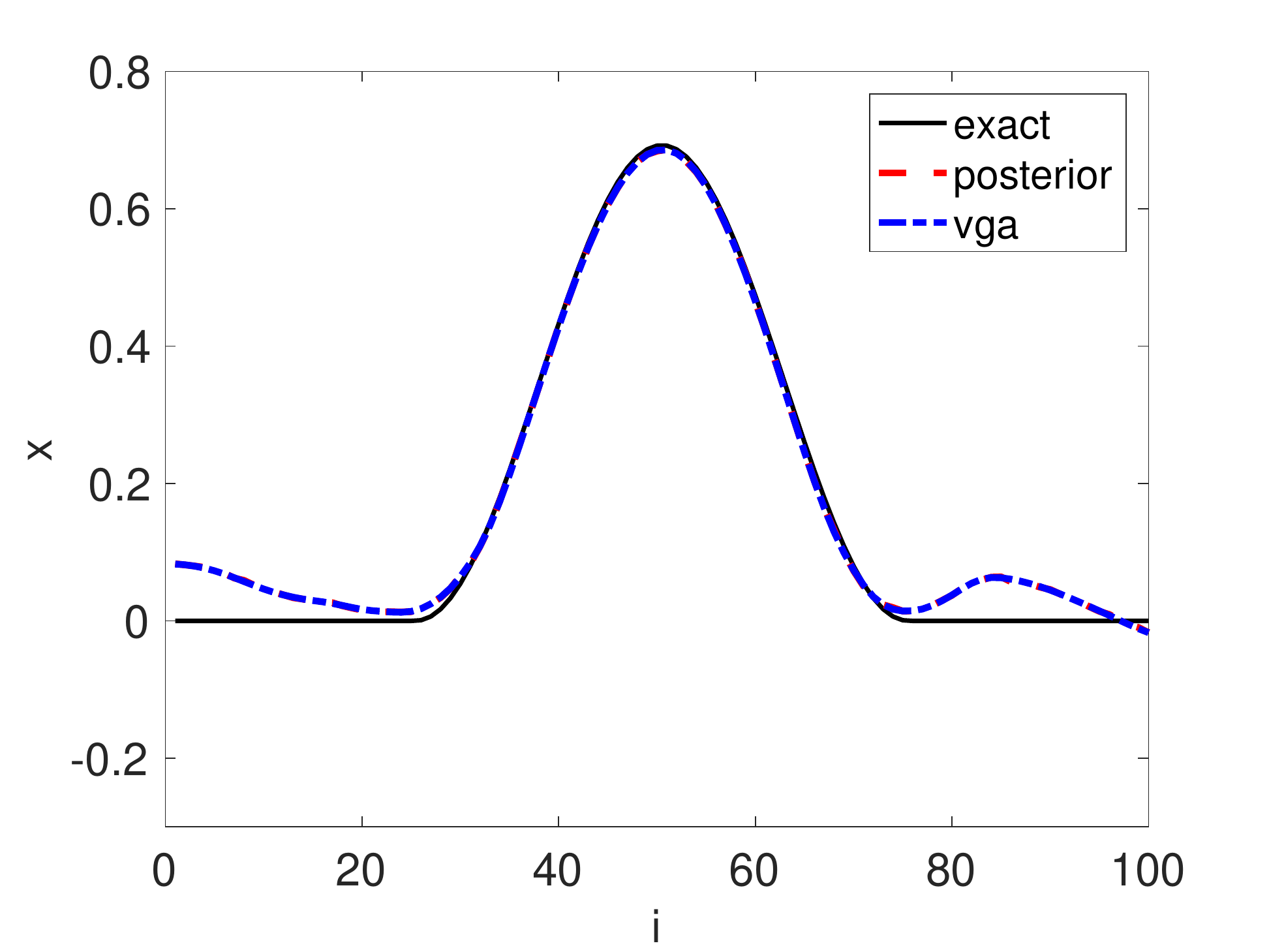} & \includegraphics[width=0.33\textwidth]{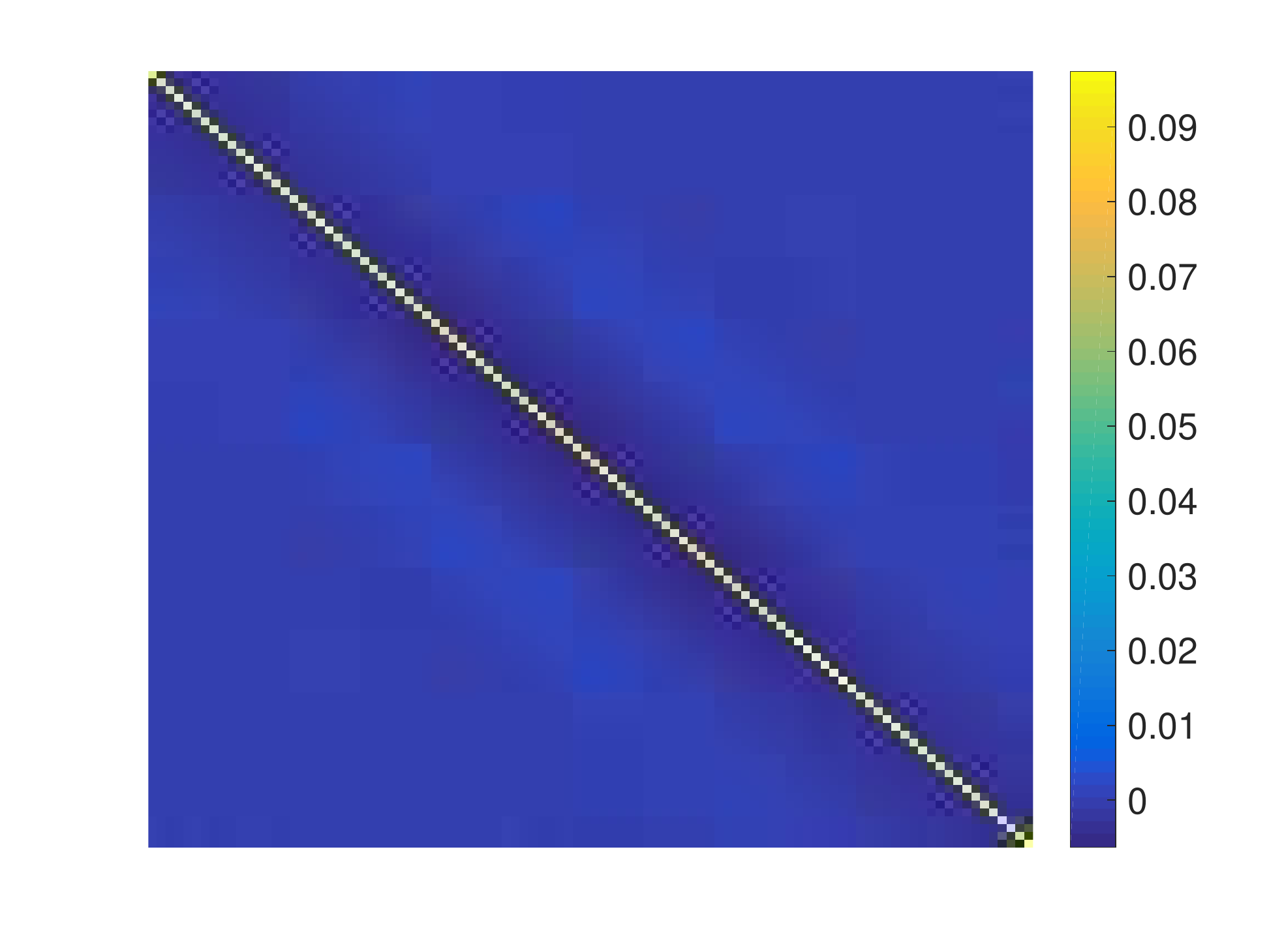} & \includegraphics[width=0.33\textwidth]{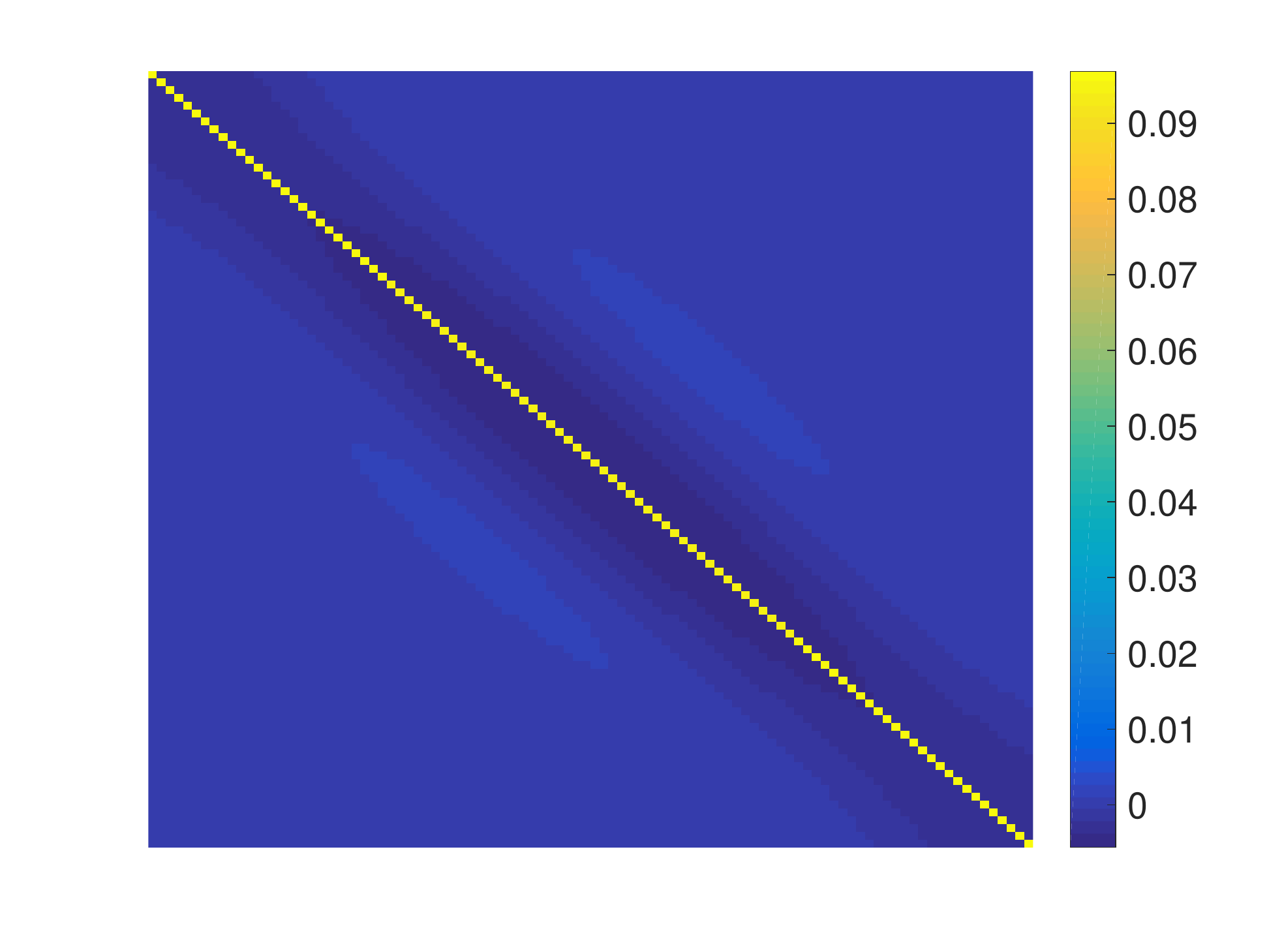}\\
(a) mean  & (b) MCMC & (c) VGA
\end{tabular}
\caption{(a) The mean by MCMC and VGA versus the exact solution, and the covariance by (b) MCMC and (c) VGA
for \texttt{phillips} with $\mathbf{C}_0=1.00\times 10^{-1}\mathbf{\bar C}_0$.\label{fig:mean_comp}}
\end{figure}

\subsection{Numerical reconstructions}\label{ssec:recon}

Last, we present VGAs for one- and two-dimensional examples. The numerical results for
the following four 1D examples, i.e., \texttt{phillips}, \texttt{foxgood}, \texttt{gravity} and \texttt{heat}, for both
$L^2$- and $H^1$-priors, are presented in Figs. \ref{fig:ph_L2}-\ref{fig:he_L2}. For the example \texttt{phillips}
with either prior, the mean $\mathbf{\bar{x}}$ by Algorithm \ref{alg:vb} agrees very
well with the true solution $\mathbf{x}^\dagger$. However, near the boundary, the mean $\mathbf{\bar x}$ is less accurate.
This might be attributed to the fact that in these regions, the Poisson count is relatively small, and it may be
insufficient for an accurate recovery. For the $L^2$-prior, the optimal $\mathbf{C}$ is diagonal
dominant, and decays rapidly away from the diagonal, cf. Fig. \ref{fig:ph_L2}(b). For the $H^1$-prior,
$\mathbf{C}$ remains largely diagonally dominant, but the off-diagonal entries decay a bit slower. Thus, it is
valid to assume that $\mathbf{C}$ is dominated by local interactions as in Section \ref{ssec:complexity}.
These observations remain largely valid for the other 1D examples, despite that they are much more ill-posed.

\begin{figure}[hbt!]
\centering
\begin{tabular}{cc}
\includegraphics[scale=0.25]{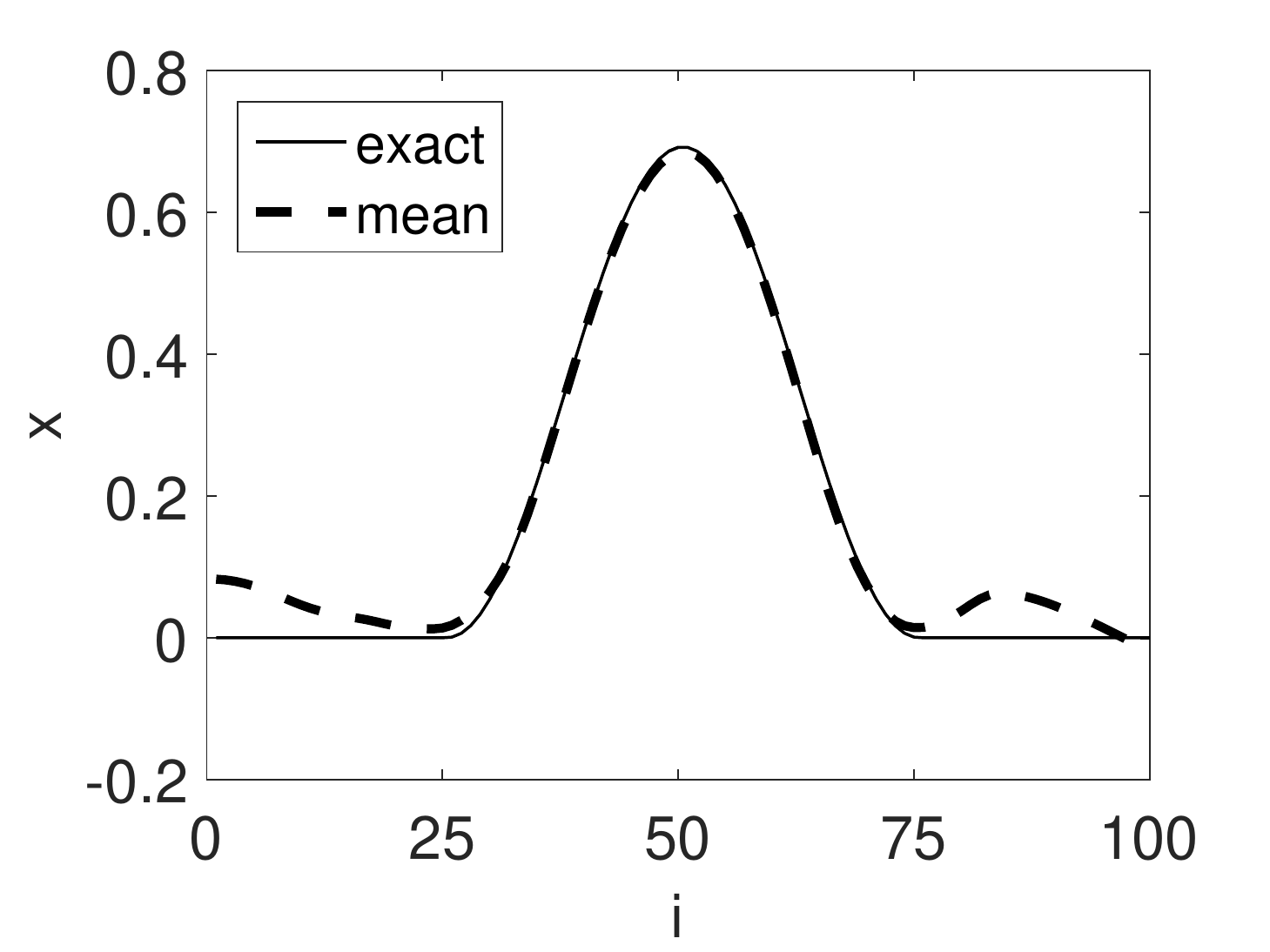}  \includegraphics[scale=0.25]{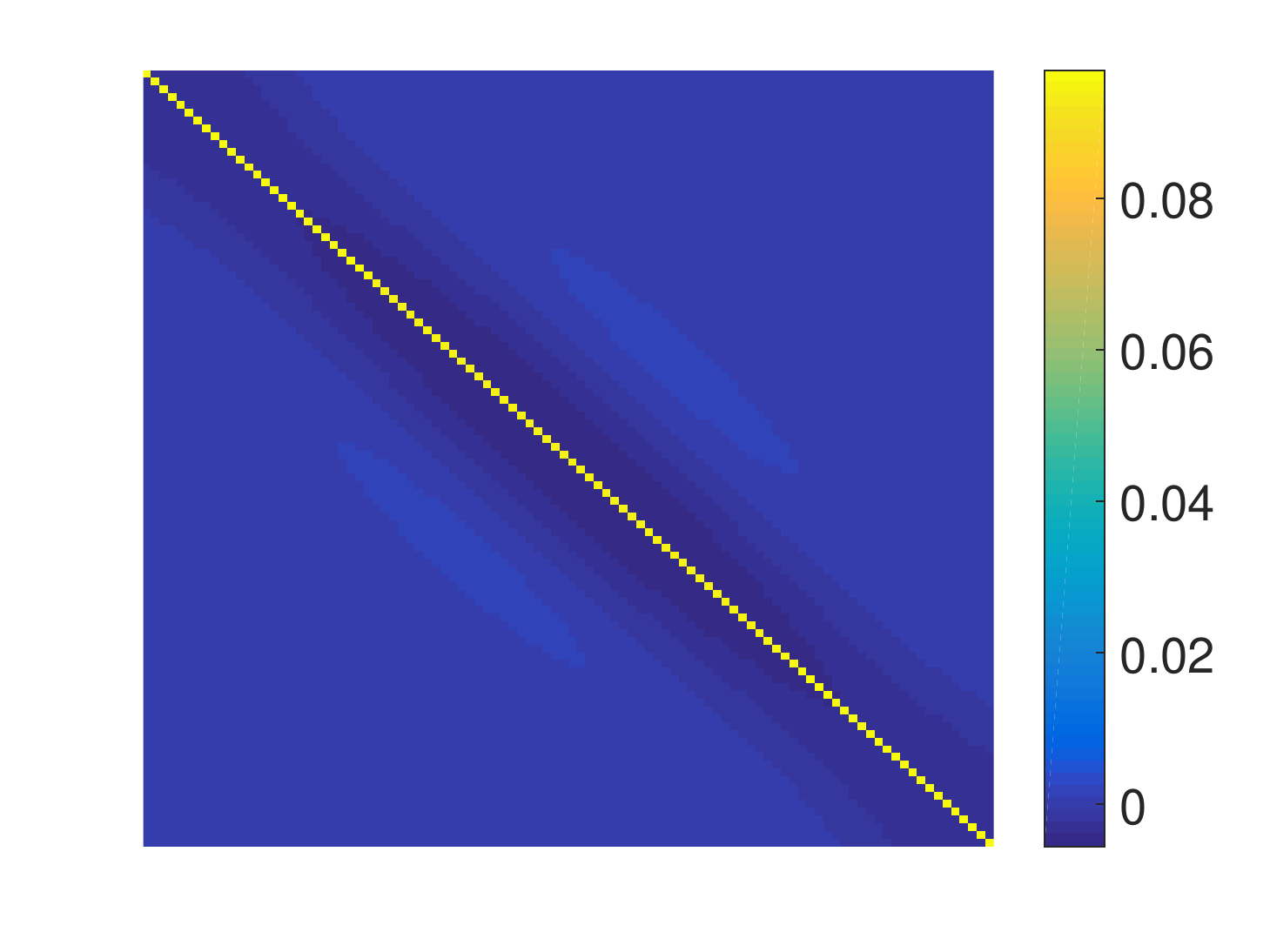}&\includegraphics[scale=0.25]{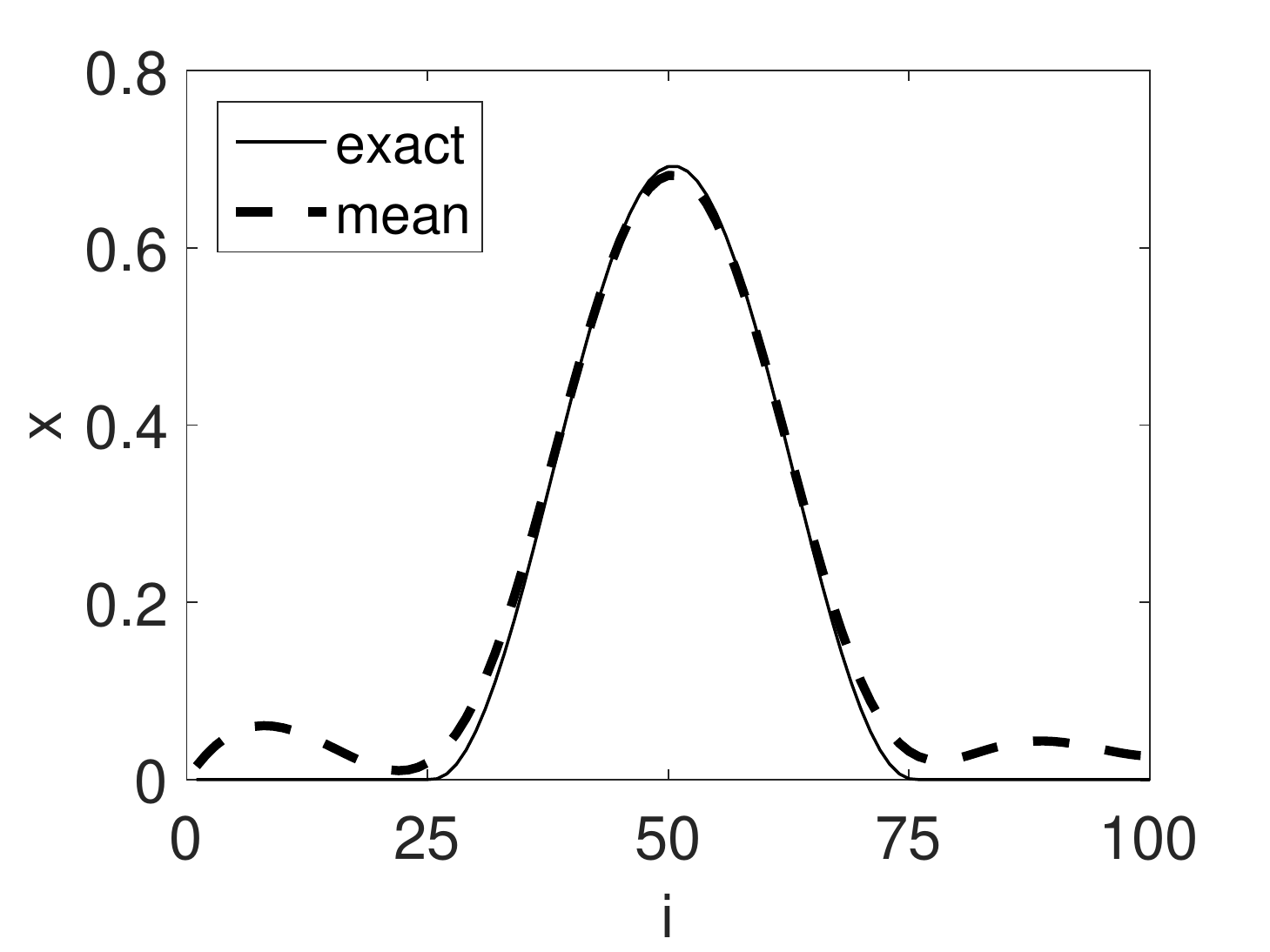} \includegraphics[scale=0.25]{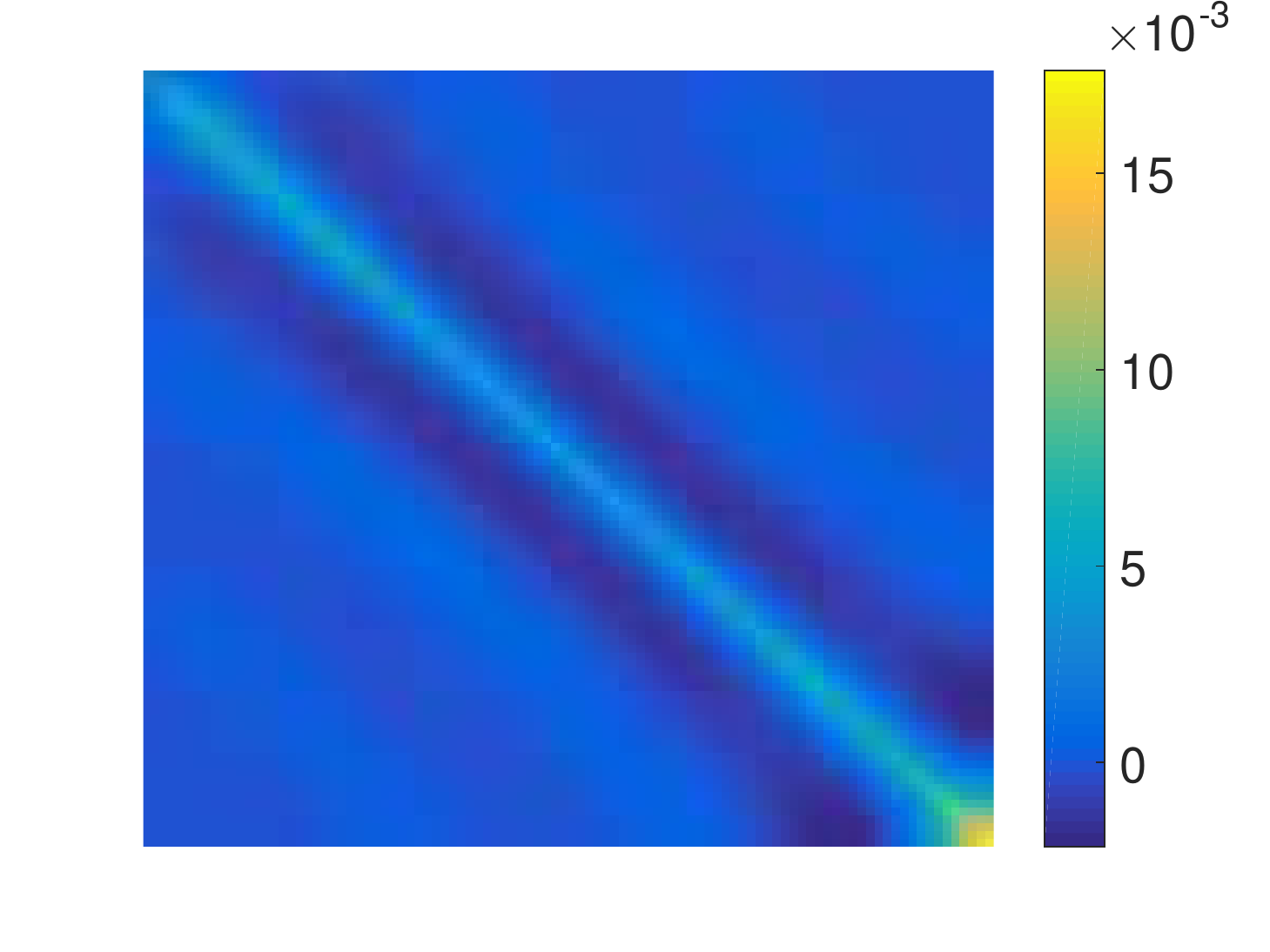}\\
(a) $\mathbf{C}_0=1.00\times 10^{-1}\mathbf{\bar C}_0$ & (b) $\mathbf{C}_0=2.5\times 10^{-3}\mathbf{\bar C}_1$
\end{tabular}
\caption{The Gaussian approximation for \texttt{phillips}.\label{fig:ph_L2}}
\end{figure}

\begin{figure}[hbt!]
\centering
\begin{tabular}{cc}
\includegraphics[scale=0.25]{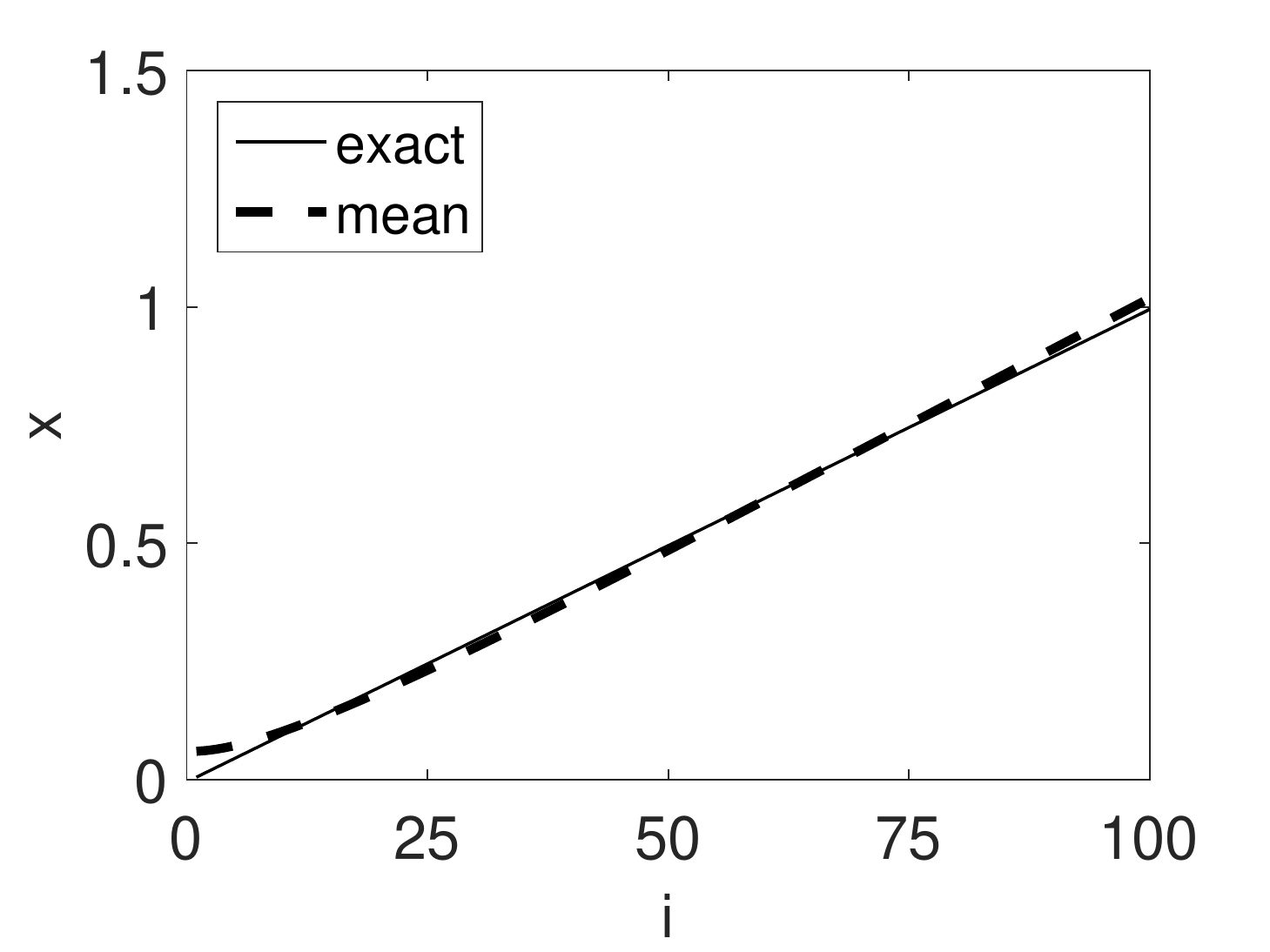}\includegraphics[scale=0.25]{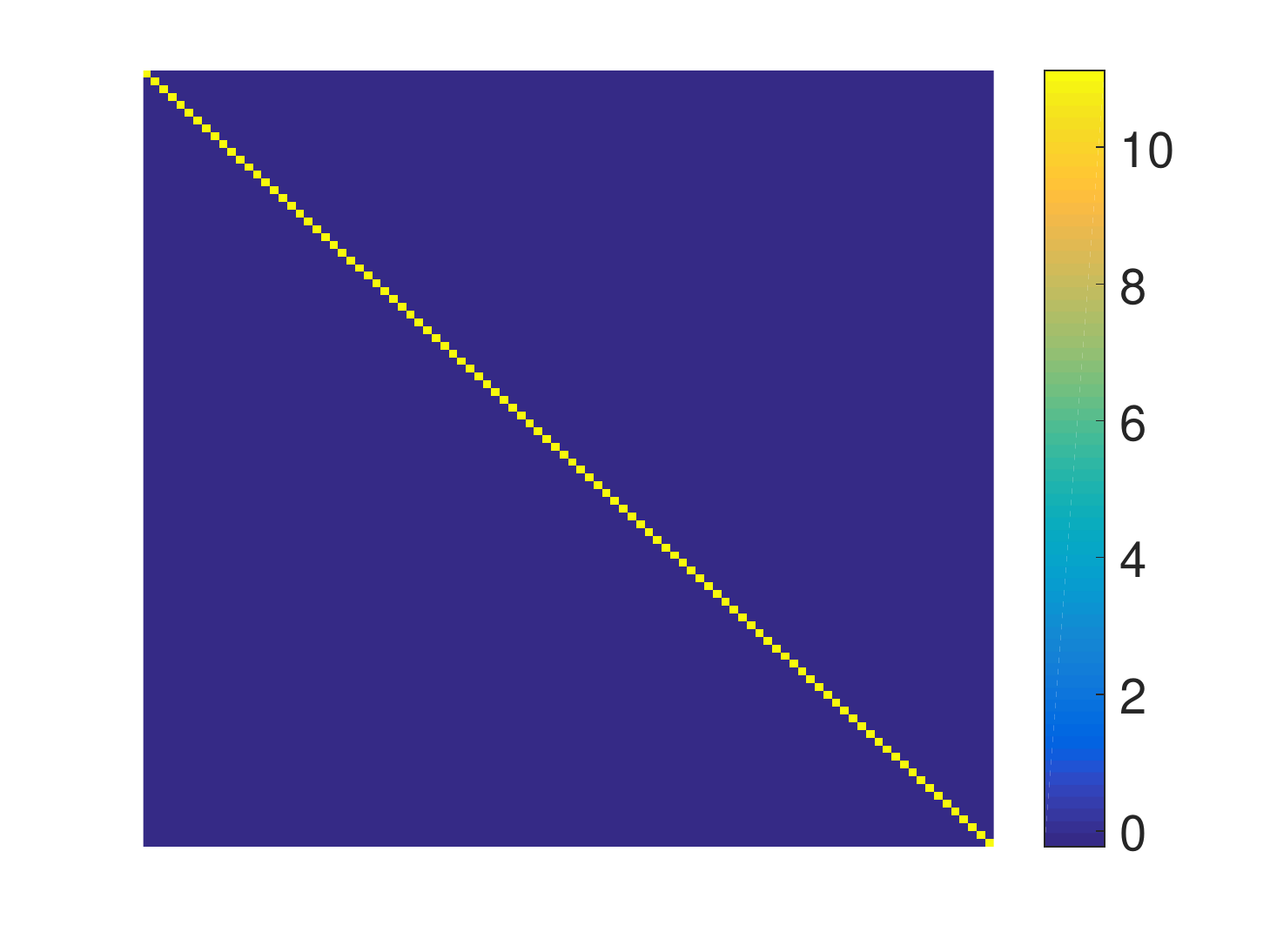} & \includegraphics[scale=0.25]{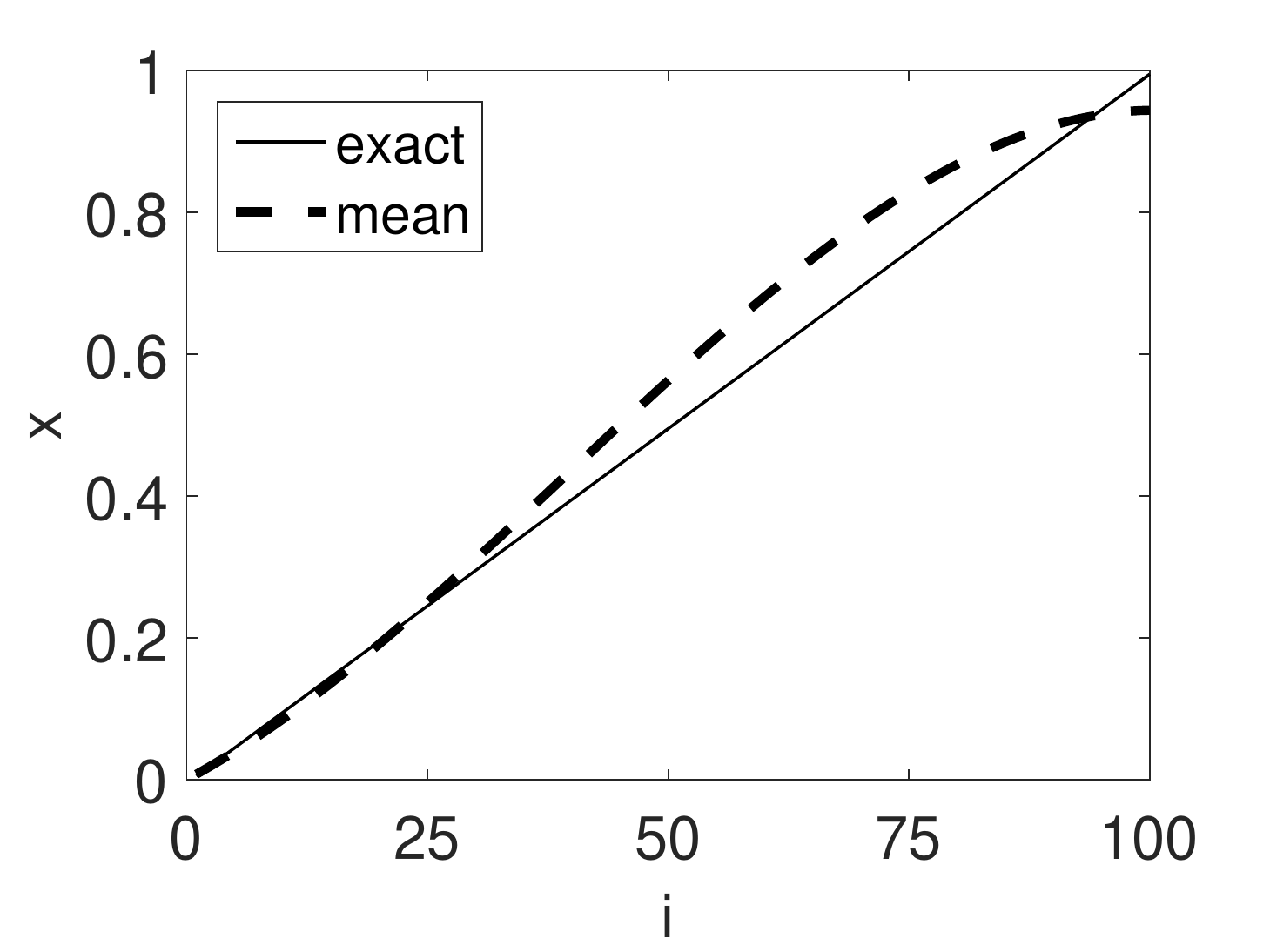}\includegraphics[scale=0.25]{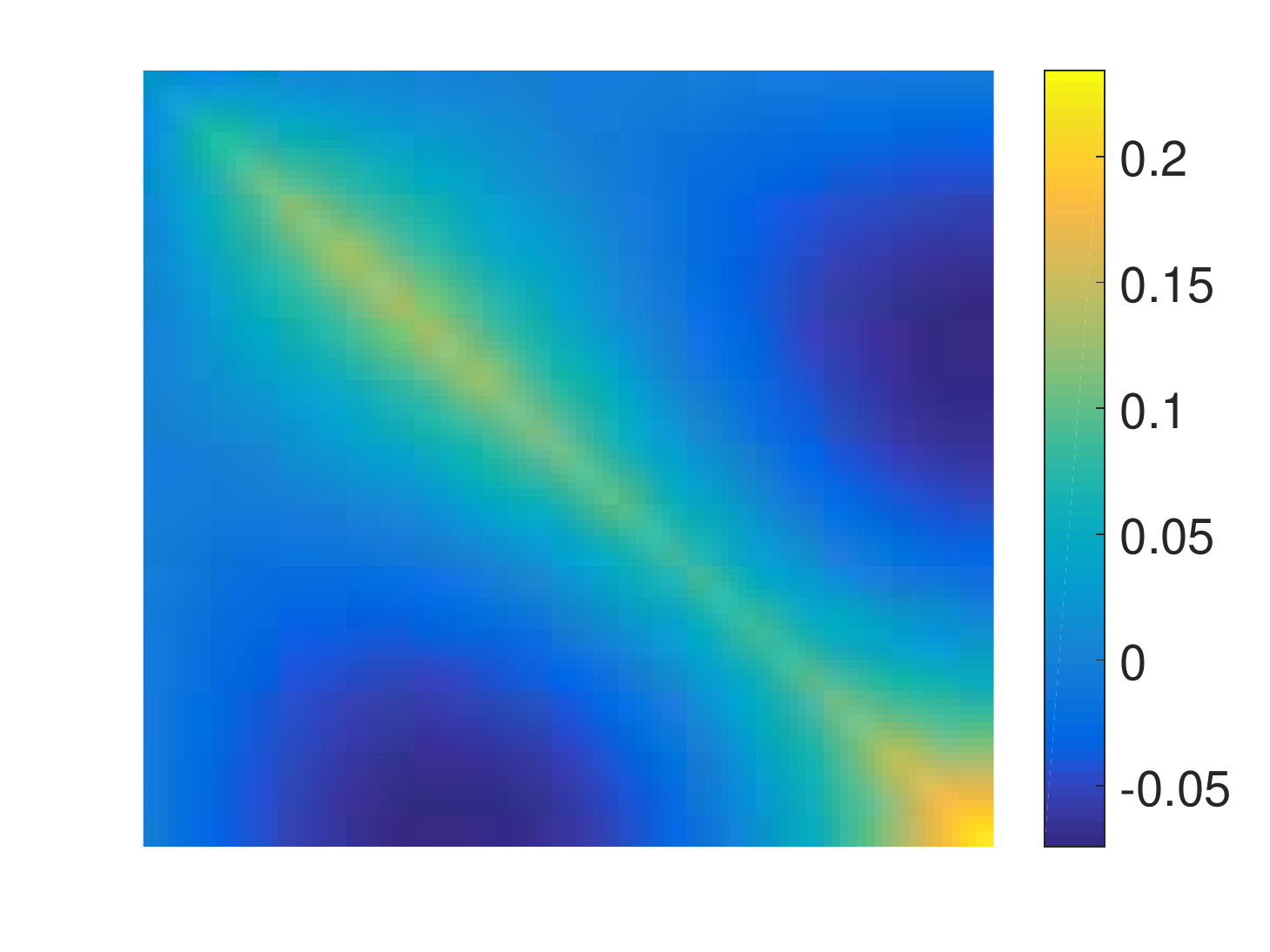}\\
(a) $\mathbf{C}_0=1.12\times 10^{1}\mathbf{\bar C}_0$ & (b) $\mathbf{C}_0=9.8\times 10^{-3}\mathbf{\bar C}_1$
\end{tabular}
\caption{The Gaussian approximation for \texttt{foxgood}.\label{fig:fg_L2}}
\end{figure}

\begin{figure}[hbt!]
\centering
\begin{tabular}{cc}
\includegraphics[scale=0.25]{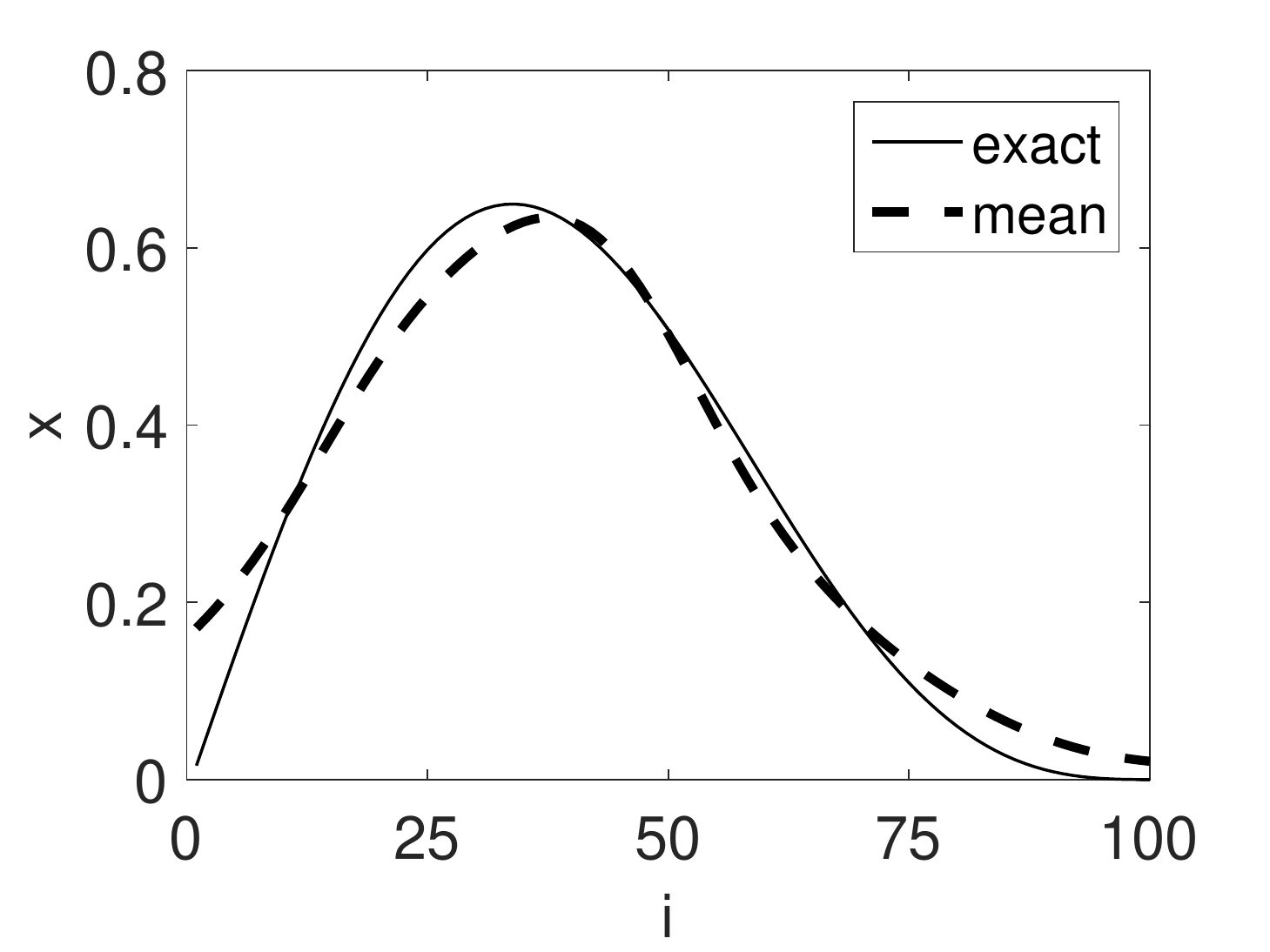}\includegraphics[scale=0.25]{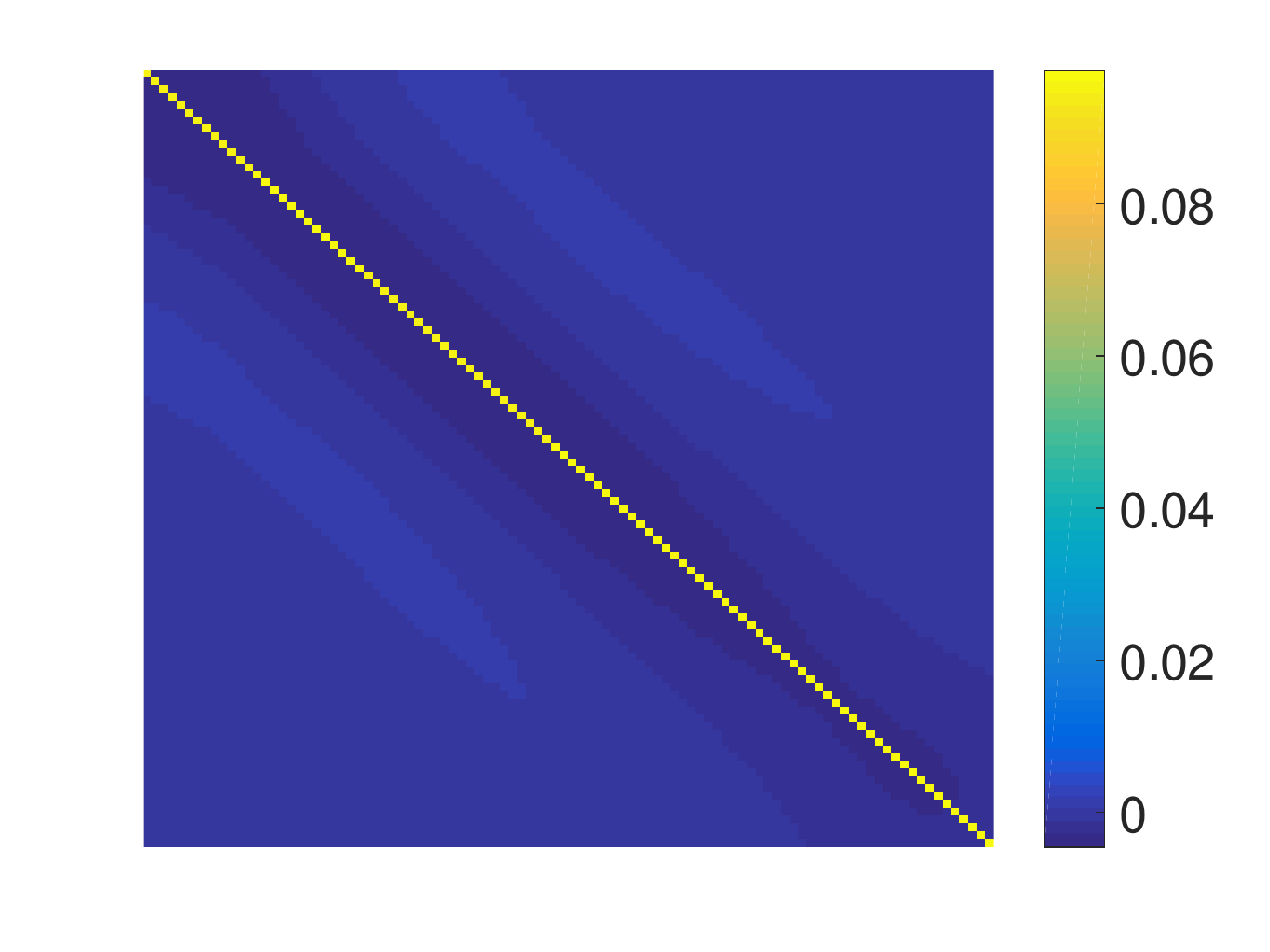}&\includegraphics[scale=0.25]{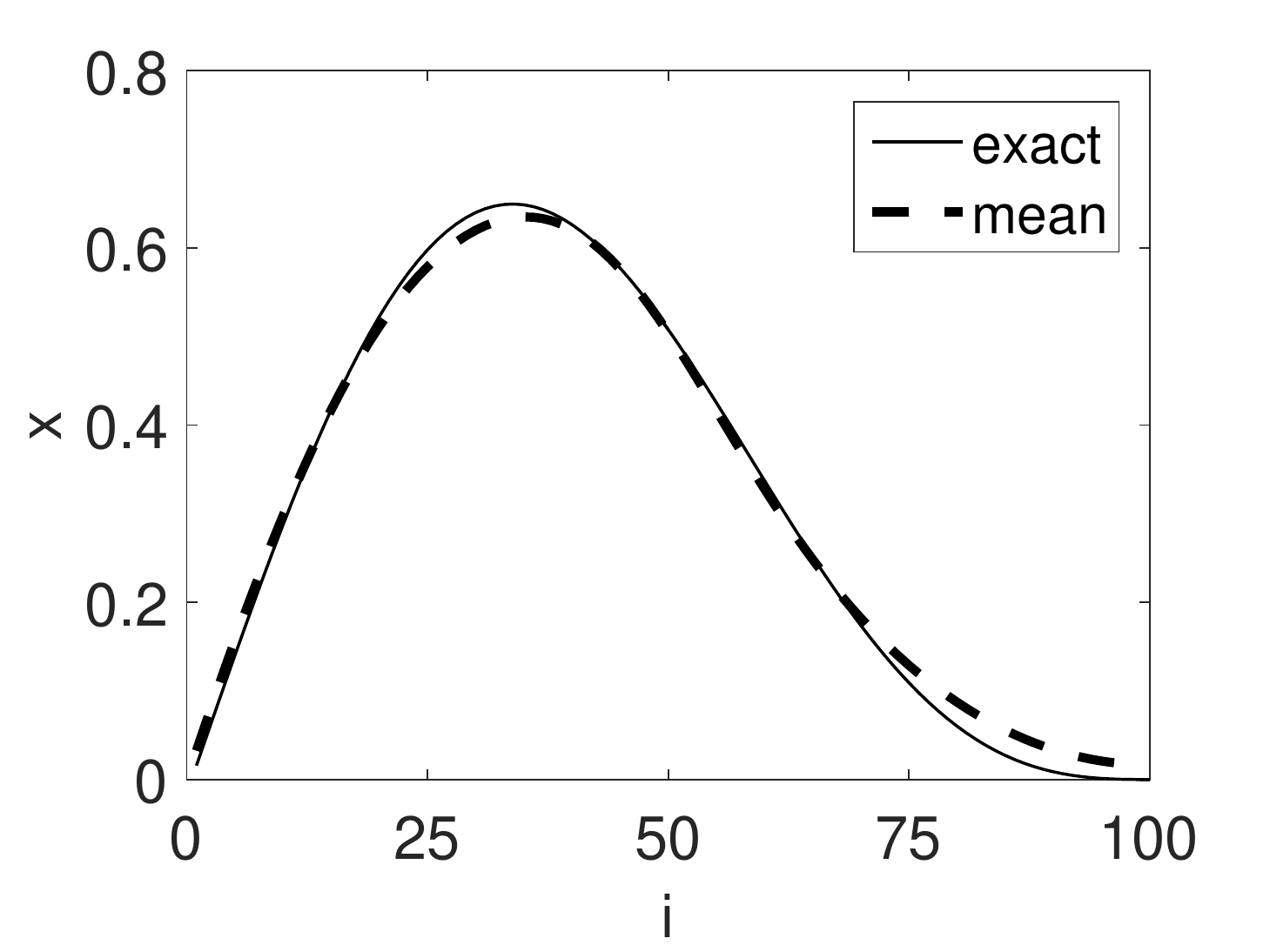}\includegraphics[scale=0.25]{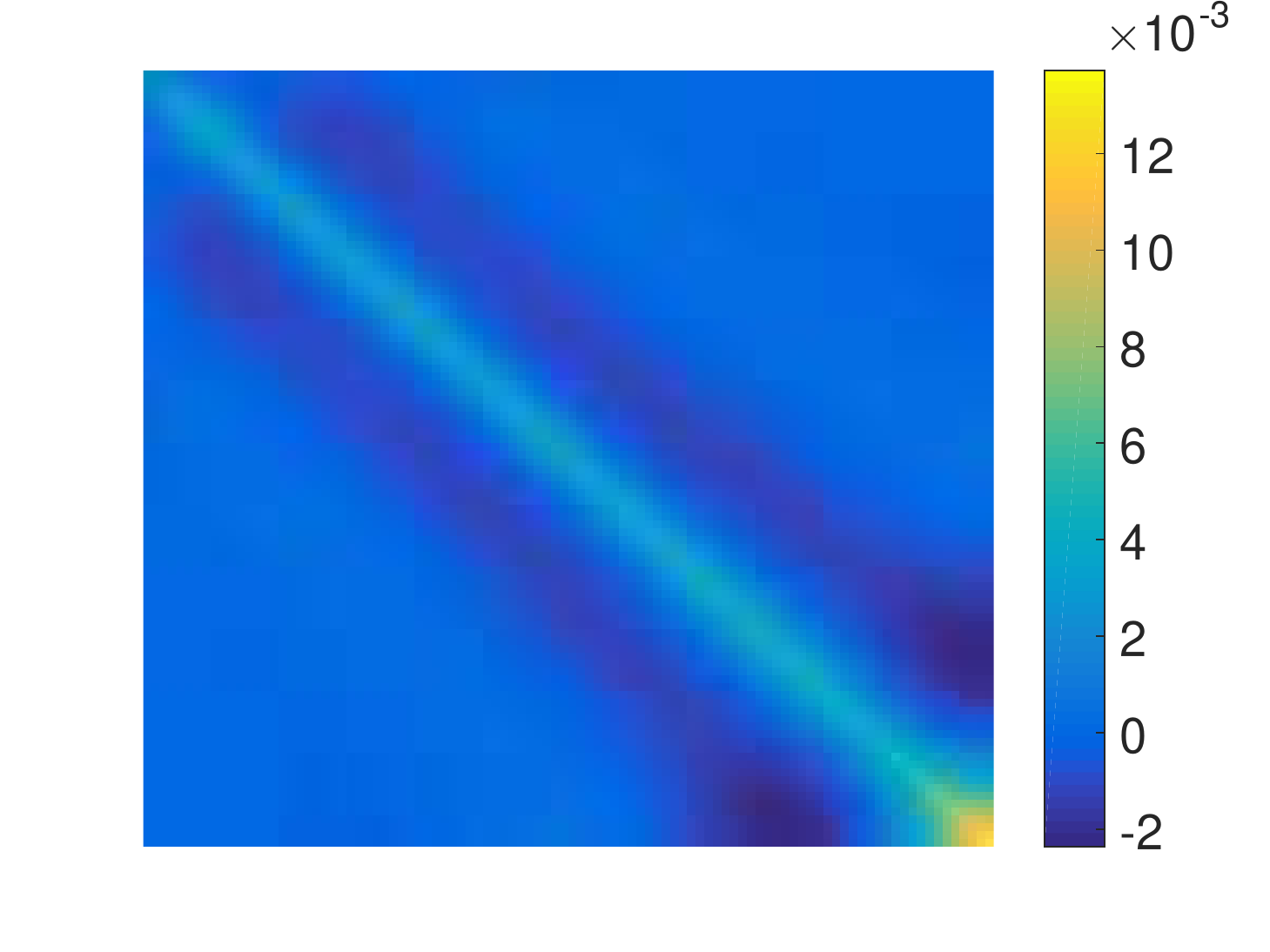} \\
(a) $\mathbf{C}_0=1\times 10^{-1}\mathbf{\bar C}_0$ & (b) $\mathbf{C}_0=1.5\times 10^{-3}\mathbf{\bar C}_1$
\end{tabular}
\caption{The Gaussian approximation for \texttt{gravity}.\label{fig:gr_L2}}
\end{figure}

\begin{figure}[hbt!]
\centering
\begin{tabular}{cc}
\includegraphics[scale=0.25]{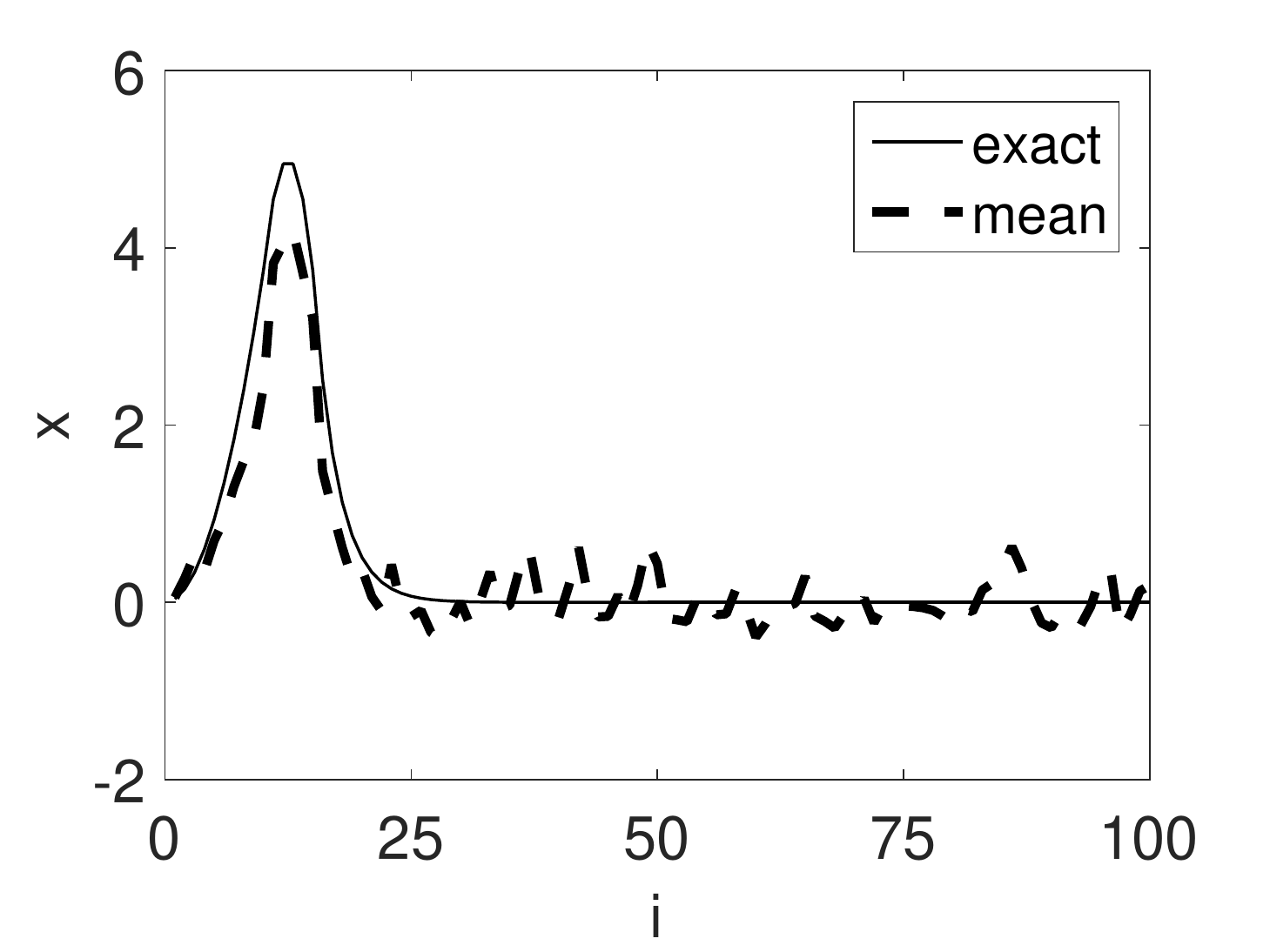}\includegraphics[scale=0.25]{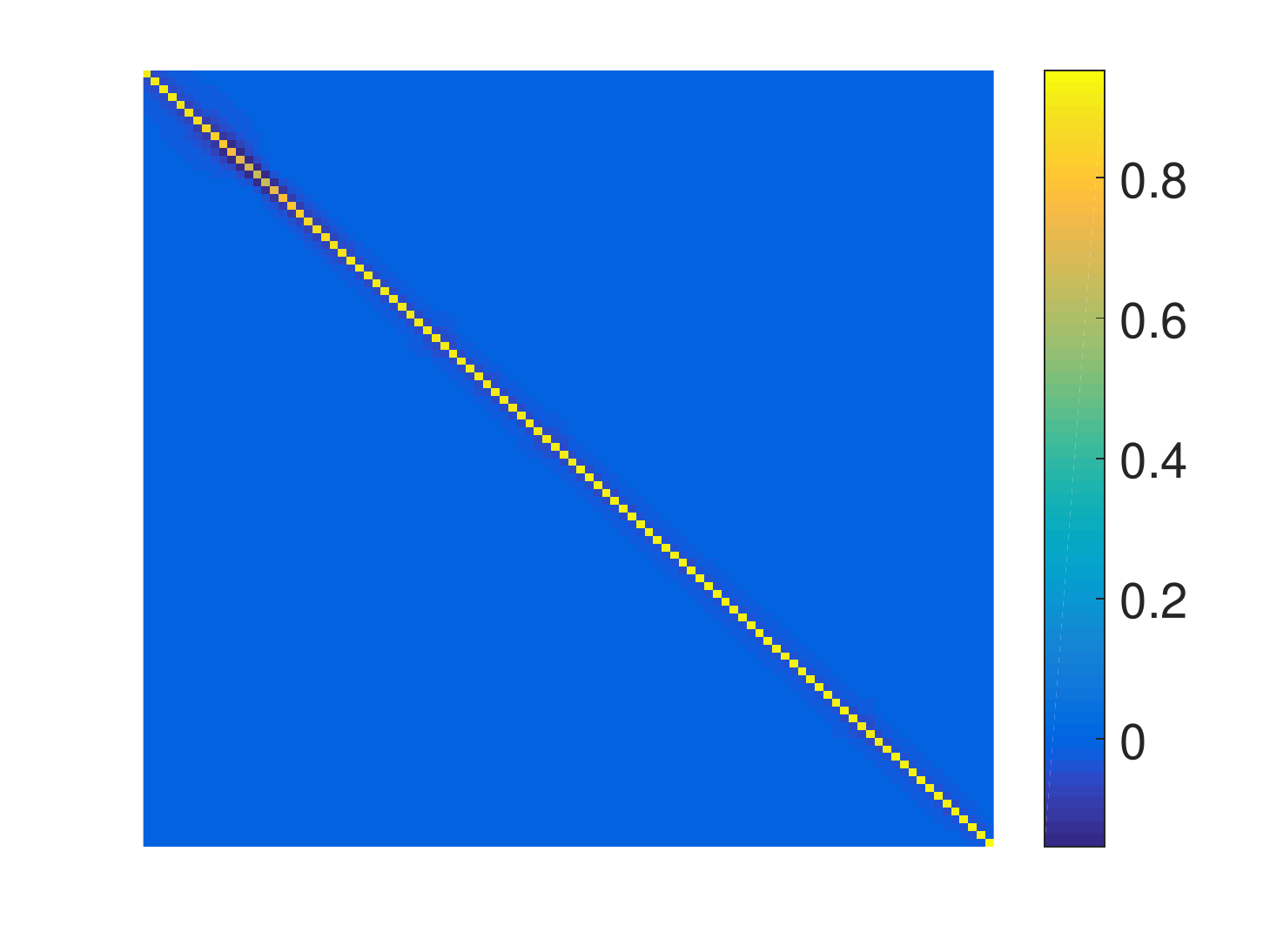}&\includegraphics[scale=0.25]{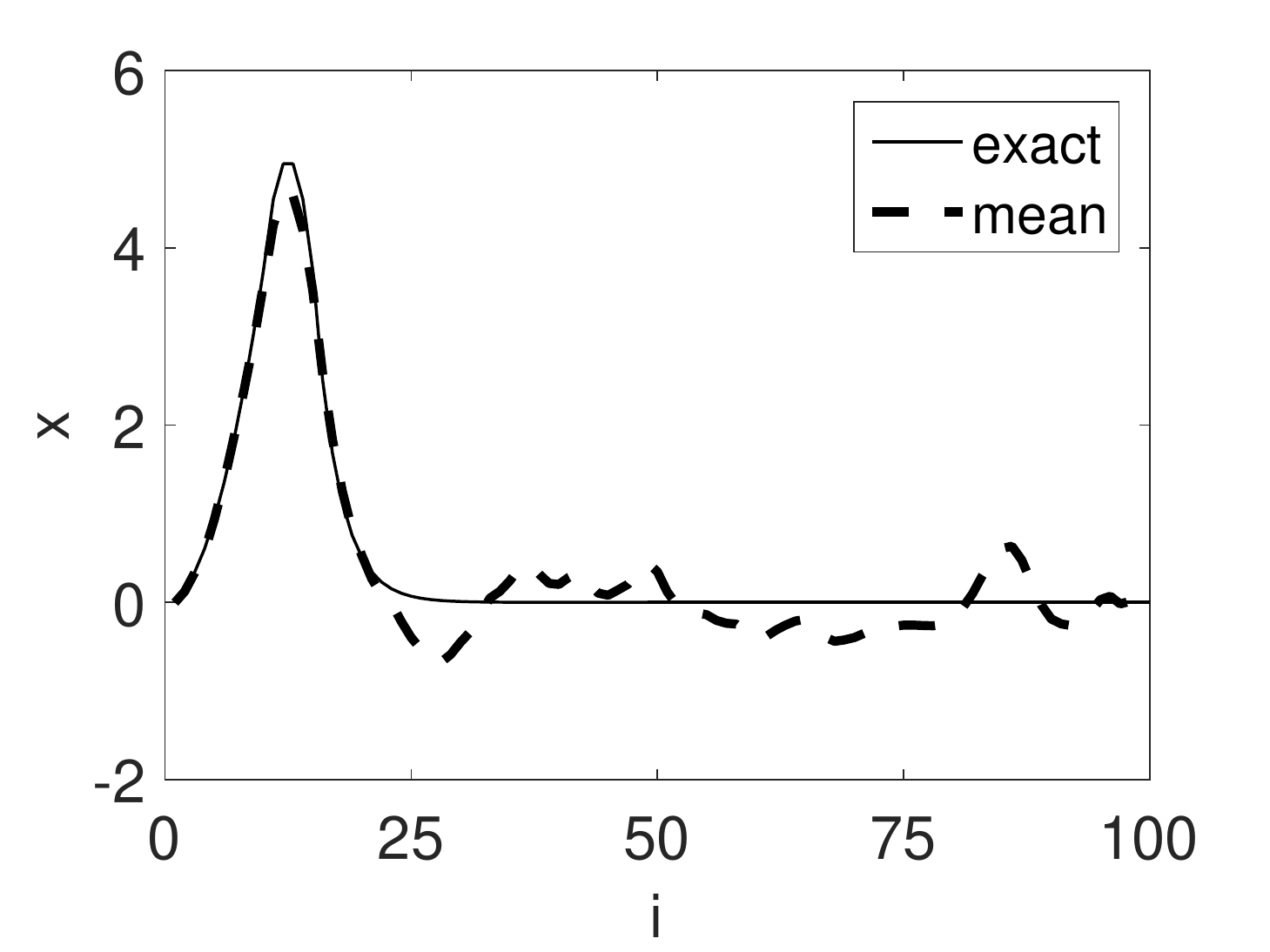}\includegraphics[scale=0.25]{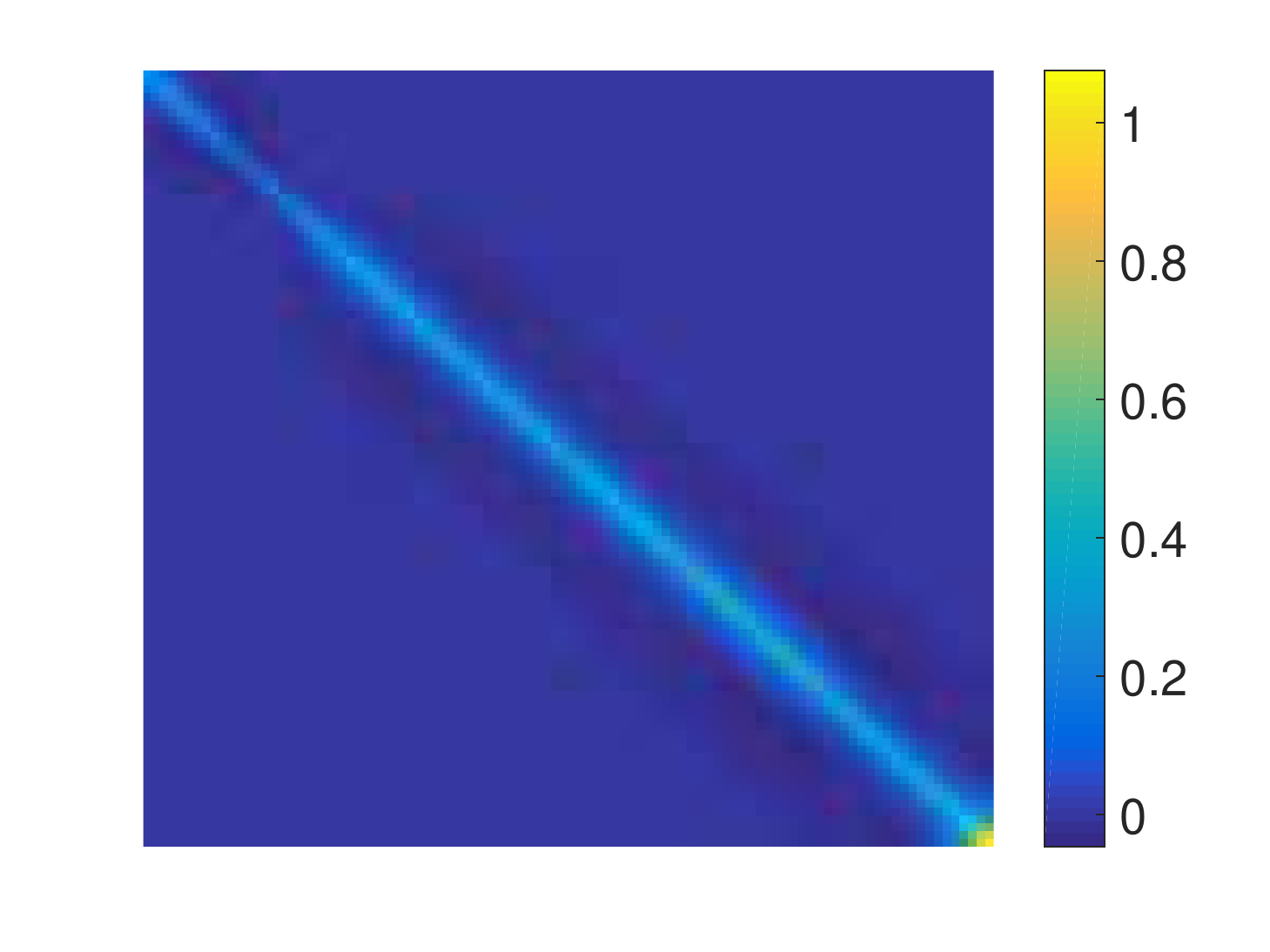}\\
(a) $\mathbf{C}_0=3.2\times 10^{-1}\mathbf{\bar C}_0$ & (b) $\mathbf{C}_0=1\times 10^{0}\mathbf{\bar C}_1$
\end{tabular}
\caption{The Gaussian approximation for \texttt{heat}.\label{fig:he_L2}}
\end{figure}

Last, we test Algorithm \ref{alg:vb} on a 2D image of size $128\times 128$. In this example, the matrix $\mathbf{A}\in \mathbb{R}^{16384\times 16384}$
is a (discrete) Gaussian blurring kernel with a blurring width $99$, variance $1.5$ and a circular boundary condition.
Since the blurring width is large, the matrix $\mathbf{A}$ is indeed low-rank, and we employ a rSVD approximation
of rank $2000$, where the rank is determined by inspecting the singular value spectrum. The true solution $\mathbf{x}^\dag$
consists of two Gaussian blobs, cf. Fig. \ref{fig:2d}(a), and thus we employ a smooth prior with $\mathbf{C}_0 = 6.00\times
10^{-2}\mathbf{L}^{-1} \mathbf{L}^{-t}$, where $\mathbf{L}=\mathbf{I}\otimes\mathbf{L}_1+\mathbf{L}_1\otimes\mathbf{I}$ is
the 2D first-order finite difference matrix. Since the problem size is very large, we restrict $\mathbf{C}$
to be a sparse matrix such that every pixel interacts only with at most four neighboring pixels. This allows
reducing the computational cost greatly. The mean
$\mathbf{\bar x}$ is nearly identical with the true solution $\mathbf{x}^\dag$, and the error is very small, cf. Fig. \ref{fig:2d}.
The structural similarity index between the mean $\mathbf{\bar x}$ and the exact solution $\mathbf{x}^\dag$ is 0.812.
We also compare the VGA solution with the MAP estimator. The $\ell_2$ error of the mean of the VGA is $9.7205$,
which is slightly smaller than that of the MAP estimator ($9.7355$). To indicate the
uncertainty around the mean $\mathbf{\bar x}$, we show in Fig. \ref{fig:2d}(f) the diagonal entries of
$\mathbf{C}$ (i.e., the variance at each pixel). The variances are relatively large at pixels where the mean $\mathbf{\bar x}$ is less accurate.

In summary, the VGA can provide a reliable point estimator together with
useful covariance estimates.

\begin{figure}[hbt!]
\centering
\begin{tabular}{ccc}
\includegraphics[width=0.33\textwidth]{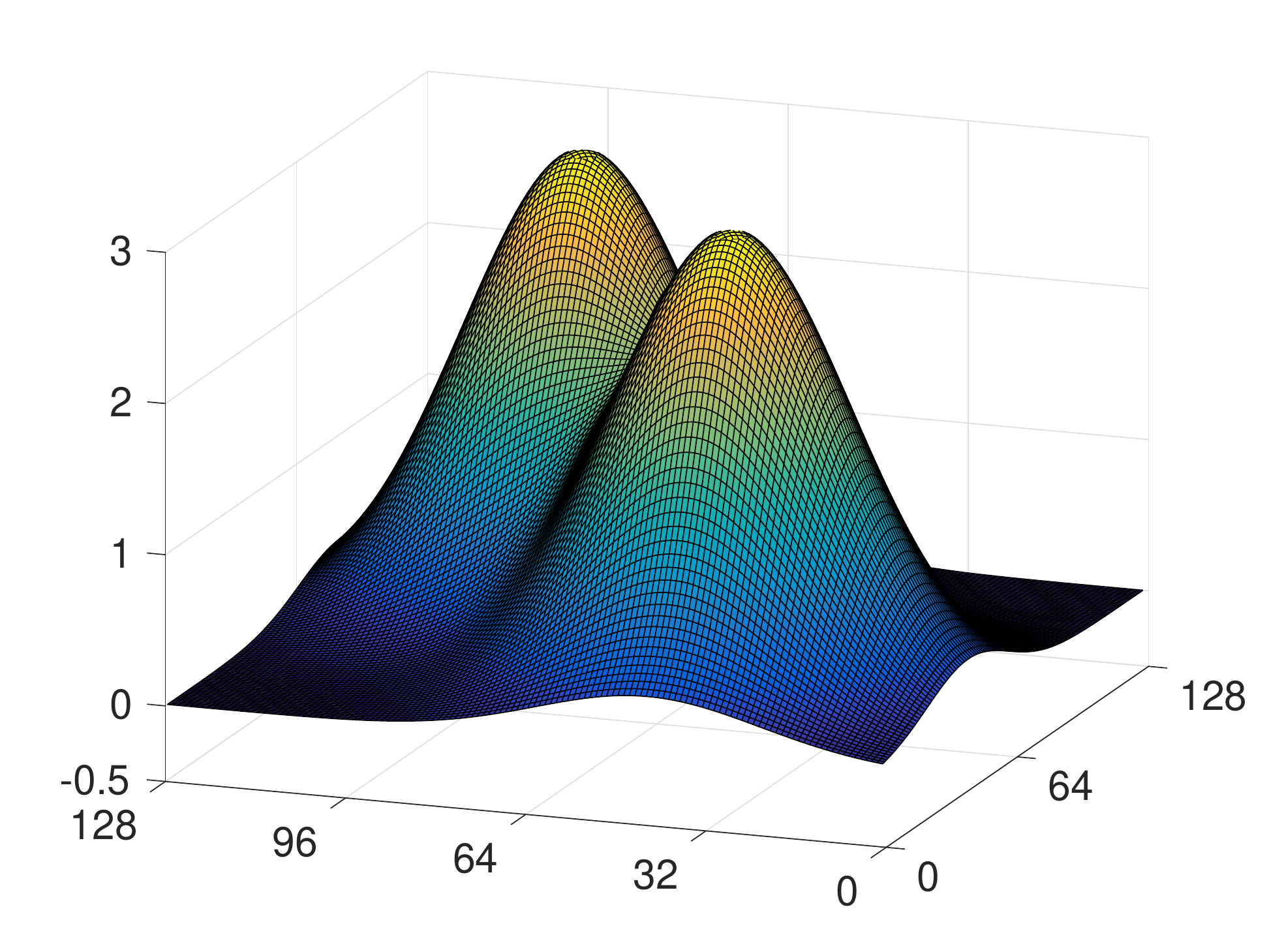} & \includegraphics[width=0.33\textwidth]{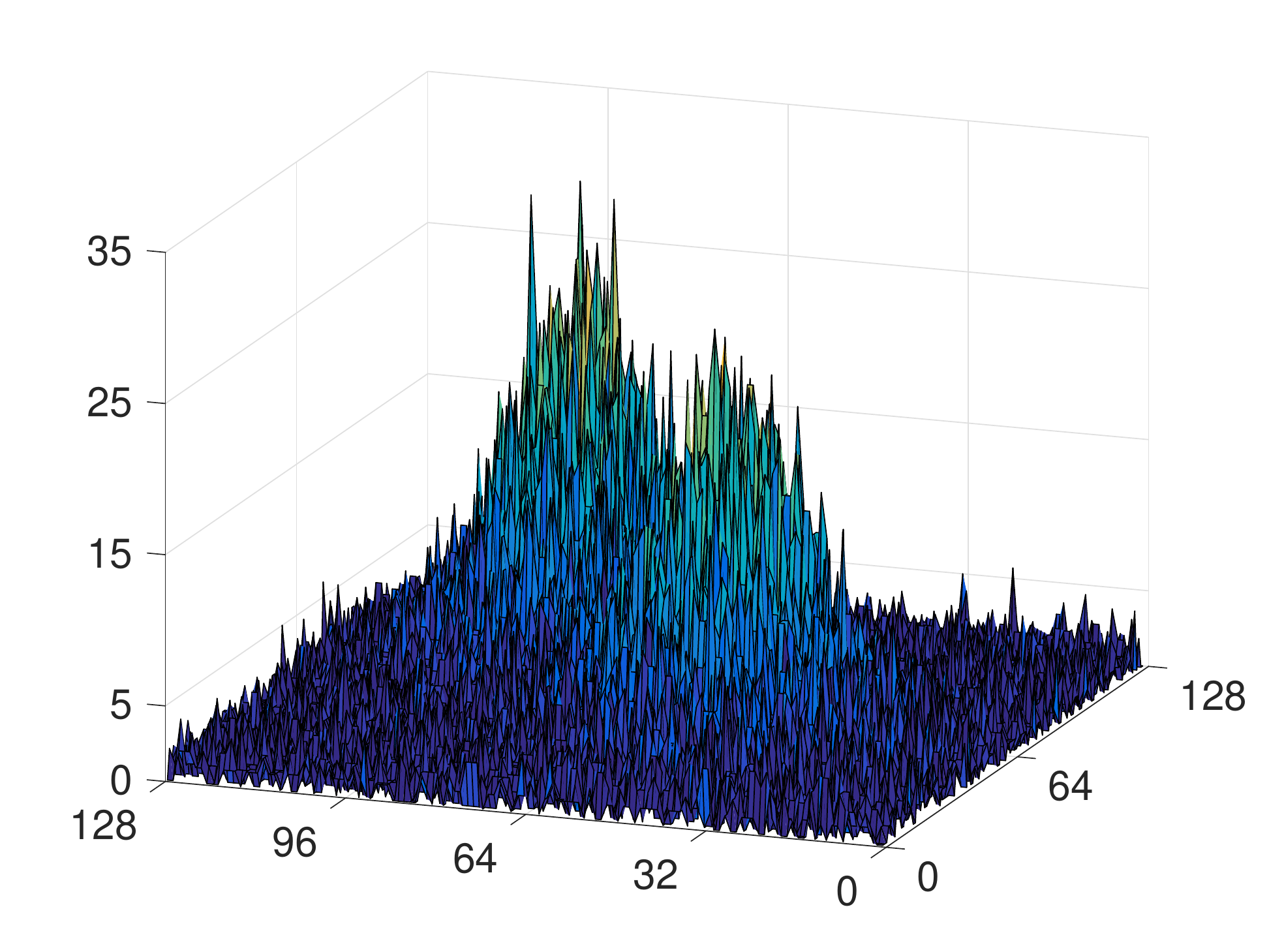} & \includegraphics[width=0.33\textwidth]{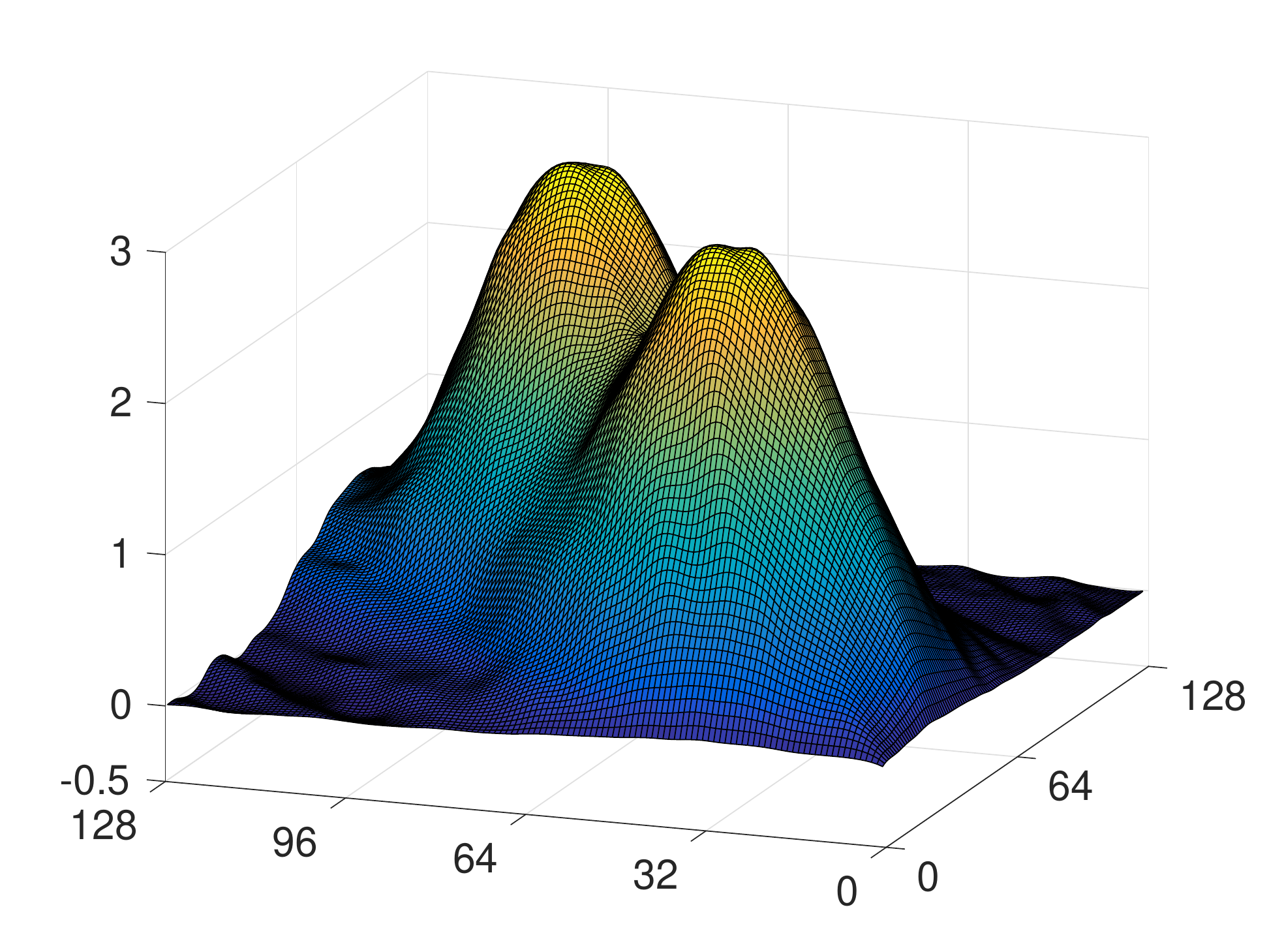} \\
(a) true solution $\mathbf{x}^\dag$ & (b) Poisson sample $\mathbf{y}$ & (c) MAP $\mathbf{x}_{\text{MAP}}$ \\
\includegraphics[width=0.33\textwidth]{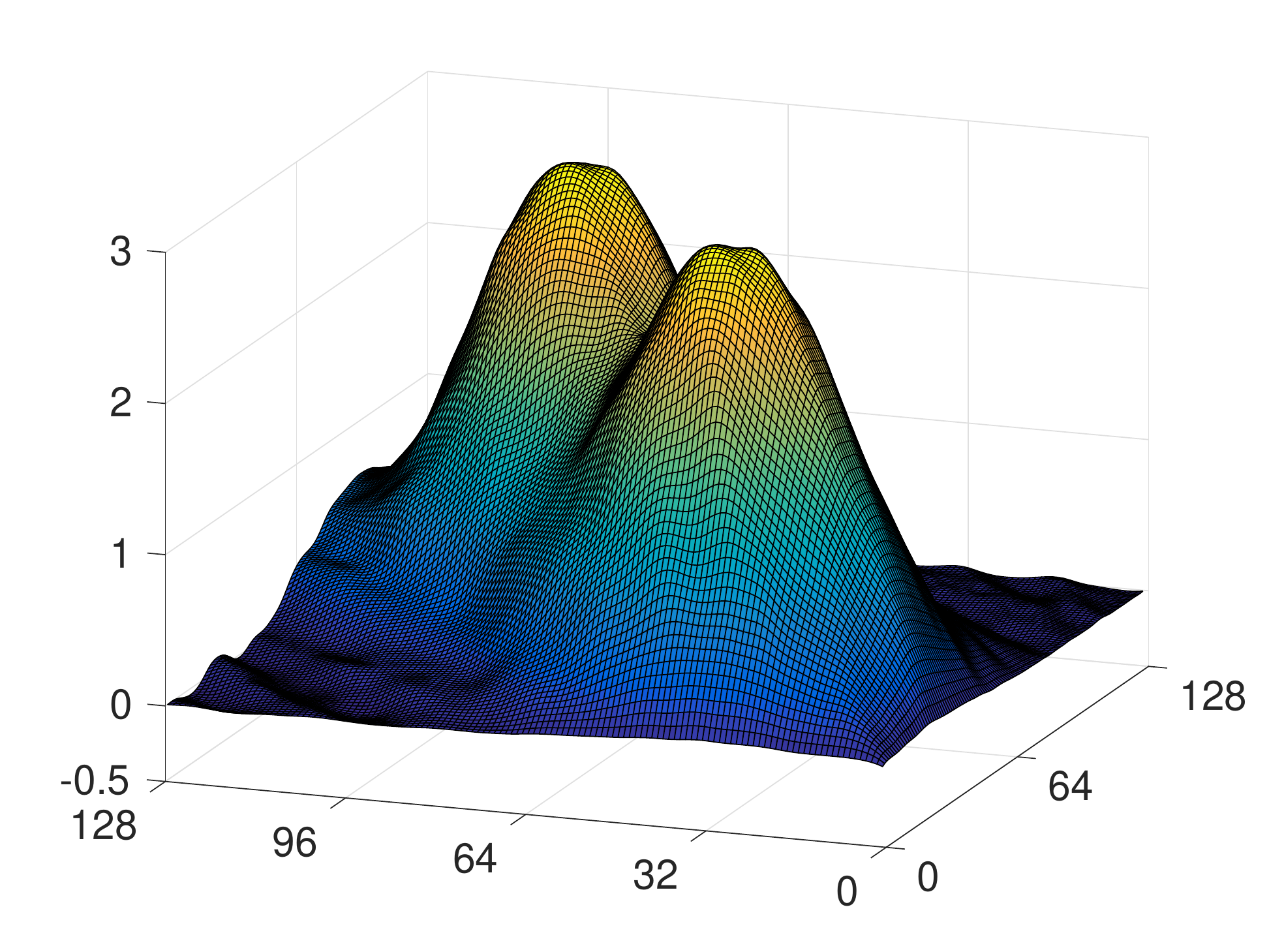}& \includegraphics[width=0.33\textwidth]{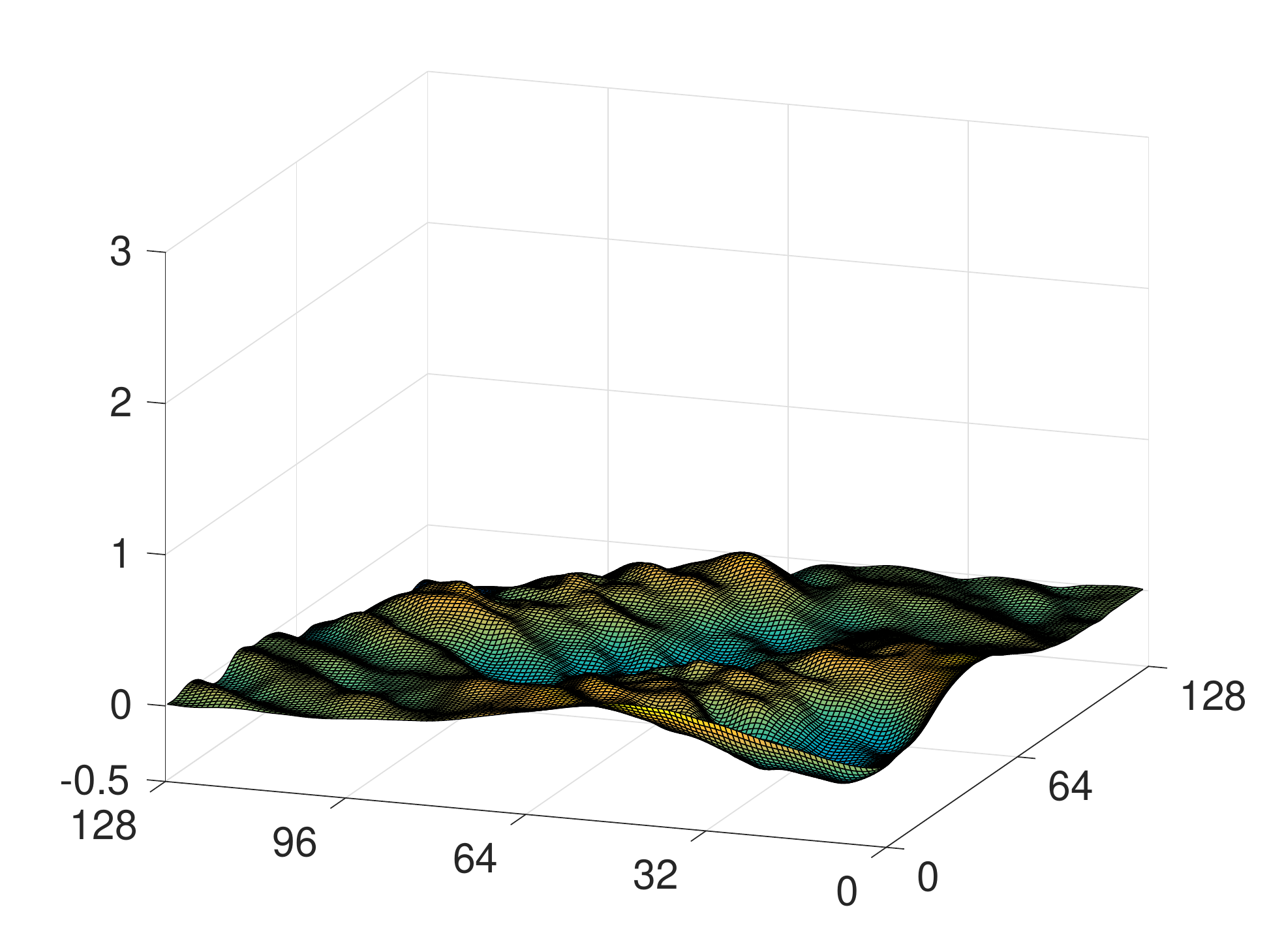} & \includegraphics[width=0.33\textwidth]{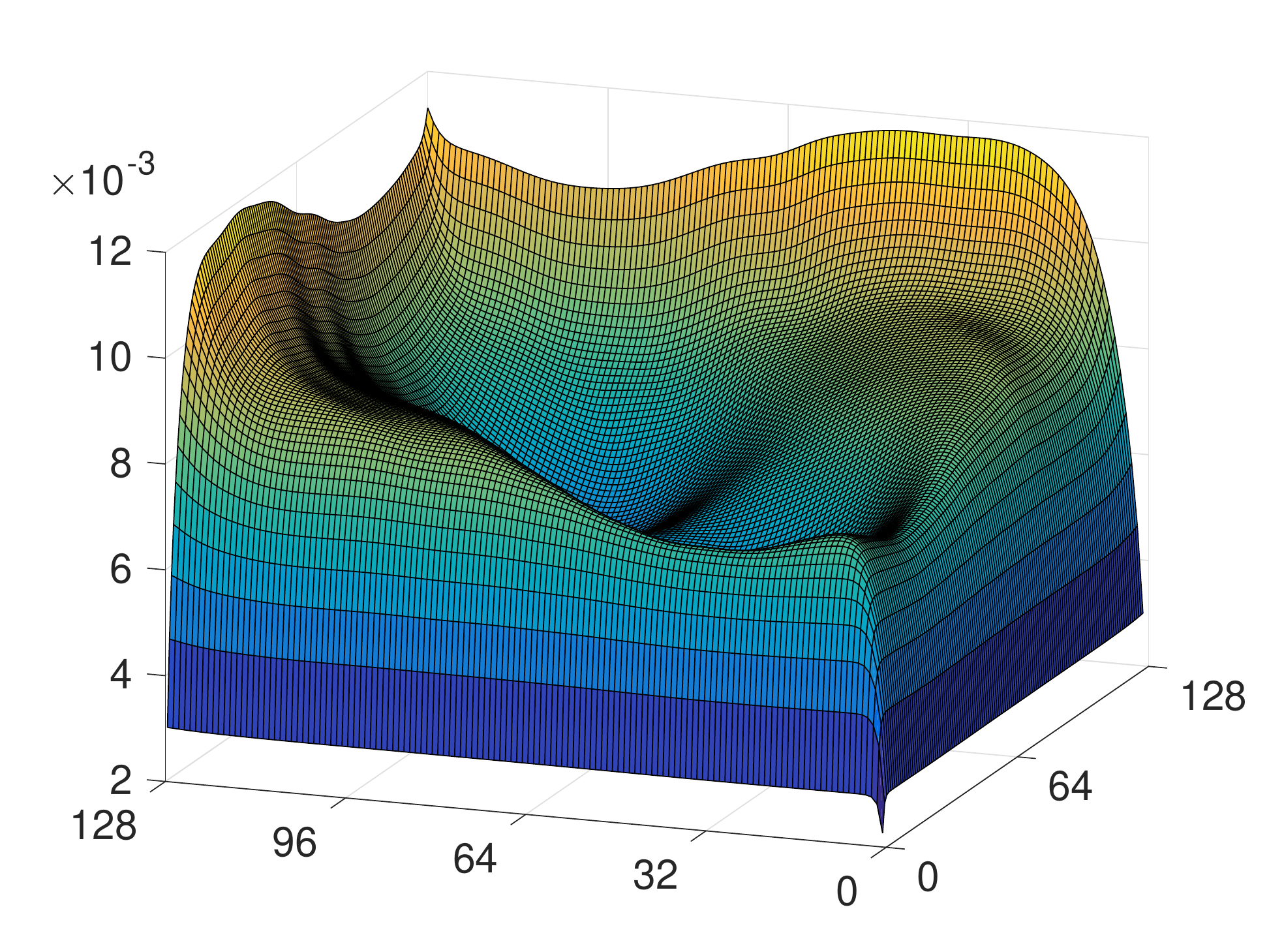}\\
(d) mean $\mathbf{\bar x}$ & (e) error $\mathbf{x}^\dag-\mathbf{\bar x}$ & (f) variance $\mathrm{diag}(\mathbf{C})$
\end{tabular}
\caption{The Gaussian approximation for image deblurring.}\label{fig:2d}
\end{figure}

\section{Conclusions}

In this work, we have presented a study of the variational Gaussian approximation to the Poisson data (under the
log linear link function) with respect to the Kullback-Leibler divergence. We derived explicit expressions for
the lower bound functional and its gradient, and proved its strict concavity and  existence and uniqueness of an
optimal Gaussian approximation. Then we developed an efficient algorithm for maximizing the functional, discussed
its convergence properties, and described practical strategies for reducing the complexity per iteration. Further,
we analyzed hierarchical modeling for automatically determining the hyperparameter using the variational Gaussian
approximation, and proposed a monotonically convergent algorithm for the joint estimation. These discussions were
supported by extensive numerical experiments.

There are several avenues for further study. First, one of fundamental issues is the quality
of the Gaussian approximation relative to the true posterior distribution. In general this
issue has been long standing, and it also remains to be analyzed for the Poisson model. Second,
the variational Gaussian approximation can be viewed as a nonstandard regularization
scheme, by also penalizing the covariance. This naturally motivates the study on its regularizing
property from the perspective of classical regularization theory, e.g., consistency and
convergence rates. Third, the approach generally gives a very
reasonable approximation. This suggests itself as a preconditioner for sampling techniques, e.g., variational
approximation as the proposal distribution (i.e., independence sampler) in the standard
Metropolis-Hastings type algorithm or as the base distribution for importance sampler. It is
expected to significantly speed up the convergence of these sampling procedures, which is
confirmed by the preliminary experiments. We plan to study these aspects in future works.

\section*{Acknowledgements}

The work of B. Jin is supported by UK EPSRC grant EP/M025160/1, and that
of C. Zhang by a departmental studentship.
\appendix

\newcommand{\bA}{{\mathbf{A}}}
\newcommand{\bC}{{\mathbf{C}}}
\newcommand{\bD}{{\mathbf{D}}}

\section{On the iteration \eqref{eqn:iter-C}}\label{app:iter-C}

In this appendix, we discuss an interesting property of the iteration \eqref{eqn:iter-C},
for the initial guess $\mathbf{C}^0=\mathbf{C}_0$. We denote the fixed point
map in \eqref{eqn:iter-C} by $\mathbf{T}$, i.e.,
\begin{equation*}
  \mathbf{T}(\mathbf{C}) = (\mathbf{C}_0^{-1}+\mathbf{A}^t\mathrm{diag}(e^{\mathbf{A\bar x}+\frac{1}{2}\mathrm{diag}(\mathbf{ACA}^t)})\mathbf{A})^{-1}.
\end{equation*}

The next result gives the antimonotonicity of the map $\mathbf{T}$ on $\mathcal{S}_m^+$, i.e., for $\mathbf{C},\tilde{\mathbf{C}}
\in\mathcal{S}_m^+$, if $0\leq \mathbf{C}\leq \tilde\bC$, then
$\mathbf{T}(\bC)\geq \mathbf{T}(\tilde \bC)$.

\begin{lemma}\label{lem:anti-monotone}
The mapping $\mathbf{T}$ is antimonotone.
\end{lemma}
\begin{proof}
Let $\mathbf{C},\tilde{\mathbf{C}}\in\mathcal{S}_m^+$. If $\mathbf{C}\leq\mathbf{\tilde C}$, then
${\rm diag}(\mathbf{ACA}^t)\leq {\rm diag}(\mathbf{A\tilde CA}^t)$
componentwise. The claim follows from the identity
$\mathbf{T}(\mathbf{C})-\mathbf{T}(\mathbf{\tilde C}) = \mathbf{T}(\mathbf{C})\mathbf{A}^t{\rm diag}(e^{\mathbf{A\bar x}+\frac{1}{2}\mathrm{diag}
  (\mathbf{A\tilde CA}^t)}-e^{\mathbf{A\bar x}+\frac{1}{2}\mathrm{diag}(\mathbf{ACA}^t)})\mathbf{A}\mathbf{T}(\mathbf{\tilde C})\geq 0.$
\end{proof}

The next result shows the monotonicity of the sequence $\{\mathbf{C}^k\}$ generated by \eqref{eqn:iter-C}.
\begin{lemma}\label{lem:mon-C}
For any initial guess $\mathbf{C}^0\in \mathcal{S}_m^+$, the sequence $\{\mathbf{C}^k\}_{k\geq0}$ generated by the iteration \eqref{eqn:iter-C} has the following
properties: (i) $\mathbf{C}^k\geq 0$ for all $k\geq 0$; (ii) $\mathbf{C}^k \leq \mathbf{C}_0$ for all $k\geq 0$;
(iii) If $\mathbf{C}^k\geq \mathbf{C}^j$ then $\mathbf{C}^{k+1}\leq \mathbf{C}^{j+1}$;
(iv) If $\mathbf{C}^{k}\geq \mathbf{C}^{j}$ then $\mathbf{C}^{k+2}\geq \mathbf{C}^{j+2}$.
\end{lemma}
\begin{proof}
Properties (i) and (ii) are obvious. Properties (iii) and (iv) are direct consequences of
the fact that the map $\mathbf{T}$ is antimonotone on $\mathcal{S}_m^+$, cf. Lemma \ref{lem:anti-monotone}.
\end{proof}

The next result shows that the sequence constitutes two subsequences, each converging to a fixed point of $\mathbf{T}^2$,
which implies either a periodic orbit of period 2 of the map $\mathbf{T}$ or a fixed point of $\mathbf{T}$,
\begin{theorem}
With the initial guess $\mathbf{C}^0=\bC_0$, the sequence $\{\mathbf{C}^k\}_{k\geq 0}$ generated by iteration \eqref{eqn:iter-C}
converges to a fixed-point of $\mathbf{T}^2$.
\end{theorem}
\begin{proof}
Lemma \ref{lem:mon-C}(ii) implies
\begin{equation}\label{eqn:iterC-2}
\mathbf{C}^2\leq \mathbf{C}^0,
\end{equation}
so we can use Lemma \ref{lem:mon-C}(iv) inductively to argue that $\{\mathbf{C}^{2k}\}_{k\geq 0}$ is a decreasing
sequence. From \eqref{eqn:iterC-2} and Lemma \ref{lem:mon-C}(iii), we deduce $\mathbf{C}^1\leq \mathbf{C}^3$,
which together with Lemma \ref{lem:mon-C}(iv) implies that the sequence $\{\mathbf{C}^{2k+1}\}_{k\geq 0}$ is
increasing. By the boundedness and monotonicity, both $\{\mathbf{C}^{2k}\}_{k\geq 0}$ and $\{\mathbf{C}^{2k+1}\}_{k\geq 0}$
converge, with the limit $\mathbf{C}^*$ and $\mathbf{C}^{**}$, respectively. These are the limits of the
fixed point map $\mathbf{T}^2$.
\end{proof}

\begin{remark}
By Lemma \ref{lem:mon-C}, $\mathbf{C}^*\geq \mathbf{C}^{**}$, and if $\mathbf{C}^*=\mathbf{C}^{**}$,
the whole sequence converges. Generally, the interval of matrices $[\mathbf{C}^{**},\mathbf{C}^*]$ provides a lower and
sharp bounds for the fixed point of the iteration \eqref{eqn:iter-C} (which is a priori known to be unique and to exist).
By repeating the argument in \cite[Theorem 2.2]{ElSayed:2001}, one may also examine the convergence of the sequence for
the initial guess either  $\mathbf{C}^0<\mathbf{C}^{**}$ or $\mathbf{C}^0>\mathbf{C}^*$.
\end{remark}

\section{Differentiability of the regularized solution}\label{app:sensitivity}
In this part, we discuss the differentiability of the regularized solution $(\mathbf{\bar x}_\alpha,\mathbf{C}_\alpha)$
in $\alpha$.  For simplicity, we omit the subscript $\alpha$. By differentiating \eqref{eqn:opt} in $\alpha$
and chain rule, we obtain (with $\dot {\bar{\mathbf x}} =\frac{d\mathbf{\bar x}}{d\alpha}$
and $\dot{\mathbf{ C}}=\frac{d\mathbf{ C}}{d\alpha}$):
\begin{equation}\label{eqn:sensitivity}
  \begin{aligned}
     (\mathbf{A}^t\mathbf{DA}+ \alpha \bC_0^{-1}) \dot{{\bar{\mathbf x}}} + \tfrac{1}{2} \mathbf{A}^t\mathbf{D}\mathrm{diag}(\mathbf{A}\dot{\mathbf{C}}\mathbf{A}^t) &= - \bar\bC_0^{-1}(\mathbf{\bar{x}}-\boldsymbol\mu_0),\\
	(\mathbf{C}^{-1}\dot{\mathbf{C}}\emph{} \mathbf{C}^{-1}+\tfrac{1}{2}\mathbf{A}^t \mathbf{D}^\frac{1}{2}\mathrm{diag}\mathrm{diag}(\mathbf{A}\dot {\mathbf{C}}\mathbf{A}^t)\mathbf{D}^\frac{1}{2}\mathbf{A}) + \mathbf{A}^t\mathbf{D}^\frac{1}{2}\mathrm{diag}(\mathbf{A}\dot{\bar{ \mathbf x}})\mathbf{D}^\frac{1}{2}\mathbf{A} &= - \bar\bC_0^{-1},
  \end{aligned}
\end{equation}
where $\mathbf D=\mathrm{diag}(e^{\mathbf{A{\bar{x}}}+\frac{1}{2}{\rm diag}(\mathbf{ACA}^t)})\in\mathbb{R}^{n\times n}$ is a diagonal matrix.
This constitutes a coupled linear system for $(\dot {\bar {\mathbf x}},\dot{\mathbf{C}})$. The next result
gives its unique solvability.
\begin{theorem}\label{thm:sensitivity-sol}
The sensitivity system \eqref{eqn:sensitivity} is uniquely solvable.
\end{theorem}
\begin{proof}
Since the system \eqref{eqn:sensitivity} is linear and square, it suffices to show that the homogeneous problem has only a zero solution.
To this end, by eliminating the variable $\dot{\bar{\mathbf x}}$ from the second line in \eqref{eqn:sensitivity} using the first
line, we obtain the Schur complement for $\dot{\mathbf{C}}$:
\begin{equation*}
  \mathbf{C}^{-1}\dot{\mathbf{C}}\mathbf{C}^{-1}+\tfrac{1}{2}\mathbf{A}^t \mathbf{D}^\frac{1}{2}\mathrm{diag}\mathrm{diag}(\mathbf{A}\dot{\mathbf {C}}\mathbf{A}^t)\mathbf{D}^\frac{1}{2}\mathbf{A} - \tfrac{1}{2}\mathbf{A}^t\mathbf{D}^\frac{1}{2}\mathrm{diag}(\mathbf{A}(\mathbf{A}^t\mathbf{DA}+ \alpha \bar\bC_0^{-1})^{-1}\mathbf{A}^t\mathbf{D}\mathrm{diag}(\mathbf{A}\dot{\mathbf{C}}\mathbf{A}^t))\mathbf{D}^\frac{1}{2}\mathbf{A}.
\end{equation*}
For any fixed $\mathbf{C}$, this defines a linear map on $\mathbb{R}^{m\times m}$.
Next we show its invertibility. To this end, we take inner product the map with $\dot{\mathbf{C}}$,
and show its positivity. Clearly, the first term is strictly positive. Thus it suffices to consider the last two terms.
By the cyclic property of trace, with $\mathbf{d}=\mathrm{diag}(\mathbf{D})\in\mathbb{R}^n$, we have
\begin{equation*}
  \begin{aligned}
     &({\bA}^t \mathrm{diag}(\mathbf{d}\circ\mathrm{diag}(\mathbf{A}\dot{\mathbf C}\mathbf{A}^t))\mathbf{A},\dot{\mathbf C})
     =\mathrm{tr}({\bA}^t \mathrm{diag}(\mathbf{d}\circ\mathrm{diag}(\mathbf{A}\dot{\mathbf C}\mathbf{A}^t))\mathbf{A}\dot\bC)\\
     =& (\mathbf{D}{\rm diag}(\mathrm{diag}(\bA\dot \bC\bA^t)), \bA\dot\bC\bA^t)
      =(\mathbf{D}\mathrm{diag}(\bA\dot \bC\bA^t), \mathrm{diag}(\bA\dot \bC\bA^t))=(\bar{\mathbf{e}},\bar{\mathbf{e}}),
  \end{aligned}
\end{equation*}
where $\mathbf{\bar e}=\mathbf{D}^\frac{1}{2}\mathrm{diag}(\bA\dot \bC\bA^t)\in\mathbb{R}^n$.
Similarly, by letting $\bar \bA = \bD^\frac{1}{2}\bA$, we have
\begin{equation*}
  \begin{aligned}
     &(\mathbf{A}^t\mathbf{D}\mathrm{diag}(\mathbf{A(A}^t\mathbf{DA}+ \alpha\bar\bC^{-1}_0)^{-1}\mathbf{A}^t\mathbf{D}\mathrm{diag}(\mathbf{A}\dot {\mathbf{C}}{\mathbf A}^t))\mathbf{A},\dot {\mathbf{C}})\\
      =&(\bD\mathrm{diag}(\bA(\bA^t\bD\bA+ \alpha\bar\bC_0^{-1})^{-1}\bA^t\bD\mathrm{diag}(\bA\dot \bC\bA^t)),\bA\dot \bC\bA^t)\\
     =&(\bar \bA(\bar \bA^t \bar \bA+ \alpha \bar\bC_0^{-1})^{-1}\bar \bA^t\bar{\mathbf{e}}, \bar{\mathbf{e}}).
  \end{aligned}
\end{equation*}
Since $\mathbf{I}_n-\bar \bA(\bar \bA^t \bar \bA+ \alpha \bar{\mathbf{C}}_0^{-1})^{-1}\bar \bA^t>0$,
the associated bilinear form is coercive on $\mathcal{S}_m^+$. Thus the Schur complement is invertible, and
the system \eqref{eqn:sensitivity} has a unique solution.
\end{proof}

\begin{corollary}
For any rank deficient $\bA$, $\dot \bC\neq\mathbf{ 0}$.
\end{corollary}
\begin{proof}
If $\dot\bC=\mathbf{0}$, the second equation in \eqref{eqn:sensitivity} reduces to
$\bA^t\bD^\frac{1}{2}\mathrm{diag}(\bA\dot{\bar {\mathbf{x}}})\bD^\frac{1}{2}\bA = - \bar\bC_0^{-1}.$
By assumption, $\bA$ is rank deficient, and thus the left hand side is
rank deficient, whereas the right hand side is of full rank, which leads to a contradiction.
Thus we have $\dot \bC \neq \mathbf{0}$.
\end{proof}

The next result gives a lower-bound for the derivative $\frac{d}{d\alpha}\psi(\bar{\mathbf{x}}_\alpha,\mathbf{C}_\alpha)$.
\begin{theorem}\label{thm:penalty-derivative}
The functional $\psi(\bar{\mathbf{x}}_\alpha,\mathbf{C}_\alpha)$ satisfies
\begin{equation*}
  \frac{d}{d\alpha}\psi(\bar{\mathbf{x}}_\alpha,\mathbf{C}_\alpha)\geq\alpha ( \bC_0^{-1} \dot{{\bar{\mathbf x}}},\dot{\bar{\mathbf x}})  + \tfrac{1}{2}(\mathbf{C}^{-1}\dot{\mathbf{C}} \mathbf{C}^{-1},\dot\bC).
\end{equation*}
\end{theorem}
\begin{proof}
By the definition of the functional $\psi$, we have
\begin{equation*}
  \frac{d}{d\alpha}\psi(\bar{\mathbf{x}}_\alpha,\mathbf{C}_\alpha)=- (\bar\bC_0^{-1}(\mathbf{\bar{x}}-\boldsymbol\mu_0),\dot{\bar{\mathbf x}})- \tfrac{1}{2}(\bar\bC_0^{-1},\dot\bC).
\end{equation*}
By taking inner product the first equation in \eqref{eqn:sensitivity} with $\dot{\bar{\mathbf{x}}}$,
and the second with $\frac{1}{2}\dot \bC$, we get
\begin{equation*}
  \begin{aligned}
     ((\mathbf{A}^t\mathbf{DA}+ \alpha \bC_0^{-1}) \dot{{\bar{\mathbf x}}},\dot{\bar{\mathbf x}}) + \tfrac{1}{2} (\mathbf{A}^t\mathbf{D}\mathrm{diag}(\mathbf{A}\dot{\mathbf{C}}\mathbf{A}^t),\dot{\bar{\mathbf x}}) &= - (\bar\bC_0^{-1}(\mathbf{\bar{x}}-\boldsymbol\mu_0),\dot{\bar{\mathbf x}}),\\
	\tfrac{1}{2}(\mathbf{C}^{-1}\dot{\mathbf{C}} \mathbf{C}^{-1}+\tfrac{1}{2}\mathbf{A}^t \mathbf{D}^\frac{1}{2}\mathrm{diag}\mathrm{diag}(\mathbf{A}\dot {\mathbf{C}}\mathbf{A}^t)\mathbf{D}^\frac{1}{2}\mathbf{A},\dot\bC) + \tfrac{1}{2}(\mathbf{A}^t\mathbf{D}^\frac{1}{2}\mathrm{diag}(\mathbf{A}\dot{\bar{ \mathbf x}})\mathbf{D}^\frac{1}{2}\mathbf{A},\dot\bC)  &= - \tfrac{1}{2}(\bar\bC_0^{-1},\dot\bC).
  \end{aligned}
\end{equation*}
By the cyclic property of trace and summing these two identities, we obtain
\begin{equation}\label{eqn:deriv-psi}
   \begin{aligned}
    - (\bar\bC_0^{-1} (\mathbf{\bar{x}}-\boldsymbol\mu_0),\dot{\bar{\mathbf x}})- \tfrac{1}{2}(\bar\bC_0^{-1},\dot\bC)
    =&  ( \alpha \bC_0^{-1}\dot{{\bar{\mathbf x}}},\dot{\bar{\mathbf x}})+\tfrac{1}{2}(\mathbf{C}^{-1}\dot{\mathbf{C}} \mathbf{C}^{-1},\dot\bC)\\
    & +(\mathbf{A}^t\mathbf{DA}\dot{{\bar{\mathbf x}}},\dot{\bar{\mathbf x}})  + \tfrac{1}{4}( \mathbf{D}\mathrm{diag}(\mathbf{A}\dot {\mathbf{C}}\mathbf{A}^t),\mathrm{diag}(\bA\dot\bC\bA^t))\\
    & + (\mathbf{D}^\frac{1}{2}\mathrm{diag}(\mathbf{A}\dot{\mathbf{C}}\mathbf{A}^t),\bD^\frac{1}{2}\mathbf{A}\dot{\bar{\mathbf x}}).
  \end{aligned}
\end{equation}
Meanwhile, by the Cauchy-Schwarz inequality, we have
\begin{equation*}
  (\mathbf{D}^\frac{1}{2}\mathrm{diag}(\mathbf{A}\dot{\mathbf{C}}\mathbf{A}^t),\bD^\frac{1}{2}\mathbf{A}\dot{\bar{\mathbf x}})\geq -(\bD^\frac{1}{2}\mathbf{A}\dot{\bar{\mathbf x}},\bD^\frac{1}{2}\mathbf{A}\dot{\bar{\mathbf x}}) - \tfrac{1}{4}(\mathbf{D}^\frac{1}{2}\mathrm{diag}(\mathbf{A}\dot{\mathbf{C}}\mathbf{A}^t),\mathbf{D}^\frac{1}{2}\mathrm{diag}(\mathbf{A}\dot{\mathbf{C}}\mathbf{A}^t)).
\end{equation*}
Substituting the preceding inequality into \eqref{eqn:deriv-psi} yields the desired estimate.
\end{proof}

\begin{corollary}\label{cor:monotone}
The functional $\psi(\bar{\mathbf{x}}_\alpha,\mathbf{C}_\alpha)$ is strictly increasing in $\alpha$.
\end{corollary}
\begin{proof}
By Theorem \ref{thm:sensitivity-sol}, \eqref{eqn:sensitivity} is uniquely solvable. Since
the right hand side of \eqref{eqn:sensitivity} is nonvanishing (by assumption, $\mathbf{C}_0$ is nonzero), the solution
pair $(\dot{\bar{\mathbf{x}}},\dot{\mathbf{C}})$ to \eqref{eqn:sensitivity} is nonzero. Thus, by Theorem
\ref{thm:penalty-derivative}, $\frac{d}{d\alpha}\psi(\bar{\mathbf{x}}_\alpha,\mathbf{C}_\alpha)$
is strictly positive, i.e., $\psi(\bar{\mathbf{x}}_\alpha, \mathbf{C}_\alpha)$ is strictly increasing.
\end{proof}

\begin{remark}
For the standard regularized least-squares problem, the solution is distinct for different
$\alpha$, and it never vanishes (except the trivial case $\mathbf{y}=0$).  The proof in
Corollary \ref{cor:monotone} indicates that an analogous statement holds for the Poisson model \eqref{eqn:poisson}.
\end{remark}

\bibliographystyle{abbrv}
\bibliography{poisson}

\begin{thebibliography}{10}

\bibitem{AndersonKleindorfer:1969}
W.~N. Anderson, Jr., G.~B. Kleindorfer, P.~R. Kleindorfer, and M.~B. Woodroofe.
\newblock Consistent estimates of the parameters of a linear system.
\newblock {\em Ann. Math. Stat.}, 40:2064--2075, 1969.

\bibitem{archambeau2007gaussian}
C.~Archambeau, D.~Cornford, M.~Opper, and J.~Shawe-Taylor.
\newblock Gaussian process approximations of stochastic differential equations.
\newblock {\em JMLR: Workshop Conf. Proc.}, 1:1--16, 2007.

\bibitem{barber1998ensemble}
D.~Barber and C.~M. Bishop.
\newblock Ensemble learning in {B}ayesian neural networks.
\newblock {\em NATO ASI Series F Comput. Syst. Sci.}, 168:215--238, 1998.

\bibitem{BardsleyLuttman:2016}
J.~M. Bardsley and A.~Luttman.
\newblock A {M}etropolis-{H}astings method for linear inverse problems with
  {P}oisson likelihood and {G}aussian prior.
\newblock {\em Int. J. Uncertain. Quantif.}, 6(1):35--55, 2016.

\bibitem{Bertsekas:2016}
D.~P. Bertsekas.
\newblock {\em Nonlinear {P}rogramming}.
\newblock Athena Scientific, Belmont, MA, second edition, 1999.

\bibitem{Bishop:2006}
C.~M. Bishop.
\newblock {\em {P}attern {R}ecognition and {M}achine {L}earning}.
\newblock Springer, Singapore, 2006.

\bibitem{CameronTrivedi:1998}
A.~Cameron and P.~K. Trivedi.
\newblock {\em {Regression Analysis of Count Data}}.
\newblock Cambridge University Press, 1998.

\bibitem{challis2013gaussian}
E.~Challis and D.~Barber.
\newblock Gaussian {Kullback-Leibler} approximate inference.
\newblock {\em J. Mach. Learn. Res.}, 14(1):2239--2286, 2013.

\bibitem{Dwyer:1967}
P.~S. Dwyer.
\newblock Some applications of matrix derivatives in multivariate analysis.
\newblock {\em J. Amer. Stat. Assoc.}, 62(318):607--625, 1967.

\bibitem{ElSayed:2001}
S.~M. El-Sayed and A.~C.~M. Ran.
\newblock On an iteration method for solving a class of nonlinear matrix
  equations.
\newblock {\em SIAM J. Matrix Anal. Appl.}, 23(3):632--645, 2001/02.

\bibitem{EnglHankeNeubauer:1996}
H.~W. Engl, M.~Hanke, and A.~Neubauer.
\newblock {\em Regularization of {I}nverse {P}roblems}.
\newblock Kluwer Academic, Dordrecht, 1996.

\bibitem{ErdoganFessler:1999}
H.~Erdo\v{g}an and J.~A. Fessler.
\newblock Monotonic algorithms for transmission tomography.
\newblock {\em IEEE Trans. Med. Imag.}, 18(9):801--814, 1999.

\bibitem{gartner2012approximation}
B.~G{\"a}rtner and J.~Matousek.
\newblock {\em Approximation {A}lgorithms and {S}emidefinite {P}rogramming}.
\newblock Springer-Verlag, Berlin, Heidelberg,, 2012.

\bibitem{GolubVanLoan:2013}
G.~H. Golub and C.~F. Van~Loan.
\newblock {\em Matrix {C}omputations}.
\newblock Johns Hopkins University Press, Baltimore, MD, fourth edition, 2013.

\bibitem{halko2011finding}
N.~Halko, P.-G. Martinsson, and J.~A. Tropp.
\newblock Finding structure with randomness: {P}robabilistic algorithms for
  constructing approximate matrix decompositions.
\newblock {\em SIAM Rev.}, 53(2):217--288, 2011.

\bibitem{hall2011theory}
P.~Hall, J.~T. Ormerod, and M.~P. Wand.
\newblock Theory of {G}aussian variational approximation for a {P}oisson mixed
  model.
\newblock {\em Stat. Sinica}, 21(1):369--389, 2011.

\bibitem{hinton1993keeping}
G.~E. Hinton and D.~Van~Camp.
\newblock Keeping the neural networks simple by minimizing the description
  length of the weights.
\newblock In {\em {COLT'93}, {Proc. 6th Annual Conf. Comput. Learning Theory}},
  pages 5--13, New York, 1993. ACM.

\bibitem{HohageWerner:2016}
T.~Hohage and F.~Werner.
\newblock Inverse problems with {P}oisson data: statistical regularization
  theory, applications and algorithms.
\newblock {\em Inverse Problems}, 32(9):093001, 56, 2016.

\bibitem{ito2014inverse}
K.~Ito and B.~Jin.
\newblock {\em Inverse {P}roblems: {T}ikhonov {T}heory and {A}lgorithms}.
\newblock World Scientific, Hackensack, NJ, 2015.

\bibitem{ItoJinTakeuchi:2011}
K.~Ito, B.~Jin, and T.~Takeuchi.
\newblock A regularization parameter for nonsmooth {T}ikhonov regularization.
\newblock {\em SIAM J. Sci. Comput.}, 33(3):1415--1438, 2011.

\bibitem{JamesStein:1961}
W.~James and C.~Stein.
\newblock Estimation with quadratic loss.
\newblock In {\em Proc. 4th {B}erkeley {S}ympos. {M}ath. {S}tatist. and
  {P}rob., {V}ol. {I}}, pages 361--379. Univ. California Press, Berkeley,
  Calif., 1961.

\bibitem{JinZou:2009}
B.~Jin and J.~Zou.
\newblock Augmented {T}ikhonov regularization.
\newblock {\em Inverse Problems}, 25(2):025001, 25, 2009.

\bibitem{KaipioSomersalo:2005}
J.~Kaipio and E.~Somersalo.
\newblock {\em Statistical and {C}omputational {I}nverse {P}roblems}.
\newblock Springer-Verlag, New York, 2005.

\bibitem{kass1995bayes}
R.~E. Kass and A.~E. Raftery.
\newblock Bayes factors.
\newblock {\em J. Amer. Stat. Assoc.}, 90(430):773--795, 1995.

\bibitem{Kelley:1995}
C.~T. Kelley.
\newblock {\em Iterative {M}ethods for {L}inear and {N}onlinear {E}quations}.
\newblock SIAM, Philadelphia, PA, 1995.

\bibitem{KhanMohamedMurphy:2012}
M.~E. Khan, S.~Mohamed, and K.~Murphy.
\newblock Fast {B}ayesian inference for nonconjugate {G}aussian process
  regression.
\newblock In {\em NIPS}, pages 3140--3148, 2012.

\bibitem{KulisSustik:2009}
B.~Kulis, M.~A. Sustik, and I.~S. Dhillon.
\newblock Low-rank kernel learning with {B}regman matrix divergences.
\newblock {\em J. Mach. Learn. Res.}, 10:341--376, 2009.

\bibitem{KullbackLeibler:1951}
S.~Kullback and R.~A. Leibler.
\newblock On information and sufficiency.
\newblock {\em Ann. Math. Stat.}, 22:79--86, 1951.

\bibitem{LuStuartWeber:2016}
Y.~Lu, A.~M. Stuart, and H.~Weber.
\newblock Gaussian approximations for probability measures on $\mathbf {R}^ d$.
\newblock SIAM/ASA J. Uncertainty Quantification, in press. arXiv:1611.08642,
  2016.

\bibitem{minka2001expectation}
T.~P. Minka.
\newblock Expectation propagation for approximate {B}ayesian inference.
\newblock In {\em Proc. 17th Conf. Uncertainty in Artificial Intelligence},
  pages 362--369. Morgan Kaufmann Publishers Inc., 2001.

\bibitem{OpperArchambeau:2009}
M.~Opper and C.~Archambeau.
\newblock The variational {G}aussian approximation revisited.
\newblock {\em Neural Comput.}, 21(3):786--792, 2009.

\bibitem{OrmerodWand:2012}
J.~T. Ormerod and M.~P. Wand.
\newblock Gaussian variational approximate inference for generalized linear
  mixed models.
\newblock {\em J. Comput. Graph. Statist.}, 21(1):2--17, 2012.

\bibitem{pillow2007likelihood}
J.~Pillow.
\newblock Likelihood-based approaches to modeling the neural code.
\newblock In K.~Doya, S.~Ishii, A.~Pouget, and R.~Rao, editors, {\em Bayesian
  {B}rain: {P}robabilistic {A}pproaches to {N}eural {C}oding}, pages 53--70.
  MIT Press, Cambridge, 2007.

\bibitem{PinskiSimpson:2015}
F.~J. Pinski, G.~Simpson, A.~M. Stuart, and H.~Weber.
\newblock {Kullback-Leibler} approximation for probability measures on
  infinite-dimensional spaces.
\newblock {\em SIAM J. Math. Anal.}, 47(6):4091--4122, 2015.

\bibitem{RavikumarWainwrightPaskuttiYu:2011}
P.~Ravikumar, M.~J. Wainwright, G.~Raskutti, and B.~Yu.
\newblock High-dimensional covariance estimation by minimizing
  {$\ell_1$}-penalized log-determinant divergence.
\newblock {\em Electron. J. Stat.}, 5:935--980, 2011.

\bibitem{RohdeWand:2016}
D.~Rohde and M.~P. Wand.
\newblock Semiparametric mean field variational {B}ayes: general principles and
  numerical issues.
\newblock {\em J. Mach. Learn. Res.}, 17:Paper No. 172, 47, 2016.

\bibitem{SchusterKaltenbacher:2012}
T.~Schuster, B.~Kaltenbacher, B.~Hofmann, and K.~S. Kazimierski.
\newblock {\em Regularization {M}ethods in {B}anach {S}paces}.
\newblock Walter de Gruyter GmbH \& Co. KG, Berlin, 2012.

\bibitem{Stuart:2010}
A.~M. Stuart.
\newblock Inverse problems: a {B}ayesian perspective.
\newblock {\em Acta Numer.}, 19:451--559, 2010.

\bibitem{ThierneyKadane:1986}
L.~Tierney and J.~B. Kadane.
\newblock Accurate approximations for posterior moments and marginal densities.
\newblock {\em J. Amer. Stat. Assoc.}, 81(393):82--86, 1986.

\bibitem{WainwrightJordan:2008}
M.~J. Wainwright and M.~I. Jordan.
\newblock Graphical models, exponential families, and variational inference.
\newblock {\em Found. Trends\copyright Mach. Learn.}, 1(1--2):1--305, 2008.

\bibitem{YavuzFessler:1997}
M.~Yavuz and J.~A. Fessler.
\newblock New statistical models for randoms-precorrected {PET} scans.
\newblock In {\em {Information Processing in Medical Imaging (Lecture Notes in
  Computer Science)}}, volume 1230, pages 190--203. Springer-Verlag, 1997.

\bibitem{ZhangLeithead:2007}
Y.~Zhang and W.~E. Leithead.
\newblock Approximate implementation of the logarithm of the matrix determinant
  in {G}aussian process regression.
\newblock {\em J. Stat. Comput. Simul.}, 77(4):329--348, 2007.

\end{thebibliography}

\end{document}